\documentclass[11pt,a4paper,reqno]{amsart}
\usepackage{amsmath,amssymb,amsthm,mathtools}
\usepackage{booktabs,tabularx,array,graphicx,needspace}
\usepackage{placeins}
\usepackage[margin=1in]{geometry}
\usepackage{pgf}
\usepackage[colorlinks=true,linkcolor=blue,citecolor=blue,urlcolor=blue]{hyperref}
\numberwithin{equation}{section}
\newtheorem{theorem}{Theorem}[section]
\newtheorem{lemma}[theorem]{Lemma}
\newtheorem{proposition}[theorem]{Proposition}
\newtheorem{corollary}[theorem]{Corollary}

\newtheorem{thm}[theorem]{Theorem}

\newtheorem{coro}[theorem]{Corollary}
\newtheorem*{informaltheorem}{Theorem}
\theoremstyle{remark}
\newtheorem{remark}[theorem]{Remark}
\newtheorem*{remark*}{Remark}

\newcommand{\T}{\mathbb T}
\newcommand{\PP}{\mathbb P}

\DeclareMathOperator{\divergence}{div}
\title[Viscosity-uniform error estimates in Besov spaces]{Viscosity-uniform error estimates for a semi-implicit Fourier method in Besov spaces}
\author{Xinyu Cheng}
\address{Research Institute of Intelligent Complex Systems, Fudan University, Shanghai, P.R. China}
\email{xycheng@fudan.edu.cn}
\author{Zhaonan Luo}
\address{School of Science, Shenzhen Campus of Sun Yat-sen University, Shenzhen, P.R. China}
\email{luozhn7@mail.sysu.edu.cn}
\author{Sheng Wang}
\address{School of Mathematical Sciences, Peking University, Beijing, P.R. China}
\email{everest\_sheng@pku.edu.cn}
\begin{document}
\begin{abstract}
We prove error estimates for a semi-implicit Fourier--Galerkin method
for two-dimensional incompressible flow on the torus, with viscosity
$0\le\nu\le1$. The temporal error is first order in $H^s$ for
$H^{s+2}$ initial data and integer $s\ge1$.
For data in $B^m_{2,1}$, with integer $m\ge3$, the fully discrete error
is of order $\tau+N^{-(m-2)}$ in $B^1_{2,1}$ and $B^0_{\infty,1}$.
These estimates hold on a common local interval, with constants and
time-step thresholds independent of viscosity and Fourier resolution.
For data in $B^{2+\varepsilon}_{2,1}$, $0<\varepsilon<1$, the error
relative to Euler flow is bounded by
$C(\tau^{(1+\varepsilon)/2}+N^{-\varepsilon}
+\nu^{(1+\varepsilon)/2})$.
The consistency bounds use viscosity-weighted regularity. Orthogonality
of the Leray and Fourier projectors in the block energies gives the
endpoint estimates, and stability with respect to the initial data
yields the fractional rate. An $L^2$ residual tolerance $\delta$
contributes $CN\delta$ to the endpoint error of computed iterates.
We report refinement studies for smooth and fixed multiscale data,
together with qualitative flow comparisons. An exactly solvable rough
shear has fractional errors, up to logarithms, in $B^1_{2,1}$ and
first-order errors in $B^0_{\infty,1}$.
\end{abstract}
\maketitle

\section{Introduction}\label{sec:oct-introduction}

Let $U^\nu$ be the velocity of an incompressible fluid on
$\T^2=(\mathbb R/(2\pi\mathbb Z))^2$:
\begin{equation}\label{eq:oct-intro-model}
 \partial_tU^\nu+U^\nu\cdot\nabla U^\nu+\nabla p^\nu
 =\nu\Delta U^\nu,\qquad
 \nabla\cdot U^\nu=0,\qquad U^\nu(0)=u_0.
\end{equation}
For $0\le\nu\le1$, we approximate the velocity by a semi-implicit
Fourier method and write $V=U^0$ for the Euler solution. We seek
error bounds in norms that control the uniform velocity error,
with constants bounded as $\nu\downarrow0$. The continuous and
discrete solutions are compared on a common local interval
determined by the initial data.

Energy estimates are central to Leray's construction of weak
solutions for viscous flow~\cite{L34}. For two-dimensional Euler
flow, Yudovich~\cite{Y63} proved global well-posedness for bounded
vorticity. Marion and Temam~\cite{MT98} survey the continuous
theory, Galerkin approximation, and time discretization.
The error estimates below require bounds on the discrete velocity
in the Sobolev or Besov class of the continuous solution.

The time discretization is
\begin{equation}\label{eq:oct-intro-scheme}
 \frac{u^{n+1}-u^n}{\tau}
 +\PP(u^n\cdot\nabla u^{n+1})=\nu\Delta u^{n+1},
 \qquad u^0=u_0,
\end{equation}
where $\PP$ is the Leray projector. The Fourier--Galerkin method
applies the orthogonal cutoff $\Pi_N$, retaining $|k|_\infty\le N$,
to the nonlinear term and starts from $\Pi_Nu_0$.
Each update is linear in the new velocity. Since $u^n$ is divergence
free, periodicity gives
$(u^n\cdot\nabla u^{n+1},u^{n+1})=0$.
The resulting kinetic-energy identity holds for every time step.
At higher regularity, the same cancellation leaves commutator terms
in the differentiated and localized equations. We bound these terms
uniformly in viscosity on a common local interval. The step
restriction for the inner iteration is treated separately.

Chorin introduced the projection method and established convergence
estimates for discrete Navier--Stokes
approximations~\cite{Chorin68,Chorin69}. In the periodic setting, Guo and
Zou~\cite{GZ03} analyzed nonlinear convergence of a Fourier spectral
projection method. He~\cite{He08} studied
implicit--explicit Euler time discretization for smooth and
nonsmooth data. For finite elements, Heywood
and Rannacher related spatial error estimates to the regularity of
the solution near the initial time~\cite{HR82} and subsequently
proved second-order bounds for Crank--Nicolson time
discretization~\cite{HR90}. For the method studied here, the
regularity bounds used in both stability and consistency must
remain uniform as viscosity vanishes.

In low-regularity integrators, part of the evolution is incorporated
into the numerical step to reduce the differentiability required
by the consistency estimate; see Rousset and Schratz~\cite{RS21}. Energy stability and convergence have been proved
for an iterative low-regularity Cahn--Hilliard method~\cite{CLWCH26},
and optimal error estimates hold for a surface quasi-geostrophic
(SQG) integrator at endpoint regularity~\cite{CLWSQG}.
For Navier--Stokes, Li, Ma, and Schratz~\cite{LMS22}
apply the viscous semigroup to the transporting velocity in a
semi-implicit exponential integrator. Their first-order $L^2$
error bound is uniform in viscosity under uniform solution bounds
in $L^\infty(0,T;W^{2,p})$, $p>2$. Under a uniform $H^4$ solution
bound, Ma's second-order exponential method~\cite{Ma26} gives an
$L^2$ error estimate with a constant that remains bounded as
viscosity tends to zero. At fixed positive viscosity,
parabolic smoothing permits substantially rougher initial data:
Li, Ma, and Ueda~\cite{LMU22} use graded time steps for a
semi-implicit finite element method with divergence-free $L^2$
data, and Li, Rao, Zhang, and Zhou~\cite{LRZZ25} obtain second-order
temporal and nearly second-order spatial bounds for an
implicit--explicit Runge--Kutta finite element method with critical
$L^2$ data. These latter estimates allow singular behaviour near
the initial time. We study the semi-implicit update
\eqref{eq:oct-intro-scheme} in Besov norms that control the uniform
velocity error. Its first-order bounds for $B^3_{2,1}$ data and
fractional bounds for $B^{2+\varepsilon}_{2,1}$ data, $0<\varepsilon<1$,
hold through $t=0$ on a common local interval, including at zero viscosity.

For the continuous inviscid limit, Masmoudi proved Sobolev convergence
and quantitative bounds for vortex-patch data~\cite{M07}. Hmidi and Keraani treated viscous flows in
borderline Besov spaces~\cite{HK08}, while Guo, Li, and Yin established
local well-posedness in $B^1_{\infty,1}$ and an inviscid
limit~\cite{GLY19}. For the SQG equations, quantitative inviscid
convergence in a lower-order Besov norm and convergence in the
topology of the initial data were established in~\cite{CLWVV}.
Uniform-in-time estimates for the
vanishing-damping limit of the inviscid Oldroyd--B model were
proved in~\cite{CLYY24}.
To estimate the discrete consistency defect, we use the paraproduct
and transport commutator estimates in~\cite{BCD11}, together with
the fractional product and commutator inequalities in~\cite{Li19}.

This manuscript combines and revises the Sobolev and Besov analyses
of \eqref{eq:oct-intro-scheme} presented in the original versions
of our earlier preprints~\cite{CLW24,CLWB25}.
The endpoint error is measured in $B^0_{\infty,1}$, the sum of
the supremum norms of the dyadic blocks. It controls $L^\infty$;
in two dimensions,
\begin{equation}\label{eq:m2an-intro-embedding}
 B^1_{2,1}\hookrightarrow B^0_{\infty,1}\hookrightarrow L^\infty.
\end{equation}
The main Besov error bounds can be summarized as follows.
\begin{informaltheorem}[Main results, informal]
Let $u_0$ be divergence free on $\T^2$, and let $U^\nu$ and
$V=U^0$ be the Navier--Stokes and Euler velocities with initial
datum $u_0$, respectively.
Let $u_N^n$ be the exact Fourier--Galerkin iterates at viscosity $\nu$,
with $u_N^0=\Pi_Nu_0$ and $t_n=n\tau$.
Under either regularity assumption below, there are $T>0$, $\tau_0>0$,
and $C>0$, depending only on the indicated initial norm and regularity
index, such that the following bounds hold for $0\le\nu\le1$, $N\ge1$,
$0<\tau\le\tau_0$, and $M\tau\le T$.

If $u_0\in B^m_{2,1}$, with integer $m\ge3$, then
\begin{equation}\label{eq:m2an-intro-endpoint-rate}
 \max_{0\le n\le M}\|u_N^n-V(t_n)\|_{B^0_{\infty,1}}
 \le C\bigl(\tau+N^{-(m-2)}+\nu\bigr).
\end{equation}
If $u_0\in B^{2+\varepsilon}_{2,1}$, with $0<\varepsilon<1$, then
\begin{equation}\label{eq:m2an-intro-fractional-rate}
 \max_{0\le n\le M}\|u_N^n-V(t_n)\|_{B^0_{\infty,1}}
 \le C\bigl(\tau^{(1+\varepsilon)/2}+N^{-\varepsilon}
                  +\nu^{(1+\varepsilon)/2}\bigr).
\end{equation}
Both bounds also hold in $B^1_{2,1}$. When the comparison is with
$U^\nu(t_n)$, the respective viscosity term is omitted. The constants,
common interval, and time-step threshold are independent of $\nu$ and $N$.
\end{informaltheorem}
By \eqref{eq:m2an-intro-embedding}, these estimates control the uniform
velocity error, including at zero viscosity. The precise statements are
given in Theorem~\ref{thm:oct-bes-sharp-Fourier-Besov1} and
Corollaries~\ref{cor:oct-bes-inviscid}, \ref{oct5:cor-fractional-sharp},
and~\ref{oct5:cor-fractional-inviscid}. Subsection~\ref{sec:oct-higher}
extends the first-order estimates to $B^r_{\infty,1}$ for integer $r\ge0$.

We test each dyadic block in $L^2$ and sum in $\ell^1$.
In these localized energy identities, the Leray projector drops
out and diffusion contributes a nonnegative term. The sharp
Fourier cutoff is also an $L^2$ contraction on every block,
including those whose support crosses the cutoff boundary.
Its truncation error can thus be estimated in the stability norm
without a uniform $L^\infty$ multiplier bound for $\Pi_N$.
We assume initial data in $B^m_{2,1}$, which contains
$B^m_{\infty,1}$ on the torus. The $B^0_{\infty,2}$ error follows
from an $H^1$ estimate; that Besov norm does not in general
control $L^\infty$.

For integers $s\ge1$ and $m\ge s+2$, $u_0\in H^m$ gives
\begin{equation}\label{eq:m2an-intro-sobolev-rate}
 \max_{t_n\le T}\|u_N^n-U^\nu(t_n)\|_{H^s}
 \le C\bigl(\tau+N^{-(m-s-1)}\bigr);
\end{equation}
see Theorem~\ref{thm:oct-refined-sobolev}.
For integers $s\ge1$, the first-order $H^s$ time estimates are proved
for $H^{s+2}$ data. The separate $L^2$ estimate assumes $H^3$ initial data.
We control the viscous consistency defect by error dissipation in the
Sobolev estimate and by viscosity-weighted smoothing in the Besov estimate,
avoiding an unweighted bound on the second time derivative.
For fractional regularity, we smooth the initial data at frequency $R\simeq\tau^{-1/2}$
and apply a first-order estimate whose dependence on the high
norm is linear. The unsmoothed low norm controls both the common
lifespan and the stability constants.
For computed Galerkin iterates, the effect of an inexact linear
solve is measured in $B^0_{\infty,1}$: under the assumptions of
Corollary~\ref{cor:m2an-implemented-endpoint}, an $L^2$ residual
bounded by $\delta$ in the equation divided by $\tau$ adds
$CN\delta$ to the error in this endpoint norm.

The numerical experiments separate temporal, spatial, and
vanishing-viscosity errors. For low-regularity tests, the initial
datum is fixed and its projection changes with the resolution.
For the multiscale manufactured solution, the temporal error is
computed against the exact solution; for the unforced smooth flow,
we use an independent reference calculation. An exactly solvable transported shear measures
the spatial tail. For the rough lacunary shear, both norms in
\eqref{eq:m2an-intro-embedding} can be evaluated directly:
the temporal and viscosity errors have fractional rates, up to
logarithms, in $B^1_{2,1}$, while both errors are first order in
$B^0_{\infty,1}$. This example leaves open the optimal uniform
rate in the endpoint norm. Qualitative comparisons include
double-shear roll-up at several viscosities.

The Sobolev estimates are proved in Section~\ref{sec:oct-sobolev}.
Section~\ref{sec:oct-bes-endpoint-revision} gives the endpoint Besov
estimates and their consequences in higher norms, followed by the
fractional estimates in Section~\ref{oct5:sec-fractional-smoothing}.
Section~\ref{rev:sec-picard}
treats the iterative solver and its residuals.
Section~\ref{sec:oct-numerics} presents the numerical experiments
for rough data, smooth solutions, and flow dynamics, followed by
a discussion of the results. Appendices~\ref{rev:sec-boundary}
and~\ref{rev:sec-fem} give conditional boundary and finite element estimates.
Appendix~\ref{rev:sec-H1-dissipation} gives an $H^1$ dissipation criterion.

\section{Preliminaries}
\label{sec:m2an-notation}
We work on $\T^2=(\mathbb R/(2\pi\mathbb Z))^2$ with normalized
measure $(2\pi)^{-2}dx$. Our Fourier and Sobolev conventions are
\[
 z(x)=\sum_{k\in\mathbb Z^2}\widehat z(k)e^{ik\cdot x},
 \qquad \|z\|_{H^s}^2=\sum_k(1+|k|^2)^s|\widehat z(k)|^2,
 \qquad \Lambda=(I-\Delta)^{1/2}.
\]
The Leray projector $\PP$ has symbol $I-k\otimes k/|k|^2$ for
$k\ne0$ and $I$ at $k=0$. The orthogonal projector $\Pi_N$ retains the
modes $|k|_\infty\le N$. The real divergence-free Fourier space is
\[
 X_N:=\bigl\{v\in L^2(\T^2;\mathbb R^2):
 \Pi_Nv=v,\quad\nabla\cdot v=0\bigr\}.
\]
These projectors commute with derivatives and with the dyadic
multipliers introduced below.

Choose a smooth radial cutoff $\chi$ with $0\le\chi\le1$, equal to
one on $r\le1$ and zero on $r\ge2$, and define
\[
 \widehat{\Delta_0z}(k)=\chi(|k|)\widehat z(k),\qquad
 \widehat{\Delta_jz}(k)=\left[\chi(2^{-j}|k|)-\chi(2^{-j+1}|k|)\right]\widehat z(k),\quad j\ge1.
\]
These symbols form a partition of unity. The inhomogeneous Besov norm is
\[
 \|z\|_{B^s_{p,q}}=\left\|(2^{js}\|\Delta_jz\|_p)_{j\ge0}\right\|_{\ell^q}.
\]
Here $1\le p,q\le\infty$, and vector norms use the Euclidean magnitude.
Other admissible cutoffs give equivalent norms. The numerical cutoff
is specified in Section~\ref{sec:expanded-discretization}.

For nonnegative quantities, $a\asymp b$ means $cb\le a\le Cb$ with
constants $c,C>0$ independent of the variable or parameter in the
limit under consideration.

The two-dimensional Bernstein inequality and finite volume imply
\begin{align}
 \|z\|_{B^0_{\infty,1}}&\le C\|z\|_{B^1_{2,1}},&
 \|z\|_{B^0_{\infty,2}}&\le C\|z\|_{H^1},\label{eq:oct-embeddings}\\
 \|z\|_{B^s_{2,1}}&\le C\|z\|_{B^s_{\infty,1}},&
 \|z\|_{H^s}&\le C\|z\|_{B^s_{\infty,2}}.\label{eq:oct-finite-volume}
\end{align}
We use the inhomogeneous paraproduct and commutator estimates of
\cite[Chapters 2 and 3]{BCD11}. In the $\ell^1$ Besov argument, the
algebra $B^1_{2,1}\hookrightarrow L^\infty$ controls multiplication.
The $\ell^2$ endpoint estimate follows from the $H^1$ bound.

Write
\[
 H^r_\sigma(\mathbb T^2)
 :=\{z\in H^r(\mathbb T^2;\mathbb R^2):\nabla\cdot z=0\}.
\]
All initial velocities are assumed to be divergence free.
For $0\le\nu\le1$, let $U^\nu$ solve
\begin{equation}
 \partial_tU^\nu+\mathbb P(U^\nu\cdot\nabla U^\nu)
 =\nu\Delta U^\nu,
 \qquad U^\nu(0)=u_0,
 \label{eq:oct-NS}
\end{equation}
and let $V$ solve the corresponding Euler equation
\begin{equation}
 \partial_tV+\mathbb P(V\cdot\nabla V)=0,
 \qquad V(0)=u_0.
 \label{eq:oct-Euler}
\end{equation}
The fully discrete Fourier--Galerkin approximation is
\begin{equation}
 \frac{u_N^{n+1}-u_N^n}{\tau}
 +\Pi_N\mathbb P(u_N^n\cdot\nabla u_N^{n+1})
 =\nu\Delta u_N^{n+1},
 \qquad u_N^0=\Pi_Nu_0.
 \label{eq:oct-scheme}
\end{equation}
Here $t_n=n\tau$, $M\tau=T$, and $0\le n\le M-1$.

The pure time-discrete scheme is
\begin{equation}
 \frac{u^{n+1}-u^n}{\tau}
 +\mathbb P(u^n\cdot\nabla u^{n+1})
 =\nu\Delta u^{n+1},
 \qquad u^0=u_0.
 \label{eq:oct-time-scheme}
\end{equation}

The continuous equations and both schemes preserve the spatial mean
of $u_0$: the transport and diffusion terms integrate to zero by
periodicity and incompressibility, and $\Pi_N$ retains the zero mode.

\begin{proposition}[Unconditional kinetic-energy identity]
\label{prop:oct-energy}
For any $N$, $\tau>0$, and $\nu\ge0$, the Fourier scheme \eqref{eq:oct-scheme} satisfies, for $0\le n\le M-1$,
\begin{equation}
 \|u_N^{n+1}\|_2^2+\|u_N^{n+1}-u_N^n\|_2^2
 +2\nu\tau\|\nabla u_N^{n+1}\|_2^2=\|u_N^n\|_2^2.
 \label{eq:oct-energy-identity}
\end{equation}
\end{proposition}
\begin{proof}
Pair the equation with $2\tau u_N^{n+1}$. The projector drops from the inner product because the test field belongs to $X_N$. Periodicity and $\divergence u_N^n=0$ give
\[
 (u_N^n\cdot\nabla u_N^{n+1},u_N^{n+1})
 =\tfrac12\int_{\T^2}u_N^n\cdot\nabla|u_N^{n+1}|^2=0.
\]
The polarization identity and integration by parts give \eqref{eq:oct-energy-identity}. The same argument applies to the pure time-discrete equation.
\end{proof}

Identity~\eqref{eq:oct-energy-identity} holds for every time step.
The regularity and error bounds below are local in time, with a step
threshold determined by solution norms.

\section{Sobolev estimates}
\label{sec:oct-sobolev}

All Sobolev indices in this section are nonnegative integers.

\subsection{Discrete regularity}

We use the transport estimate
\begin{equation}
 \left|
 \left\langle
 \Lambda^m(a\cdot\nabla z),\Lambda^mz
 \right\rangle
 \right|
 \le C_m\|a\|_{H^m}\|z\|_{H^m}^2,
 \qquad m>2,
 \label{eq:oct-high-transport}
\end{equation}
when both $a$ and $z$ belong to $H^m_\sigma$.

For every $0\le r\le m$, Fourier truncation gives
\begin{equation}
 \|(I-\Pi_N)f\|_{H^r}
 \le C N^{-(m-r)}\|f\|_{H^m}.
 \label{eq:oct-projection}
\end{equation}
The constants are independent of $N$ and $\nu$.

For fixed $u_N^n$, equation~\eqref{eq:oct-scheme} is a linear problem in
$X_N$. Its homogeneous equation gives
\[
 \|z\|_2^2+\nu\tau\|\nabla z\|_2^2=0,
\]
by skew-symmetry of transport. The operator is therefore invertible.

Taking the $H^m$ inner product of \eqref{eq:oct-scheme} with
$u_N^{n+1}$ and using \eqref{eq:oct-high-transport}, we obtain
\begin{align}
 &\frac1{2\tau}
 \left(
 \|u_N^{n+1}\|_{H^m}^2-
 \|u_N^n\|_{H^m}^2+
 \|u_N^{n+1}-u_N^n\|_{H^m}^2
 \right)
 +\nu\|\nabla u_N^{n+1}\|_{H^m}^2
 \nonumber\\
 &\qquad\le
 C_m\|u_N^n\|_{H^m}\|u_N^{n+1}\|_{H^m}^2.
 \label{eq:oct-stability-step}
\end{align}
Let $X_n=\|u_N^n\|_{H^m}$.  If
$X_n\le2\|u_0\|_{H^m}$ and
$4C_m\tau\|u_0\|_{H^m}\le1/2$, then
\eqref{eq:oct-stability-step} implies
\[
 X_{n+1}^2
 \le
 \frac{1}{1-4C_m\tau\|u_0\|_{H^m}}X_n^2.
\]
Put $A=\|u_0\|_{H^m}$; the case $A=0$ is immediate.  Under the induction
hypothesis $X_j\le2A$ for $0\le j\le n$, iteration of the preceding
one-step estimate gives
\begin{equation}
 X_{n+1}^2
 \le (1-4C_m\tau A)^{-(n+1)}A^2.
 \label{eq:oct-stability-iteration}
\end{equation}
For $0\le x\le1/2$, $-\log(1-x)\le2x$.  Consequently,
\[
 (1-4C_m\tau A)^{-(n+1)}
 \le \exp\bigl(8C_m(n+1)\tau A\bigr)
 \le \exp(8C_mTA).
\]
Choose
\[
 4C_m\tau A\le\frac12,
 \qquad
 8C_mTA\le\log4.
\]
Then \eqref{eq:oct-stability-iteration} gives
$X_{n+1}^2\le4A^2$, so the induction closes and
\[
 \max_{0\le n\le M}X_n\le2\|u_0\|_{H^m}.
\]
Both restrictions depend only on $\|u_0\|_{H^m}$.

\subsection{Error estimates}
\label{sec:oct-refined}

For integers $s\ge1$, the first-order $H^s$ estimate assumes
$m\ge s+2$. We split the consistency defect into convective and viscous
parts and absorb the
latter in the error dissipation. This avoids a uniform bound on
$\partial_{tt}U^\nu$.
Throughout the paper, viscosity-weighted higher-norm terms are omitted
when $\nu=0$.

\begin{thm}[Sobolev error and inviscid comparison]
\label{thm:oct-refined-sobolev}
Let \(s\ge1\) and \(m\ge s+2\) be integers, and let
\(u_0\in H^m_\sigma(\mathbb T^2)\).  Suppose that \([0,T]\) is a common
interval of existence on which
\begin{equation}
 \sup_{0\le\nu\le1}\sup_{0\le t\le T}
 \|U^\nu(t)\|_{H^m}
 +\sup_{0\le t\le T}\|V(t)\|_{H^m}
 \le K_m .
 \label{eq:oct-refined-common}
\end{equation}
After decreasing \(T\), if necessary, by an amount depending only on
\(\|u_0\|_{H^m}\), there is
\(\tau_0=\tau_0(T,K_m)>0\), independent of \(N\) and \(\nu\), such that
for \(0<\tau\le\tau_0\), \(N\ge2\), and \(0\le\nu\le1\),
\begin{equation}
 \max_{0\le n\le M}
 \|u_N^n-U^\nu(t_n)\|_{H^s}
 \le C\left(\tau+N^{-(m-s-1)}\right).
 \label{eq:oct-refined-NS-error}
\end{equation}
The corresponding time-discrete solution satisfies
\begin{equation}
 \max_{0\le n\le M}
 \|u^n-U^\nu(t_n)\|_{H^s}
 \le C\tau.
 \label{eq:oct-refined-time-error}
\end{equation}
In addition,
\begin{equation}
 \sup_{0\le t\le T}
 \|U^\nu(t)-V(t)\|_{H^s}
 \le C\nu,
 \label{eq:oct-refined-inviscid}
\end{equation}
and hence
\begin{equation}
 \max_{0\le n\le M}
 \|u_N^n-V(t_n)\|_{H^s}
 \le C\left(\tau+N^{-(m-s-1)}+\nu\right).
 \label{eq:oct-refined-Euler-error}
\end{equation}
All constants are independent of \(\nu\), \(\tau\), and \(N\).
In particular, \(s=1,m=3\) gives the uniform estimate
\[
 \max_n\|u_N^n-U^\nu(t_n)\|_{H^1}
 \le C(\tau+N^{-1}).
\]
\end{thm}

\begin{proof}
\emph{Continuous estimates.}
The \(H^m\) energy identity and the divergence-free cancellation
give, for \(m>2\),
\begin{equation}
 \frac12\frac{d}{dt}\|U^\nu\|_{H^m}^2
 +\nu\|\nabla U^\nu\|_{H^m}^2
 \le C_m\|U^\nu\|_{H^m}^3.
 \label{eq:oct-refined-continuous-energy}
\end{equation}
Integrating this inequality and using
\eqref{eq:oct-refined-common}, as well as the equivalence between
\(\|\nabla z\|_{H^m}+\|z\|_2\) and \(\|z\|_{H^{m+1}}\) on the torus,
yields
\begin{equation}
 \sup_{0\le\nu\le1}
 \nu\int_0^T\|U^\nu(t)\|_{H^{m+1}}^2\,dt
 \le C(T,K_m,\|u_0\|_{H^m}).
 \label{eq:oct-refined-weighted-energy}
\end{equation}
The $L^2$ term contributes at most $\nu T K_m^2$.

For \(s\ge1\), multiplication by an \(H^{s+1}\) function is bounded on
\(H^s\), and the following non-endpoint estimates hold:
\begin{align}
 \left|
 \left\langle
 \Lambda^s(a\cdot\nabla z),\Lambda^sz
 \right\rangle
 \right|
 &\le C\|a\|_{H^{s+2}}\|z\|_{H^s}^2,
 \label{eq:oct-refined-transport}\\
 \|b\cdot\nabla c\|_{H^s}
 &\le C\|b\|_{H^s}\|c\|_{H^{s+2}},
 \label{eq:oct-refined-tame1}\\
 \|a\cdot\nabla h\|_{H^s}
 &\le C\|a\|_{H^{s+2}}\|h\|_{H^{s+1}}.
 \label{eq:oct-refined-tame2}
\end{align}
No divergence-free assumptions are needed. In \eqref{eq:oct-refined-transport},
the principal transport term is bounded by
$\tfrac12\|\nabla\cdot a\|_\infty\|z\|_{H^s}^2$ and vanishes if
$\nabla\cdot a=0$. The other terms and \eqref{eq:oct-refined-tame2}
follow from the Leibniz rule and $H^{s+2}\hookrightarrow W^{s,\infty}$.
At $s=1$, the only additional
product in \eqref{eq:oct-refined-tame1} is bounded by
\[
 \|b\,D^2c\|_{L^2}
 \le \|b\|_{L^4}\|D^2c\|_{L^4}
 \le C\|b\|_{H^1}\|c\|_{H^3},
\]
using the two-dimensional embedding \(H^1\hookrightarrow L^4\).
The other terms use \(\|\nabla c\|_\infty\le C\|c\|_{H^3}\).
For \(s\ge2\), the analogous undifferentiated factor is controlled by
\(H^s\hookrightarrow L^\infty\), and all remaining Leibniz terms use
\(H^{s+2}\hookrightarrow W^{s,\infty}\).

Equation \eqref{eq:oct-NS} and the algebra property
of \(H^{s+1}(\mathbb T^2)\), namely
\[
 \|U^\nu\cdot\nabla U^\nu\|_{H^{s+1}}
 \le C\|U^\nu\|_{H^{s+1}}\|\nabla U^\nu\|_{H^{s+1}}
 \le C\|U^\nu\|_{H^{s+2}}^2,
\]
imply
\begin{align}
 \int_0^T\|\partial_tU^\nu(t)\|_{H^{s+1}}^2\,dt
 &\le
 C\int_0^T\|U^\nu(t)\|_{H^{s+2}}^4\,dt
 +C\nu^2\int_0^T\|U^\nu(t)\|_{H^{s+3}}^2\,dt
 \nonumber\\
 &\le C(T,K_m)
 +C\nu\left(
 \nu\int_0^T\|U^\nu(t)\|_{H^{m+1}}^2\,dt
 \right)
 \le C.
 \label{eq:oct-refined-ut}
\end{align}
We used $m\ge s+2$ and the weighted estimate
\eqref{eq:oct-refined-weighted-energy} for the highest derivative.

\emph{Consistency.}
Define
\begin{align}
 q^{n+1}
 &:=
 \frac1\tau\int_{t_n}^{t_{n+1}}
 \bigl(U^\nu(t)-U^\nu(t_{n+1})\bigr)\,dt,
 \label{eq:oct-refined-q}\\
 \eta_c^{n+1}
 &:=
 \mathbb P\left[
 U^\nu(t_n)\cdot\nabla U^\nu(t_{n+1})
 -\frac1\tau\int_{t_n}^{t_{n+1}}U^\nu(t)\cdot\nabla U^\nu(t)\,dt
 \right].
 \label{eq:oct-refined-etac}
\end{align}
Averaging the continuous equation over $[t_n,t_{n+1}]$ gives
\begin{equation}
 \frac{U^\nu(t_{n+1})-U^\nu(t_n)}{\tau}
 +\mathbb P\bigl(U^\nu(t_n)\cdot\nabla U^\nu(t_{n+1})\bigr)
 -\nu\Delta U^\nu(t_{n+1})
 =\eta_c^{n+1}+\nu\Delta q^{n+1}.
 \label{eq:oct-refined-defect-split}
\end{equation}

By the fundamental theorem of calculus and Cauchy--Schwarz,
\begin{align}
 \|q^{n+1}\|_{H^{s+1}}^2
 &\le C\tau\int_{t_n}^{t_{n+1}}
 \|\partial_tU^\nu(t)\|_{H^{s+1}}^2\,dt.
 \label{eq:oct-refined-q-bound}
\end{align}
For the convective defect, use
\begin{align*}
 &U^\nu(t_n)\cdot\nabla U^\nu(t_{n+1})
 -U^\nu(t)\cdot\nabla U^\nu(t)\\
 &\qquad=
 \bigl(U^\nu(t_n)-U^\nu(t)\bigr)\cdot\nabla U^\nu(t_{n+1})
 +U^\nu(t)\cdot\nabla\bigl(U^\nu(t_{n+1})-U^\nu(t)\bigr).
\end{align*}
Estimates \eqref{eq:oct-refined-tame1}--\eqref{eq:oct-refined-tame2},
\eqref{eq:oct-refined-common}, and the fundamental theorem of calculus give
\begin{equation}
 \|\eta_c^{n+1}\|_{H^s}^2
 \le C\tau\int_{t_n}^{t_{n+1}}
 \|\partial_tU^\nu(t)\|_{H^{s+1}}^2\,dt.
 \label{eq:oct-refined-etac-bound}
\end{equation}
Consequently, \eqref{eq:oct-refined-ut} gives the viscosity-independent bound
\begin{equation}
 \tau\sum_{n=0}^{M-1}
 \left(
 \|\eta_c^{n+1}\|_{H^s}^2
 {}+\nu\|q^{n+1}\|_{H^{s+1}}^2
 \right)
 \le C\tau^2.
 \label{eq:oct-refined-summed-defect}
\end{equation}
\emph{Galerkin error.}
The discrete regularity estimate \eqref{eq:oct-stability-iteration}
gives
\begin{equation}
 \max_{0\le n\le M}\|u_N^n\|_{H^m}\le C\|u_0\|_{H^m},
 \label{eq:oct-refined-discrete-stability}
\end{equation}
with constants independent of \(N\) and \(\nu\).

Define
\[
 \theta^n=u_N^n-\Pi_NU^\nu(t_n),\qquad
 \rho^n=(\Pi_N-I)U^\nu(t_n).
\]
Projecting \eqref{eq:oct-refined-defect-split} and subtracting it from the
scheme gives
\begin{align}
 \frac{\theta^{n+1}-\theta^n}{\tau}
 -\nu\Delta\theta^{n+1}
 +\Pi_N\mathbb P(u_N^n\cdot\nabla\theta^{n+1})
 ={}&
 -\Pi_N\eta_c^{n+1}
 -\nu\Delta\Pi_Nq^{n+1}
 \nonumber\\
 &-\Pi_N\mathbb P\bigl(
 (\theta^n+\rho^n)\cdot\nabla U^\nu(t_{n+1})
 \bigr)
 \nonumber\\
 &-\Pi_N\mathbb P\bigl(
 u_N^n\cdot\nabla\rho^{n+1}
 \bigr).
 \label{eq:oct-refined-theta}
\end{align}
Take the \(H^s\) inner product with \(\theta^{n+1}\).  The viscous defect
is kept in divergence form:
\begin{align}
 \nu\left|
 \left\langle
 \Lambda^s\nabla q^{n+1},
 \Lambda^s\nabla\theta^{n+1}
 \right\rangle
 \right|
 \le
 \frac{\nu}{2}\|\nabla\theta^{n+1}\|_{H^s}^2
 {}+C\nu\|q^{n+1}\|_{H^{s+1}}^2.
 \label{eq:oct-refined-viscous-defect}
\end{align}
The convective defect satisfies
\[
 \left|
 \langle\Lambda^s\eta_c^{n+1},
 \Lambda^s\theta^{n+1}\rangle
 \right|
 \le C\|\theta^{n+1}\|_{H^s}^2
 {}+C\|\eta_c^{n+1}\|_{H^s}^2.
\]
The remaining terms are estimated by
\eqref{eq:oct-refined-transport}--\eqref{eq:oct-refined-tame2}.  Since
\begin{align}
 \|\rho^n\|_{H^s}
 &\le CK_mN^{-(m-s)},\qquad
 \|\rho^n\|_{H^{s+1}}
 \le CK_mN^{-(m-s-1)},
 \label{eq:oct-refined-rho}
\end{align}
multiplication by \(2\tau\) and Young's inequality give
\begin{align}
 (1-C\tau)\|\theta^{n+1}\|_{H^s}^2
 &\le
 (1+C\tau)\|\theta^n\|_{H^s}^2
 {}+C\tau N^{-2(m-s-1)}
 \nonumber\\
 &\quad+
 C\tau\left(
 \|\eta_c^{n+1}\|_{H^s}^2
 {}+\nu\|q^{n+1}\|_{H^{s+1}}^2
 \right).
 \label{eq:oct-refined-recurrence}
\end{align}
Here \eqref{eq:oct-refined-viscous-defect} absorbs the viscous defect
into half of the dissipation.

Choose \(C\tau_0\le1/2\), iterate
\eqref{eq:oct-refined-recurrence}, and use
\(\theta^0=0\) and \eqref{eq:oct-refined-summed-defect}.  Discrete
Gronwall gives
\[
 \max_{0\le n\le M}\|\theta^n\|_{H^s}^2
 \le C\left(
 N^{-2(m-s-1)}+\tau^2
 \right).
\]
Combining this with \eqref{eq:oct-refined-rho} proves
\eqref{eq:oct-refined-NS-error}.  For fixed $\tau$, there are finitely
many time levels.  The uniform $H^m$ bound gives, by diagonal extraction,
weak $H^m$ and strong $H^{m-1}$ convergence at every level.  Since $m\ge3$,
the nonlinear terms pass to the limit in distributions.  Once the old
velocity is fixed, the homogeneous $L^2$ update identity gives uniqueness
of the next velocity, including at $\nu=0$.  Weak lower semicontinuity
then proves \eqref{eq:oct-refined-time-error}.

\emph{Inviscid limit.}
For $w=U^\nu-V$, subtract the two equations and write
$\nu\Delta U^\nu=\nu\Delta w+\nu\Delta V$ to obtain
\begin{equation}
 \partial_tw
 +\mathbb P(U^\nu\cdot\nabla w+w\cdot\nabla V)
 -\nu\Delta w
 =\nu\Delta V,
 \qquad w(0)=0.
 \label{eq:oct-w-equation}
\end{equation}
Using the positive-order estimates
\eqref{eq:oct-refined-transport}--\eqref{eq:oct-refined-tame1} gives
\[
 \frac{d}{dt}\|w\|_{H^s}
 \le CK_m\|w\|_{H^s}+C\nu\|V\|_{H^{s+2}}
 \le CK_m\|w\|_{H^s}+C\nu K_m.
\]
Since \(w(0)=0\), Gronwall proves
\eqref{eq:oct-refined-inviscid}.  Combining it with
\eqref{eq:oct-refined-NS-error} proves
\eqref{eq:oct-refined-Euler-error}.
\end{proof}

\begin{coro}[Low-order estimate from \(H^3\) data]
\label{coro:oct-low-L2}
Let \(m\ge3\), let \(u_0\in H^m_\sigma(\mathbb T^2)\), and assume the
common bound \eqref{eq:oct-refined-common} with this \(m\).  On the same
local interval and under the same viscosity- and mesh-independent step
restriction,
\begin{align}
 \max_{0\le n\le M}
 \|u_N^n-U^\nu(t_n)\|_{L^2}
 &\le C\left(\tau+N^{-(m-1)}\right),
 \label{eq:oct-low-L2-NS}\\
 \max_{0\le n\le M}
 \|u^n-U^\nu(t_n)\|_{L^2}
 &\le C\tau,
 \label{eq:oct-low-L2-time}\\
 \sup_{0\le t\le T}\|U^\nu(t)-V(t)\|_{L^2}
 &\le C\nu,
 \label{eq:oct-low-L2-inviscid}\\
 \max_{0\le n\le M}
 \|u_N^n-V(t_n)\|_{L^2}
 &\le C\left(\tau+N^{-(m-1)}+\nu\right).
 \label{eq:oct-low-L2-Euler}
\end{align}
Thus \(H^3\) initial data give the viscosity-uniform rate
\(O(\tau+N^{-2})\) in \(L^2\).
\end{coro}

\begin{proof}
Use the projected splitting from the proof of
Theorem~\ref{thm:oct-refined-sobolev} and test
\eqref{eq:oct-refined-theta} with $\theta^{n+1}$.
Since $\nabla\cdot u_N^n=0$,
\[
 \langle u_N^n\cdot\nabla\theta^{n+1},\theta^{n+1}\rangle=0.
\]
The uniform $H^3$ bounds and the Fourier projection estimate give
\[
 \|(\theta^n+\rho^n)\cdot\nabla U^\nu(t_{n+1})\|_2
 +\|u_N^n\cdot\nabla\rho^{n+1}\|_2
 \le C\bigl(\|\theta^n\|_2+N^{-(m-1)}\bigr),
\]
since $\|\rho^n\|_2\le CN^{-m}$ and
$\|\rho^n\|_{H^1}\le CN^{-(m-1)}$.
Absorb the viscous defect as in \eqref{eq:oct-refined-viscous-defect}.
The consistency bound \eqref{eq:oct-refined-summed-defect} at $s=1$
also controls the sum with $\eta_c^{n+1}$ in $L^2$ and $q^{n+1}$
in $H^1$. Thus \eqref{eq:oct-refined-recurrence} holds with $s=0$,
and discrete Gronwall, together with $\theta^0=0$, yields
$\max_n\|\theta^n\|_2\le C(\tau+N^{-(m-1)})$.
Adding the projection error proves \eqref{eq:oct-low-L2-NS}.

The time-discrete estimate \eqref{eq:oct-low-L2-time} and the
continuous inviscid estimate \eqref{eq:oct-low-L2-inviscid} follow
from Theorem~\ref{thm:oct-refined-sobolev} with $s=1$ and
$H^1\hookrightarrow L^2$. The triangle inequality then gives
\eqref{eq:oct-low-L2-Euler}.
\end{proof}

\begin{remark}[Dependence on spatial regularity]
\label{rem:m2an-stronger-data}
The spatial term is $N^{-(m-s-1)}$ at positive Sobolev order and
$N^{-(m-1)}$ in $L^2$. In particular, $H^4$ data give
$C(\tau+N^{-3})$ in $L^2$ for Navier--Stokes and
$C(\tau+N^{-3}+\nu)$ for Euler.
\end{remark}

\section{Endpoint Besov estimates}
\label{sec:oct-bes-endpoint-revision}

We write $U=U^\nu$ when the viscosity is fixed.
The endpoint estimates follow from localized $L^2$ energies.
Orthogonality controls the sharp projector on each block, including
those whose supports cross the cutoff.

\subsection{Projection and regularity}

\begin{lemma}[Sharp projection in an $L^2$-based Besov space]
\label{lem:oct-bes-sharp-projection}
Let $\Pi_N$ be the orthogonal Fourier projector onto
$\{|k|_\infty\leq N\}$.  For every $s\in\mathbb R$,
\begin{equation}
 \|\Pi_N f\|_{B^s_{2,1}}\leq \|f\|_{B^s_{2,1}}.
 \label{eq:oct-bes-sharp-stability}
\end{equation}
If $m>s$, then
\begin{equation}
 \|(I-\Pi_N)f\|_{B^s_{2,1}}
 \leq C_{m,s}N^{s-m}\|f\|_{B^m_{2,1}}.
 \label{eq:oct-bes-sharp-tail}
\end{equation}
The constants are independent of $N$.
\end{lemma}

\begin{proof}
The multipliers $\Pi_N$ and $\Delta_j$ commute, and $\Pi_N$ is an
orthogonal contraction on $L^2$.  This proves
\eqref{eq:oct-bes-sharp-stability} after multiplication by $2^{js}$ and
summation in $j$.  Moreover,
$(I-\Pi_N)\Delta_j=0$ when $2^j\leq cN$, for a constant depending only
on the supports of the Littlewood--Paley cutoffs.  On every remaining
block, including the transition blocks, we use only the $L^2$
contraction property.  Since $2^j\geq cN$ there,
\[
 2^{js}\|(I-\Pi_N)\Delta_jf\|_2
 \leq C N^{s-m}2^{jm}\|\Delta_jf\|_2.
\]
Summing this inequality proves \eqref{eq:oct-bes-sharp-tail}.
\end{proof}

We use the critical product and commutator estimates
\begin{align}
 \|fg\|_{B^1_{2,1}}
 &\leq C\|f\|_{B^1_{2,1}}\|g\|_{B^1_{2,1}},
 \label{eq:oct-bes-critical-product}\\
 \sum_{j\geq0}2^j
 \|[\Delta_j,a\cdot\nabla]b\|_2
 &\leq C\|a\|_{B^2_{2,1}}\|b\|_{B^1_{2,1}},
 \qquad \operatorname{div}a=0.
 \label{eq:oct-bes-critical-commutator}
\end{align}
They follow from Bony's decomposition and the critical embedding
$B^1_{2,1}(\mathbb T^2)\hookrightarrow L^\infty$.  In particular,
\begin{equation}
 \|a\cdot\nabla b\|_{B^1_{2,1}}
 \leq C\|a\|_{B^2_{2,1}}\|b\|_{B^2_{2,1}}.
 \label{eq:oct-bes-product-B2}
\end{equation}

\begin{lemma}[Continuous smoothing]
\label{lem:oct-bes-continuous-smoothing}
Let $U$ be a smooth divergence-free solution of Navier--Stokes on
$[0,T]$, $0\leq\nu\leq1$.  If
\[
 \sup_{0\leq t\leq T}\|U(t)\|_{B^3_{2,1}}\leq K_3,
\]
then
\begin{equation}
 \nu\int_0^T\|U(t)\|_{B^5_{2,1}}\,dt
 \leq C\bigl(\|U(0)\|_{B^3_{2,1}}+TK_3^2+T\|U(0)\|_2\bigr).
 \label{eq:oct-bes-nu-B5}
\end{equation}
The constant is independent of $\nu$.
More generally, for every $s>2$, the bound
\[
 \sup_{0\leq t\leq T}\|U(t)\|_{B^s_{2,1}}\leq K_s
\]
implies
\begin{equation}
 \nu\int_0^T\|U(t)\|_{B^{s+2}_{2,1}}\,dt
 \leq C_{s,T,K_s,\|U(0)\|_{B^s_{2,1}}}.
 \label{eq:oct-bes-nu-higher-smoothing}
\end{equation}
\end{lemma}

\begin{proof}
Apply $\Delta_j$ to the equation and take the $L^2$ inner product with
$\Delta_jU$.  The divergence-free transport term cancels, and the
localized Poincare inequality gives, for $j\geq1$,
\[
 \frac{d}{dt}\|\Delta_jU\|_2
 +c\nu2^{2j}\|\Delta_jU\|_2
 \leq \|[\Delta_j,U\cdot\nabla]U\|_2.
\]
After multiplication by $2^{3j}$, summation, and use of the standard
$B^3_{2,1}$ commutator estimate, we obtain
\[
 \frac{d}{dt}\|U\|_{B^3_{2,1}}
 +c\nu\|U\|_{B^5_{2,1}}
 \leq C\|U\|_{B^3_{2,1}}^2
   +C\nu\|\Delta_0U\|_2.
\]
For the low block, the energy inequality gives
$\|\Delta_0U(t)\|_2\leq C\|U(0)\|_2$. Integration proves
\eqref{eq:oct-bes-nu-B5}.  The same calculation after multiplication
by $2^{js}$, using the standard $B^s_{2,1}$ commutator estimate,
proves \eqref{eq:oct-bes-nu-higher-smoothing}.
\end{proof}

\begin{lemma}[Uniform local stability in $B^m_{2,1}$]
\label{lem:oct-bes-Bm-stability}
Let $m>2$ and $A=\|u_0\|_{B^m_{2,1}}$.  There are constants
$c_m,C_m>0$ such that, if
\begin{equation}
 T\leq \frac{c_m}{1+A},
 \label{eq:oct-bes-local-time}
\end{equation}
then the continuous solution, the time-discrete solution $u^n$, and
the sharp Fourier--Galerkin solution $u_N^n$ with initial value
$\Pi_Nu_0$ satisfy
\begin{align}
 &\sup_{0\leq t\leq T}\|U(t)\|_{B^m_{2,1}}
 +\nu\int_0^T\|U(t)\|_{B^{m+2}_{2,1}}\,dt
 \leq C_m(A+T\|u_0\|_2),
 \label{eq:oct-bes-continuous-Bm-stability}\\
 &\max_{0\leq n\leq M}
   \bigl(\|u^n\|_{B^m_{2,1}}+\|u_N^n\|_{B^m_{2,1}}\bigr)
 \leq C_m A,
 \label{eq:oct-bes-discrete-Bm-stability}\\
 &\nu\tau\sum_{n=0}^{M-1}
 \bigl(\|u^{n+1}\|_{B^{m+2}_{2,1}}
       +\|u_N^{n+1}\|_{B^{m+2}_{2,1}}\bigr)
 \leq C_m(A+TA^2+T\|u_0\|_2).
 \label{eq:oct-bes-discrete-Bm-dissipation}
\end{align}
All constants are independent of $\nu\in[0,1]$, $N$, and
$\tau=T/M$.
\end{lemma}

\begin{proof}
The case $A=0$ is the zero solution; assume $A>0$.
The continuous estimate is the level-$m$ version of the proof of
Lemma~\ref{lem:oct-bes-continuous-smoothing}.  In particular,
\begin{equation}
 \frac{d}{dt}\|U\|_{B^m_{2,1}}
 +c_m\nu\|U\|_{B^{m+2}_{2,1}}
 \leq C_m\|U\|_{B^m_{2,1}}^2+C_m\nu\|\Delta_0U\|_2,
 \label{eq:oct-bes-continuous-Bm-inequality}
\end{equation}
where the last term is bounded by the $L^2$ energy inequality.  The
standard continuity bootstrap under \eqref{eq:oct-bes-local-time}
proves \eqref{eq:oct-bes-continuous-Bm-stability}.

For either discrete scheme, apply $\Delta_j$ and pair the result with
$\Delta_ju^{n+1}$.  In the Fourier case
$\Delta_ju_N^{n+1}\in X_N$, so self-adjointness gives
\[
 \langle\mathbb P\Pi_N\Delta_jF,\Delta_ju_N^{n+1}\rangle
 =\langle\Delta_jF,\Delta_ju_N^{n+1}\rangle.
\]
Transport cancellation and the level-$m$ commutator estimate give, for
$w^n=u^n$ or $w^n=u_N^n$,
\begin{align}
 \|w^{n+1}\|_{B^m_{2,1}}
 +c_m\nu\tau\|w^{n+1}\|_{B^{m+2}_{2,1}}
 \leq{}&\|w^n\|_{B^m_{2,1}}
 +C_m\tau\|w^n\|_{B^m_{2,1}}
                 \|w^{n+1}\|_{B^m_{2,1}}
 \nonumber\\
 &+C_m\nu\tau\|\Delta_0w^{n+1}\|_2.
 \label{eq:oct-bes-discrete-Bm-step}
\end{align}
The last term is controlled by the discrete $L^2$ energy inequality.
If $X_n=\|w^n\|_{B^m_{2,1}}$ and $X_n\leq2A$, then, after reducing
$c_m$ in \eqref{eq:oct-bes-local-time},
\[
 X_{n+1}\leq
 (1-C_m\tau X_n)^{-1}(X_n+C_m\nu\tau A)
 \leq (1+4C_mA\tau)X_n+C'_m\tau A.
\]
Here we used $\|\Delta_0w^{n+1}\|_2\leq A$.  Discrete Gronwall and a
further reduction of $c_m$ give $X_{n+1}<2A$, which closes
the bootstrap.  Summing \eqref{eq:oct-bes-discrete-Bm-step} then proves
\eqref{eq:oct-bes-discrete-Bm-dissipation}.  Finally,
$\|\Pi_Nu_0\|_{B^m_{2,1}}\leq A$ follows from
Lemma~\ref{lem:oct-bes-sharp-projection}.
\end{proof}

\begin{remark*}[Approximation of the initial data]
The estimates hold first for smooth data and extend by
the usual Friedrichs approximation.
For the pure time-discrete scheme, the uniform Galerkin bound gives,
at each of the finitely many time levels, compact convergence in
$H^{m-\epsilon}$ with $0<\epsilon<m-2$. Passing to the limit yields the
update; blockwise convergence and Fatou's lemma retain the $B^m_{2,1}$
bound. With the previous field fixed, uniqueness follows from the
$L^2$ skew-symmetric energy estimate.
\end{remark*}

\subsection{Consistency}

Set $U^n=U(t_n)$ and define the residual
\begin{equation}
 \eta^{n+1}:=
 \frac{U^{n+1}-U^n}{\tau}
 +\mathbb P(U^n\cdot\nabla U^{n+1})
 -\nu\Delta U^{n+1}.
 \label{eq:oct-bes-exact-residual}
\end{equation}
Set
\begin{align}
 q^{n+1}
 &:=\frac1\tau\int_{t_n}^{t_{n+1}}
       \bigl(U(t)-U^{n+1}\bigr)\,dt,
 \label{eq:oct-bes-q-definition}\\
 \eta_c^{n+1}
 &:=\mathbb P\left(
 U^n\cdot\nabla U^{n+1}
 -\frac1\tau\int_{t_n}^{t_{n+1}}
       U(t)\cdot\nabla U(t)\,dt\right).
 \label{eq:oct-bes-etac-definition}
\end{align}
Time integration of the equation gives
\begin{equation}
 \eta^{n+1}=\eta_c^{n+1}+\nu\Delta q^{n+1}.
 \label{eq:oct-bes-defect-splitting}
\end{equation}

\begin{lemma}[Summed consistency]
\label{lem:oct-bes-summed-consistency}
Assume
\begin{equation}
 \sup_{0\leq t\leq T}\|U(t)\|_{B^3_{2,1}}
 +\nu\int_0^T\|U(t)\|_{B^5_{2,1}}\,dt
 \leq K.
 \label{eq:oct-bes-consistency-assumption}
\end{equation}
Then
\begin{equation}
 \tau\sum_{n=0}^{M-1}\|\eta_c^{n+1}\|_{B^1_{2,1}}
 +\nu\tau\sum_{n=0}^{M-1}\|q^{n+1}\|_{B^3_{2,1}}
 \leq C_{T,K}\tau.
 \label{eq:oct-bes-consistency-bound}
\end{equation}
The constant is uniform for $0\leq\nu\leq1$.
\end{lemma}

\begin{proof}
The equation and the product estimates imply
\begin{align}
 \int_0^T\|\partial_tU\|_{B^1_{2,1}}\,dt
 &\leq C_{T,K},
 \label{eq:oct-bes-ut-B1}\\
 \int_0^T\|\partial_tU\|_{B^2_{2,1}}\,dt
 &\leq C_{T,K},
 \label{eq:oct-bes-ut-B2}\\
 \nu\int_0^T\|\partial_tU\|_{B^3_{2,1}}\,dt
 &\leq C_{T,K}.
 \label{eq:oct-bes-nu-ut-B3}
\end{align}
Indeed, the first two bounds follow from
\begin{align*}
 \|\partial_tU\|_{B^1_{2,1}}
 &\leq C\|U\|_{B^2_{2,1}}^2
       +\nu\|U\|_{B^3_{2,1}},\\
 \|\partial_tU\|_{B^2_{2,1}}
 &\leq C\|U\|_{B^2_{2,1}}\|U\|_{B^3_{2,1}}
       +\nu\|U\|_{B^4_{2,1}},
\end{align*}
and the elementary inequality
$\|U\|_{B^4_{2,1}}\leq
 C(\|U\|_{B^3_{2,1}}+\|U\|_{B^5_{2,1}})$.
Likewise,
\[
 \nu\|\partial_tU\|_{B^3_{2,1}}
 \leq C\nu\|U\|_{B^3_{2,1}}\|U\|_{B^4_{2,1}}
       +\nu^2\|U\|_{B^5_{2,1}},
\]
which proves \eqref{eq:oct-bes-nu-ut-B3}.

For $t\in[t_n,t_{n+1}]$, the critical product law gives
\begin{align*}
 &\|U^n\cdot\nabla U^{n+1}
       -U(t)\cdot\nabla U(t)\|_{B^1_{2,1}}\\
 &\quad\leq C K
 \left(
   \int_{t_n}^{t}\|\partial_tU(s)\|_{B^1_{2,1}}\,ds
  +\int_t^{t_{n+1}}\|\partial_tU(s)\|_{B^2_{2,1}}\,ds
 \right).
\end{align*}
Integrating first in $t$ and then summing in $n$ yields
\[
 \tau\sum_n\|\eta_c^{n+1}\|_{B^1_{2,1}}
 \leq C K\tau
 \left(
  \|\partial_tU\|_{L^1(0,T;B^1_{2,1})}
 +\|\partial_tU\|_{L^1(0,T;B^2_{2,1})}
 \right).
\]
Finally, \eqref{eq:oct-bes-q-definition} and Fubini's theorem give
\begin{align*}
 \nu\tau\sum_n\|q^{n+1}\|_{B^3_{2,1}}
 &\leq
 \nu\sum_n\int_{t_n}^{t_{n+1}}
       \int_t^{t_{n+1}}
       \|\partial_tU(s)\|_{B^3_{2,1}}\,ds\,dt\\
 &\leq \nu\tau
       \|\partial_tU\|_{L^1(0,T;B^3_{2,1})}.
\end{align*}
Equations \eqref{eq:oct-bes-ut-B1}--\eqref{eq:oct-bes-nu-ut-B3}
complete the proof.
\end{proof}

\subsection{Time-discrete error}
\label{subsec:oct-det-time-endpoint}

The $L^2$ block estimates remove pressure and retain
frequency-localized dissipation. Direct endpoint testing has the
limitations described below.

\begin{remark}[Diffusion and endpoint testing]
\label{rem:oct-det-pinfty}
For a smooth vector field $w_j=\Delta_jw$ and $2\le p<\infty$,
integration by parts gives
\begin{align}
 \left\langle\Delta w_j,|w_j|^{p-2}w_j\right\rangle
 ={}&-\int |w_j|^{p-2}|\nabla w_j|^2\,dx
 \nonumber\\
 &-(p-2)\sum_{\ell=1}^2\int |w_j|^{p-4}
       (w_j\cdot\partial_\ell w_j)^2\,dx
 \leq0.
 \label{eq:oct-det-diffusion-sign}
\end{align}
For a pressure-free transport--diffusion equation, this sign yields an
$L^p$ norm estimate that passes to $p=\infty$. For velocity equations,
the nonlinear test is generally not divergence free:
\[
 \left\langle\nabla p_j,|w_j|^{p-2}w_j\right\rangle
 =-\left\langle p_j,\operatorname{div}(|w_j|^{p-2}w_j)\right\rangle.
\]
Thus the pressure cancellation in an $L^2$ test does not extend to this
test when $p\ne2$.

Moreover, the positive constant in a frequency-localized coercive
inequality depends on $p$; see \cite{Li13}. To see why it cannot stay
positive as $p\to\infty$, fix an integer $K\ge1$ and set
\[
 g(y)=\frac32\cos y-\frac12\cos^3y
     =\frac98\cos y-\frac18\cos(3y),
 \qquad f_K(x)=g(Kx_1).
\]
This polynomial has frequencies $K$ and $3K$. Near each point $y_0$
where $|g(y_0)|=1$,
\[
 1-|g(y)|\asymp |y-y_0|^4,
 \qquad |g'(y)|\asymp |y-y_0|^3.
\]
With $z=p^{1/4}(y-y_0)$, the contributions near these extrema give
\begin{equation}
 \frac{-\displaystyle\int_{\mathbb T^2}
       (\Delta f_K)|f_K|^{p-2}f_K\,dx}
      {\displaystyle\int_{\mathbb T^2}|f_K|^p\,dx}
 =(p-1)
 \frac{\displaystyle\int |f_K|^{p-2}|\nabla f_K|^2\,dx}
      {\displaystyle\int |f_K|^p\,dx}
 \asymp K^2p^{-1/2}\longrightarrow0.
 \label{eq:oct-det-flat-maximum}
\end{equation}
Consequently, even on this fixed Fourier band, an estimate of the form
\[
 -\langle\Delta f,|f|^{p-2}f\rangle\ge cK^2\|f\|_p^p
\]
cannot hold with a positive constant independent of $p$. The example
concerns band-limited fields and does not require $f_K$ to be the image
of a prescribed block operator.
\end{remark}

\begin{proposition}[$B^1_{2,1}$ error and its endpoint consequence]
\label{prop:oct-bes-time-Besov1}
Let $u_0\in B^3_{2,1}(\mathbb T^2)$ be divergence free, set
$A=\|u_0\|_{B^3_{2,1}}$, and let $u^n$ solve the time-discrete
scheme with $u^0=u_0$.  Choose the common local interval as in
Lemma~\ref{lem:oct-bes-Bm-stability}.  There is
$\tau_0=\tau_0(T,A)>0$ such that, for $0<\tau\leq\tau_0$,
\begin{equation}
 \max_{0\leq n\leq M}\|u^n-U(t_n)\|_{B^1_{2,1}}
 \leq C_{T,A}\tau.
 \label{eq:oct-det-time-B1-natural}
\end{equation}
Consequently, by blockwise Bernstein,
\begin{equation}
 \max_{0\leq n\leq M}\|u^n-U(t_n)\|_{B^0_{\infty,1}}
 \leq C_{T,A}\tau.
 \label{eq:oct-det-time-Besov1}
\end{equation}
Both estimates are uniform for $0\leq\nu\leq1$.
\end{proposition}

\begin{proof}
Put $U^n=U(t_n)$ and $e^n=u^n-U^n$.
Lemma~\ref{lem:oct-bes-Bm-stability}, used with $m=3$, gives a constant
$K$ depending only on $T$ and $A$ such that
\begin{equation}
 \sup_{0\leq t\leq T}\|U(t)\|_{B^3_{2,1}}
 +\sup_{0\leq n\leq M}\|u^n\|_{B^3_{2,1}}
 +\nu\int_0^T\|U(t)\|_{B^5_{2,1}}\,dt
 \leq K.
 \label{eq:oct-det-uniform-B3}
\end{equation}
We use $q^{n+1}$ and $\eta_c^{n+1}$ defined in
\eqref{eq:oct-bes-q-definition}--\eqref{eq:oct-bes-etac-definition}.
Subtracting the time-integrated exact equation from the update gives,
after application of $\Delta_j$,
\begin{align}
 \Delta_je^{n+1}
 ={}&\Delta_je^n
 +\nu\int_{t_n}^{t_{n+1}}
   \Delta\Delta_j\bigl(u^{n+1}-U(t)\bigr)\,dt
 \nonumber\\
 &+\int_{t_n}^{t_{n+1}}\mathbb P\Delta_j
 \bigl(U(t)\cdot\nabla U(t)
       -u^n\cdot\nabla u^{n+1}\bigr)\,dt.
 \label{eq:oct-det-integral-error}
\end{align}

By mean preservation, $e^n$ and $\Delta_0e^n$ have zero mean.
Combining the torus
Poincar\'e inequality at $j=0$ with the annular estimate for $j\ge1$ gives
\begin{align}
 \|\nabla\Delta_0e^n\|_2
 &\geq c\|\Delta_0e^n\|_2,\nonumber\\
 \|\nabla\Delta_je^n\|_2
 &\geq c2^j\|\Delta_je^n\|_2,
 \qquad j\geq1.
 \label{eq:oct-det-low-poincare}
\end{align}

Set $z_j=\Delta_je^{n+1}$. Pairing \eqref{eq:oct-det-integral-error}
with $z_j$ gives $\|z_j\|_2^2=I_1+I_2+I_3$, where
\begin{equation}
 I_1=\langle\Delta_je^n,z_j\rangle
 \leq\|\Delta_je^n\|_2\|z_j\|_2.
 \label{eq:oct-det-I1}
\end{equation}
Split the viscous term as $I_2=I_{2,1}+I_{2,2}$, using
$u^{n+1}-U(t)=e^{n+1}+U^{n+1}-U(t)$. Then
\eqref{eq:oct-det-low-poincare} gives
\begin{equation}
 I_{2,1}
 =-\nu\tau\|\nabla z_j\|_2^2
 \leq-c\nu\tau2^{2j}\|z_j\|_2^2,
 \label{eq:oct-det-I21}
\end{equation}
where $2^{2j}=1$ at $j=0$.  From \eqref{eq:oct-bes-q-definition},
\begin{align}
 I_{2,2}
 &=\nu\int_{t_n}^{t_{n+1}}
   \langle\Delta\Delta_j(U^{n+1}-U(t)),z_j\rangle\,dt
 \nonumber\\
 &=\nu\tau\langle\nabla\Delta_jq^{n+1},\nabla z_j\rangle,
 \label{eq:oct-det-I22identity}
\end{align}
and the upper Bernstein inequality, also valid for $j=0$, yields
\begin{equation}
 |I_{2,2}|
 \leq C\nu\tau2^{2j}
 \|\Delta_jq^{n+1}\|_2\|z_j\|_2.
 \label{eq:oct-det-I22}
\end{equation}

For the nonlinear term, write
\begin{align}
 U(t)\cdot\nabla U(t)-u^n\cdot\nabla u^{n+1}
 ={}&\underbrace{(U(t)-U^n)\cdot\nabla U(t)}_{F_{3,1}(t)}
 \nonumber\\
 &+\underbrace{U^n\cdot\nabla(U(t)-U^{n+1})}_{F_{3,2}(t)}
 \nonumber\\
 &+\underbrace{(U^n-u^n)\cdot\nabla U^{n+1}}_{F_{3,3}}
 \nonumber\\
 &+\underbrace{u^n\cdot\nabla(U^{n+1}-u^{n+1})}_{F_{3,4}}.
 \label{eq:oct-det-four-term}
\end{align}
Denote the contribution of $F_{3,k}$ by $I_{3,k}$. The test $z_j$ is
divergence free, so the Leray projector drops out. Cauchy--Schwarz gives
\begin{align}
 |I_{3,1}|
 &\leq\int_{t_n}^{t_{n+1}}
 \|\Delta_jF_{3,1}(t)\|_2\,dt\,\|z_j\|_2,
 \label{eq:oct-det-I31}\\
 |I_{3,2}|
 &\leq\int_{t_n}^{t_{n+1}}
 \|\Delta_jF_{3,2}(t)\|_2\,dt\,\|z_j\|_2,
 \label{eq:oct-det-I32}\\
 |I_{3,3}|
 &\leq\tau
 \|\Delta_j(e^n\cdot\nabla U^{n+1})\|_2\,\|z_j\|_2.
 \label{eq:oct-det-I33}
\end{align}
For $F_{3,4}=-u^n\cdot\nabla e^{n+1}$, transport cancellation leaves
\begin{align}
 |I_{3,4}|
 &=\tau\left|
 \left\langle[\Delta_j,u^n\cdot\nabla]e^{n+1},z_j\right\rangle
 \right|
 \nonumber\\
 &\leq\tau
 \|[\Delta_j,u^n\cdot\nabla]e^{n+1}\|_2\|z_j\|_2.
 \label{eq:oct-det-I34}
\end{align}

Combining \eqref{eq:oct-det-I1}, \eqref{eq:oct-det-I21},
\eqref{eq:oct-det-I22}, and \eqref{eq:oct-det-I31}--
\eqref{eq:oct-det-I34}, and dividing by $\|z_j\|_2$, gives
\begin{align}
 \|z_j\|_2+c\nu\tau2^{2j}\|z_j\|_2
 \leq{}&\|\Delta_je^n\|_2
 +C\nu\tau2^{2j}\|\Delta_jq^{n+1}\|_2
 \nonumber\\
 &+\int_{t_n}^{t_{n+1}}
   \bigl(\|\Delta_jF_{3,1}(t)\|_2
        +\|\Delta_jF_{3,2}(t)\|_2\bigr)\,dt
 \nonumber\\
 &+C\tau\|\Delta_j(e^n\cdot\nabla U^{n+1})\|_2
 \nonumber\\
 &+C\tau\|[\Delta_j,u^n\cdot\nabla]e^{n+1}\|_2.
 \label{eq:oct-det-block-recurrence}
\end{align}
When $z_j=0$, the inequality is immediate.

To sum the block inequalities, set
\begin{align}
 A_{3,1}^n
 &:=\int_{t_n}^{t_{n+1}}
   \|F_{3,1}(t)\|_{B^1_{2,1}}\,dt,
 \nonumber\\
 A_{3,2}^n
 &:=\int_{t_n}^{t_{n+1}}
   \|F_{3,2}(t)\|_{B^1_{2,1}}\,dt.
 \label{eq:oct-det-A31-A32}
\end{align}
The product estimate and the fundamental theorem of calculus imply
\begin{align}
 A_{3,1}^n
 &\leq CK\int_{t_n}^{t_{n+1}}
       \int_{t_n}^{t}\|\partial_sU(s)\|_{B^1_{2,1}}\,ds\,dt,
 \label{eq:oct-det-A31-time}\\
 A_{3,2}^n
 &\leq CK\int_{t_n}^{t_{n+1}}
       \int_t^{t_{n+1}}\|\partial_sU(s)\|_{B^2_{2,1}}\,ds\,dt.
 \label{eq:oct-det-A32-time}
\end{align}
Moreover,
\begin{align*}
 \|e^n\cdot\nabla U^{n+1}\|_{B^1_{2,1}}
 &\leq CK\|e^n\|_{B^1_{2,1}},\\
 \sum_{j\geq0}2^j
 \|[\Delta_j,u^n\cdot\nabla]e^{n+1}\|_2
 &\leq CK\|e^{n+1}\|_{B^1_{2,1}}.
\end{align*}
Therefore multiplication of \eqref{eq:oct-det-block-recurrence} by $2^j$
and summation give
\begin{align}
 \|e^{n+1}\|_{B^1_{2,1}}
 +c\nu\tau\|e^{n+1}\|_{B^3_{2,1}}
 \leq{}&\|e^n\|_{B^1_{2,1}}
 +A_{3,1}^n+A_{3,2}^n
 +C\nu\tau\|q^{n+1}\|_{B^3_{2,1}}
 \nonumber\\
 &+CK\tau\|e^n\|_{B^1_{2,1}}
 +CK\tau\|e^{n+1}\|_{B^1_{2,1}}.
 \label{eq:oct-det-B1-recurrence}
\end{align}
By \eqref{eq:oct-bes-ut-B1}--\eqref{eq:oct-bes-nu-ut-B3} and
\eqref{eq:oct-det-uniform-B3},
\begin{equation}
 \|\partial_tU\|_{L^1(0,T;B^1_{2,1})}
 +\|\partial_tU\|_{L^1(0,T;B^2_{2,1})}
 +\nu\|\partial_tU\|_{L^1(0,T;B^3_{2,1})}
 \leq C_{T,K}.
 \label{eq:oct-det-ut-summed}
\end{equation}
Fubini's theorem yields
\begin{align}
 \sum_{n=0}^{M-1}(A_{3,1}^n+A_{3,2}^n)
 \leq{}&CK\tau\left(
 \|\partial_tU\|_{L^1(0,T;B^1_{2,1})}
 +\|\partial_tU\|_{L^1(0,T;B^2_{2,1})}
 \right)
 \nonumber\\
 \leq{}&C_{T,K}\tau.
 \label{eq:oct-det-A-summed}
\end{align}
Also, because
\[
 \tau q^{n+1}
 =-\int_{t_n}^{t_{n+1}}
    \int_t^{t_{n+1}}\partial_sU(s)\,ds\,dt,
\]
another application of Fubini gives
\begin{equation}
 \nu\tau\sum_{n=0}^{M-1}\|q^{n+1}\|_{B^3_{2,1}}
 \leq\nu\tau
 \|\partial_tU\|_{L^1(0,T;B^3_{2,1})}
 \leq C_{T,K}\tau.
 \label{eq:oct-det-q-summed}
\end{equation}

Drop the nonnegative dissipation in
\eqref{eq:oct-det-B1-recurrence} and choose $\tau_0$ so that
$CK\tau_0\leq1/2$.  Then
\begin{align}
 (1-CK\tau)\|e^{n+1}\|_{B^1_{2,1}}
 \leq{}&(1+CK\tau)\|e^n\|_{B^1_{2,1}}
 \nonumber\\
 &+A_{3,1}^n+A_{3,2}^n
 +C\nu\tau\|q^{n+1}\|_{B^3_{2,1}}.
 \label{eq:oct-det-final-recurrence}
\end{align}
Since $e^0=0$, discrete Gronwall together with
\eqref{eq:oct-det-A-summed}--\eqref{eq:oct-det-q-summed} yields
\begin{equation}
 \max_{0\leq n\leq M}\|e^n\|_{B^1_{2,1}}
 \leq C_{T,K}\tau.
 \label{eq:oct-det-B1-error}
\end{equation}
Finally, blockwise Bernstein in two dimensions gives
\begin{equation}
 \|f\|_{B^0_{\infty,1}}
 \leq C\sum_{j\geq0}2^j\|\Delta_jf\|_2
 =C\|f\|_{B^1_{2,1}}.
 \label{eq:oct-bes-B21-to-Binf1}
\end{equation}
Together with \eqref{eq:oct-det-B1-error}, this proves both assertions.
\end{proof}

\begin{proposition}[$B^0_{\infty,2}$ error]
\label{prop:oct-bes-time-Besov2}
Let $u_0\in B^3_{\infty,2}(\mathbb T^2)$ be divergence free. On the
local interval and for the time steps of
Theorem~\ref{thm:oct-refined-sobolev},
\begin{align}
 \max_n\|u^n-U(t_n)\|_{B^0_{\infty,2}}
 &\leq C\tau,
 \label{eq:oct-bes-time-Besov2}\\
 \max_n\|u_N^n-U(t_n)\|_{B^0_{\infty,2}}
 &\leq C(\tau+N^{-1}).
 \label{eq:oct-bes-Fourier-Besov2}
\end{align}
The constants are independent of $\nu\in[0,1]$.
\end{proposition}

\begin{proof}
Finite volume and blockwise Bernstein in two dimensions give
\begin{align*}
 \|f\|_{H^3}^2
 &\leq C\sum_j2^{6j}\|\Delta_jf\|_\infty^2
 =C\|f\|_{B^3_{\infty,2}}^2,\\
 \|f\|_{B^0_{\infty,2}}^2
 &\leq C\sum_j2^{2j}\|\Delta_jf\|_2^2
 \leq C\|f\|_{H^1}^2.
\end{align*}
Hence $B^3_{\infty,2}\hookrightarrow H^3$ and
$H^1\hookrightarrow B^0_{\infty,2}$.  Theorem
~\ref{thm:oct-refined-sobolev} gives respectively
$\|u^n-U(t_n)\|_{H^1}\leq C\tau$ when $\Pi_N=I$, and
$\|u_N^n-U(t_n)\|_{H^1}\leq C(\tau+N^{-1})$ for the sharp Fourier
scheme.  The two embeddings complete the proof.
\end{proof}

\subsection{Fourier--Galerkin error}

We estimate the Fourier--Galerkin approximation \eqref{eq:oct-scheme}.

\begin{theorem}[Fourier--Galerkin endpoint error]
\label{thm:oct-bes-sharp-Fourier-Besov1}
Let $m\geq3$ be an integer and let the divergence-free initial datum
belong to $B^m_{2,1}(\mathbb T^2)$.  Put
$A=\|u_0\|_{B^m_{2,1}}$, choose the common local interval supplied by
Lemma~\ref{lem:oct-bes-Bm-stability}, and denote the corresponding
stability constant by $K_m=K_m(m,T,A)$.
Then, for $N\geq1$ and $0<\tau\leq\tau_0(K_m)$,
\begin{equation}
 \max_{0\leq n\leq M}
 \|u_N^n-U(t_n)\|_{B^1_{2,1}}
 \leq C_{T,K_m}\bigl(\tau+N^{-(m-2)}\bigr).
 \label{eq:oct-bes-Fourier-B1-natural}
\end{equation}
Consequently, by blockwise Bernstein,
\begin{equation}
 \max_{0\leq n\leq M}
 \|u_N^n-U(t_n)\|_{B^0_{\infty,1}}
 \leq C_{T,K_m}\bigl(\tau+N^{-(m-2)}\bigr).
 \label{eq:oct-bes-corrected-Fourier-rate}
\end{equation}
The constant is independent of $\nu$, $\tau$, and $N$.
\end{theorem}

\begin{proof}
Lemma~\ref{lem:oct-bes-Bm-stability} gives
\begin{equation}
 \sup_{0\leq t\leq T}\|U(t)\|_{B^m_{2,1}}
 +\nu\int_0^T\|U(t)\|_{B^{m+2}_{2,1}}\,dt
 +\sup_n\|u_N^n\|_{B^2_{2,1}}
 \leq K_m,
 \label{eq:oct-bes-Fourier-hypothesis}
\end{equation}
where $K_m$ depends only on $m,T,A$ and is independent of $N$
and $\nu$.
Set $U_N^n=\Pi_NU^n$ and
\[
 \theta^n:=u_N^n-U_N^n\in X_N,
 \qquad \rho^n:=(\Pi_N-I)U^n.
\]
Applying $\Pi_N$ to \eqref{eq:oct-bes-exact-residual}, and then replacing
$U^n$ by $U_N^n$ in the semi-implicit nonlinear term, gives
\begin{equation}
 \frac{U_N^{n+1}-U_N^n}{\tau}
 +\mathbb P\Pi_N(U_N^n\cdot\nabla U_N^{n+1})
 -\nu\Delta U_N^{n+1}
 =\mathcal R_N^{n+1},
 \label{eq:oct-bes-projected-exact-equation}
\end{equation}
where
\begin{align}
 \mathcal R_N^{n+1}
 :={}&\Pi_N\eta_c^{n+1}
 +\nu\Delta\Pi_Nq^{n+1}
 +S_N^{n+1},
 \label{eq:oct-bes-RN-definition}\\
 S_N^{n+1}
 :={}&\mathbb P\Pi_N\left(
 U_N^n\cdot\nabla U_N^{n+1}
 -U^n\cdot\nabla U^{n+1}\right).
 \label{eq:oct-bes-SN-definition}
\end{align}
By Lemma~\ref{lem:oct-bes-sharp-projection},
\eqref{eq:oct-bes-critical-product}, and
\[
 U_N^n\cdot\nabla U_N^{n+1}
 -U^n\cdot\nabla U^{n+1}
 =\rho^n\cdot\nabla U_N^{n+1}
  +U^n\cdot\nabla\rho^{n+1},
\]
we have
\begin{align}
 \|S_N^{n+1}\|_{B^1_{2,1}}
 &\leq C\left(
  \|\rho^n\|_{B^1_{2,1}}\|U_N^{n+1}\|_{B^2_{2,1}}
 +\|U^n\|_{B^1_{2,1}}\|\rho^{n+1}\|_{B^2_{2,1}}
 \right)
 \nonumber\\
 &\leq C K_m^2N^{-(m-2)}.
 \label{eq:oct-bes-spatial-residual}
\end{align}
The differentiated projection error $\rho^{n+1}$ gives the exponent $m-2$.

Combining \eqref{eq:oct-bes-spatial-residual} with
Lemma~\ref{lem:oct-bes-summed-consistency} yields
\begin{equation}
 \tau\sum_{n=0}^{M-1}
 \left(
  \|\Pi_N\eta_c^{n+1}+S_N^{n+1}\|_{B^1_{2,1}}
  +\nu\|\Pi_Nq^{n+1}\|_{B^3_{2,1}}
 \right)
 \leq C_{T,K_m}\bigl(\tau+N^{-(m-2)}\bigr).
 \label{eq:oct-bes-RN-summed}
\end{equation}

Subtract \eqref{eq:oct-bes-projected-exact-equation} from
\eqref{eq:oct-scheme}.  The error satisfies
\begin{align}
 \frac{\theta^{n+1}-\theta^n}{\tau}
 &+\mathbb P\Pi_N(u_N^n\cdot\nabla\theta^{n+1})
 -\nu\Delta\theta^{n+1}
 \nonumber\\
 &=-\mathbb P\Pi_N(\theta^n\cdot\nabla U_N^{n+1})
  -\mathcal R_N^{n+1}.
 \label{eq:oct-bes-theta-equation}
\end{align}
Apply $\Delta_j$ and test with $\Delta_j\theta^{n+1}\in X_N$.
Self-adjointness gives
\begin{equation}
 \left\langle
  \mathbb P\Pi_N\Delta_jF,\Delta_j\theta^{n+1}
 \right\rangle
 =\left\langle
  \Delta_jF,\Delta_j\theta^{n+1}
 \right\rangle.
 \label{eq:oct-bes-projector-removal}
\end{equation}
Transport cancellation and
\eqref{eq:oct-bes-critical-product}--\eqref{eq:oct-bes-critical-commutator}
yield
\begin{align}
 \|\theta^{n+1}\|_{B^1_{2,1}}
 \leq{}&\|\theta^n\|_{B^1_{2,1}}
 +CK_m\tau\left(
   \|\theta^n\|_{B^1_{2,1}}
  +\|\theta^{n+1}\|_{B^1_{2,1}}
 \right)
 \nonumber\\
 &+\tau\|\Pi_N\eta_c^{n+1}+S_N^{n+1}\|_{B^1_{2,1}}
 +C\nu\tau\|\Pi_Nq^{n+1}\|_{B^3_{2,1}}.
 \label{eq:oct-bes-theta-recurrence}
\end{align}
Since $\theta^0=0$, absorption, discrete Gronwall, and
\eqref{eq:oct-bes-RN-summed} imply
\begin{equation}
 \max_n\|\theta^n\|_{B^1_{2,1}}
 \leq C_{T,K_m}\bigl(\tau+N^{-(m-2)}\bigr).
 \label{eq:oct-bes-theta-bound}
\end{equation}
Finally,
\[
 u_N^n-U^n=\theta^n+\rho^n,
 \qquad
 \|\rho^n\|_{B^1_{2,1}}
 \leq C N^{-(m-1)}\|U^n\|_{B^m_{2,1}}.
\]
Together with \eqref{eq:oct-bes-theta-bound}, this proves
\eqref{eq:oct-bes-Fourier-B1-natural}; the embedding
\eqref{eq:oct-bes-B21-to-Binf1} then yields
\eqref{eq:oct-bes-corrected-Fourier-rate}.
\end{proof}

\begin{remark}[$L^\infty$-based data]
On the finite-volume torus,
$B^m_{\infty,1}\hookrightarrow B^m_{2,1}\hookrightarrow H^m$.
Thus Proposition~\ref{prop:oct-bes-time-Besov1} gives its endpoint
$O(\tau)$ bound for $B^3_{\infty,1}$ data, and
Theorem~\ref{thm:oct-bes-sharp-Fourier-Besov1} gives
$O(\tau+N^{-(m-2)})$ for $B^m_{\infty,1}$ data.  Their intervals,
step thresholds, and constants may then be chosen in terms of those
stronger initial norms.  In particular, $B^4_{\infty,1}$ data give
$\tau+N^{-2}$ and $B^5_{\infty,1}$ data give $\tau+N^{-3}$.
\end{remark}

\subsection{Higher Besov norms}
\label{sec:oct-higher}

Changing the dyadic weight gives estimates in higher Besov norms.

\begin{corollary}[Approximation in higher Besov norms]
\label{cor:oct-higher}
Let $r\ge0$ be an integer, let each displayed regularity index $m$ be
an integer, and let the initial velocity be divergence free.
On the common local interval, the following estimates hold uniformly for
$\nu\in[0,1]$ and $\tau\le\tau_0$:
\begin{align}
 u_0\in B^{r+3}_{\infty,1}:\quad
 &\max_n\|u^n-U^\nu(t_n)\|_{B^r_{\infty,1}}\le C\tau,
 \label{eq:oct-higher-time1}\\
 u_0\in B^m_{\infty,1},\quad m\ge r+3:\quad
 &\max_n\|u_N^n-U^\nu(t_n)\|_{B^r_{\infty,1}}
 \le C(\tau+N^{-(m-r-2)}),
 \label{eq:oct-higher-space1}\\
 u_0\in B^m_{\infty,2},\quad m\ge r+3:\quad
 &\max_n\|u_N^n-U^\nu(t_n)\|_{B^r_{\infty,2}}
 \le C(\tau+N^{-(m-r-2)}).
 \label{eq:oct-higher-space2}
\end{align}
For the last case, the pure time-discrete error is $C\tau$.
Constants depend on the indicated initial norm and common local interval.
\end{corollary}

\begin{proof}
Set $\ell=r+1$. For a divergence-free transporting field $a$, the
product and commutator estimates become
\[
 \|fg\|_{B^\ell_{2,1}}\le C_\ell\|f\|_{B^\ell_{2,1}}\|g\|_{B^\ell_{2,1}},
 \qquad
 \sum_j2^{j\ell}\|[\Delta_j,a\cdot\nabla]b\|_2
 \le C_\ell\|a\|_{B^{\ell+1}_{2,1}}\|b\|_{B^\ell_{2,1}}.
\]
The case $\ell=1$ follows from
\eqref{eq:oct-bes-critical-product}--\eqref{eq:oct-bes-critical-commutator}.
For $\ell\ge2$, Bony's decomposition and
$B^{\ell-1}_{2,1}\hookrightarrow L^\infty$ control the derivatives in
the commutator. The principal transport term cancels by incompressibility.

Local stability at level $\ell+2$ and its continuous smoothing estimate
give
\[
 \sup_t\|U^\nu(t)\|_{B^{\ell+2}_{2,1}}
 +\nu\int_0^T\|U^\nu(t)\|_{B^{\ell+4}_{2,1}}dt\le C.
\]
The equation and the product estimate consequently imply
\[
 \int_0^T\left(\|\partial_tU^\nu\|_{B^\ell_{2,1}}
 +\|\partial_tU^\nu\|_{B^{\ell+1}_{2,1}}
 +\nu\|\partial_tU^\nu\|_{B^{\ell+2}_{2,1}}\right)\,dt\le C.
\]
For example, the highest nonlinear term is bounded by
$C\nu\|U^\nu\|_{B^{\ell+2}_{2,1}}\|U^\nu\|_{B^{\ell+3}_{2,1}}$,
and the highest viscous term by $C\nu^2\|U^\nu\|_{B^{\ell+4}_{2,1}}$.
The weighted smoothing estimate therefore gives
\[
 \tau\sum_n\left(\|\eta_c^{n+1}\|_{B^\ell_{2,1}}
 +\nu\|q^{n+1}\|_{B^{\ell+2}_{2,1}}\right)\le C\tau.
\]
Repeating the block energy calculation with weight $2^{j\ell}$ gives
the $B^\ell_{2,1}$ error and hence, by Bernstein,
$B^\ell_{2,1}\hookrightarrow B^{\ell-1}_{\infty,1}$,
\eqref{eq:oct-higher-time1}.

In the Fourier comparison, the spatial defect is bounded by
$C\|(I-\Pi_N)U\|_{B^{\ell+1}_{2,1}}$ times uniform solution norms.
Lemma~\ref{lem:oct-bes-sharp-projection} bounds this by
$CN^{-(m-\ell-1)}$. The resulting recurrence proves
\eqref{eq:oct-higher-space1}.
Finally, $B^m_{\infty,2}\hookrightarrow H^m$ and
$H^{r+1}\hookrightarrow B^r_{\infty,2}$. Apply
Theorem~\ref{thm:oct-refined-sobolev} with $s=r+1$ to obtain
\eqref{eq:oct-higher-space2}.
\end{proof}

\begin{corollary}[Endpoint inviscid limits]
\label{cor:oct-bes-inviscid}
Let $r\ge0$ be an integer. On the common local interval,
if $u_0\in B^{r+3}_{\infty,1}$, then
\[
 \sup_{0\le t\le T}\|U^\nu(t)-V(t)\|_{B^r_{\infty,1}}\le C\nu.
\]
If $u_0\in B^{r+3}_{\infty,2}$, the same estimate holds in
$B^r_{\infty,2}$. For $0<\tau\le\tau_0$ in the range of
Corollary~\ref{cor:oct-higher}, the time-discrete error relative to Euler
flow is bounded by $C(\tau+\nu)$; the applicable Fourier rate in
Corollary~\ref{cor:oct-higher} may be added independently.
For $r=0$, the assumption $u_0\in B^3_{2,1}$ suffices for the
stronger continuous inviscid bound in $B^1_{2,1}$.
\end{corollary}

\begin{proof}
In the first case put $\ell=r+1$ and use the exact difference equation
\eqref{eq:oct-w-equation}. Dyadic $L^2$ energy, the preceding product
and commutator estimates, and the nonnegative diffusion contribution give
\[
 \frac{d}{dt}\|U^\nu-V\|_{B^\ell_{2,1}}
 \le C_K\|U^\nu-V\|_{B^\ell_{2,1}}
 +C\nu\|V\|_{B^{\ell+2}_{2,1}}.
\]
Both continuous solutions have a uniform local
$B^{\ell+2}_{2,1}$ bound from the common initial datum. Gronwall and
the zero initial difference give $C\nu$; the Bernstein embedding gives
the asserted endpoint estimate. For $r=0$, this argument uses only
$B^3_{2,1}$ initial regularity and gives the stated $B^1_{2,1}$ bound.
In the second case use
Theorem~\ref{thm:oct-refined-sobolev} at $s=r+1$ and the embedding
$H^{r+1}\hookrightarrow B^r_{\infty,2}$.
\end{proof}

\section{Fractional estimates}
\label{oct5:sec-fractional-smoothing}

For initial data in $B^{2+\varepsilon}_{2,1}$, where $0<\varepsilon<1$,
we compare the continuous and discrete solutions through smoothed
initial data. The time interval and step threshold depend only on the
$B^{2+\varepsilon}_{2,1}$ norm, while the first-order bound depends
linearly on the $B^3_{2,1}$ norm of the smoothed datum.
For divergence-free initial data $g$, let $U_g$ and $u_g^n$ solve
\eqref{eq:oct-NS} and \eqref{eq:oct-time-scheme} with $u_0$ replaced by $g$.
For the original datum $u_0$, we retain the notation $U^\nu$ and $u^n$.
Here $0\le\nu\le1$. Discrete maxima are taken over $n\tau\le T_*$,
and sums over $n<M$ satisfy $M\tau\le T_*$. Kinetic energy controls
the low block.

\begin{theorem}[Fractional endpoint time error]
\label{oct5:thm-fractional-time}
Let $0<\varepsilon<1$, put $m=2+\varepsilon$, and let
$u_0\in B^m_{2,1}(\T^2)$ be divergence free.
There are $T_*>0$, $\tau_*>0$, and $C>0$, depending only on
$\varepsilon$ and $\|u_0\|_{B^{m}_{2,1}}$, such that, for
$0\le\nu\le1$, $0<\tau\le\tau_*$, and $M\tau\le T_*$,
\begin{equation}
 \max_{0\le n\le M}
 \|u^n-U^\nu(n\tau)\|_{B^{1}_{2,1}}
 \le C\tau^{(1+\varepsilon)/2}.
 \label{oct5:eq-fractional-B1}
\end{equation}
Consequently,
\begin{equation}
 \max_{0\le n\le M}
 \|u^n-U^\nu(n\tau)\|_{B^0_{\infty,1}}
 \le C\tau^{(1+\varepsilon)/2}.
 \label{oct5:eq-fractional-endpoint}
\end{equation}
The constants and thresholds are independent of viscosity.
\end{theorem}

Since $B^m_{\infty,1}\hookrightarrow B^{m}_{2,1}$, the theorem also
applies to $B^{2+\varepsilon}_{\infty,1}$ data, with constants determined
by that norm.

\subsection{Tame estimates and stability}

\begin{lemma}[Product and commutator estimates]
\label{oct5:lem-tame-dyadic}
For $s>0$, the following tame estimates hold for sufficiently regular
periodic fields:
\begin{align}
 \|fg\|_{B^{s}_{2,1}}
 &\le C_s\bigl(\|f\|_{L^\infty}\|g\|_{B^{s}_{2,1}}
                     +\|g\|_{L^\infty}\|f\|_{B^{s}_{2,1}}\bigr),
 \label{oct5:eq-tame-product}\\
 \sum_j2^{js}\|[\Delta_j,a\cdot\nabla]b\|_{L^2}
 &\le C_s\bigl(\|\nabla a\|_{L^\infty}\|b\|_{B^{s}_{2,1}}
                     +\|\nabla b\|_{L^\infty}\|a\|_{B^{s}_{2,1}}\bigr).
 \label{oct5:eq-tame-commutator}
\end{align}
At the critical low index, one also has
\begin{align}
 \|fg\|_{B^{1}_{2,1}}&\le C\|f\|_{B^{1}_{2,1}}\|g\|_{B^{1}_{2,1}},
 \label{oct5:eq-critical-algebra}\\
 \sum_j2^j\|[\Delta_j,a\cdot\nabla]b\|_{L^2}
 &\le C\|a\|_{B^{2}_{2,1}}\|b\|_{B^{1}_{2,1}},
 \label{oct5:eq-critical-commutator}\\
 \|a\cdot\nabla b\|_{B^{1}_{2,1}}&\le C\|a\|_{B^{1}_{2,1}}\|b\|_{B^{2}_{2,1}}.
 \label{oct5:eq-critical-transport-product}
\end{align}
In particular, $B^{1}_{2,1}\hookrightarrow L^\infty$ and
$B^{2}_{2,1}\hookrightarrow W^{1,\infty}$.
In \eqref{oct5:eq-critical-commutator}, $b$ need not be Lipschitz.
\end{lemma}

\begin{proof}
The bounds \eqref{oct5:eq-tame-product}--\eqref{oct5:eq-tame-commutator}
are the tame paraproduct and commutator estimates of
\cite[Chapter~2]{BCD11}. Bernstein's inequality gives the stated
embeddings, so \eqref{oct5:eq-critical-algebra} follows from
\eqref{oct5:eq-tame-product}. Applying \eqref{oct5:eq-critical-algebra}
to $a$ and $\nabla b$ yields \eqref{oct5:eq-critical-transport-product}.

For \eqref{oct5:eq-critical-commutator}, use the transport commutator
estimate of \cite[Chapters~2 and~3]{BCD11} at regularity index $1$,
with $L^2$ blocks and $\ell^1$ summation, taking
$a\in B^2_{2,1}$ and $b\in B^1_{2,1}$.
The kernel estimates pass to the torus by periodization; the
inhomogeneous low block is bounded directly.
\end{proof}

\begin{lemma}[A common local interval]
\label{oct5:lem-low-stability}
Fix $m=2+\varepsilon\in(2,3)$ and a number $A\ge1$.  All
divergence-free data $g$ satisfying $\|g\|_{B^{m}_{2,1}}\le A$ have a common
interval $[0,T_*]$ and common time-step threshold $\tau_*$ on which
their continuous, pure time-discrete, and sharp Fourier--Galerkin
solutions exist uniquely and obey
\begin{align}
 \sup_{t\le T_*}\|U_g(t)\|_{B^{m}_{2,1}}
 +\max_{n\tau\le T_*}\|u_g^n\|_{B^{m}_{2,1}}
 +\max_{n\tau\le T_*}\|u_{N,g}^n\|_{B^{m}_{2,1}}
 &\le L,
 \label{oct5:eq-low-uniform}\\
 \nu\int_0^{T_*}\|U_g(t)\|_{B^{m+2}_{2,1}}\,dt&\le C_{m,A}.
 \label{oct5:eq-low-smoothing}
\end{align}
Here $u_{N,g}^0=\Pi_Ng$ and the Galerkin update uses $\PP\Pi_N$.
$T_*$, $\tau_*$, $L$, and $C_{m,A}$ depend only on $m,A$, not on
$\nu\in[0,1]$, $N$, or any norm of $g$ higher than $B^{m}_{2,1}$.
One may take $T_*\le c_m/(1+A)$ and $\tau_*\le T_*$; both may be
reduced below by constants depending only on $m,A$.
\end{lemma}

\begin{proof}
Work first with Fourier approximations. The block energy calculation
used in Lemma~\ref{lem:oct-bes-Bm-stability} applies to both discrete
schemes. It removes the projectors, cancels the principal transport
term, and gives dissipation $c\nu2^{2j}$ for $j\ge1$. Kinetic energy
bounds the low block. By \eqref{oct5:eq-tame-commutator} and $m>2$,
\begin{align}
 \frac{d}{dt}\|U_g\|_{B^{m}_{2,1}}+c_m\nu\|U_g\|_{B^{m+2}_{2,1}}
 &\le C_m\|U_g\|_{B^{m}_{2,1}}^2+C_m\nu\|g\|_{L^2},
 \label{oct5:eq-low-continuous-step}\\
 X_{n+1}+c_m\nu\tau\|z^{n+1}\|_{B^{m+2}_{2,1}}
 &\le X_n+C_m\tau X_nX_{n+1}+C_m\nu\tau\|g\|_{L^2},
 \label{oct5:eq-low-discrete-step}
\end{align}
where $X_n=\|z^n\|_{B^{m}_{2,1}}$ and $z^n$ may be either numerical solution.
The differential inequality is understood by regularization at zero
block norms.

For sufficiently small $c_m$, the bootstrap $X_n\le 2A$ makes
$C_m\tau X_n\le1/2$ and gives
\[
 X_{n+1}\le(1+4C_mA\tau)X_n+C'_m\tau A.
\]
Discrete Gronwall on $T_*\le c_m/(1+A)$ improves the bound to
$X_{n+1}<2A$.  The analogous continuous bootstrap and the integral
of \eqref{oct5:eq-low-continuous-step} prove
\eqref{oct5:eq-low-uniform}--\eqref{oct5:eq-low-smoothing}, after
changing the fixed constant $L$.
The initial Galerkin norm is at most $\|g\|_{B^{m}_{2,1}}$, because $\Pi_N$
commutes with every block and is an orthogonal $L^2$ contraction.

To pass to the continuous limit, the equation bounds the Galerkin time
derivatives in $H^{m-2}$. Compactness and interpolation give convergence
along a subsequence in $C([0,T_*];H^{m-\delta})$ for
$0<\delta<m-2$. This space embeds in $B^{1}_{2,1}$ and retains
Lipschitz regularity. Blockwise convergence and Fatou's lemma preserve
the stated estimates. For fixed $\tau$, each Galerkin update is an
invertible linear equation by its $L^2$ energy identity. Compactness
in $H^{m-\delta}$ at the finitely many time levels gives the
time-discrete limit. The $L^2$ difference estimate with Lipschitz
background velocities proves continuous uniqueness; skew-symmetry
proves uniqueness of each discrete update.
\end{proof}

\Needspace{4\baselineskip}
\subsection{Initial-data estimates}

\begin{lemma}[Lipschitz dependence in $B^{1}_{2,1}$]
\label{oct5:lem-data-lipschitz}
If $\|g\|_{B^{m}_{2,1}},\|h\|_{B^{m}_{2,1}}\le A$, then on the common interval of
Lemma~\ref{oct5:lem-low-stability}, after reducing $\tau_*$ using
only $m,A$,
\begin{align}
 \sup_{t\le T_*}\|U_g(t)-U_h(t)\|_{B^{1}_{2,1}}
 &\le C_{m,A}\|g-h\|_{B^{1}_{2,1}},
 \label{oct5:eq-continuous-lipschitz}\\
 \max_{n\tau\le T_*}\|u_g^n-u_h^n\|_{B^{1}_{2,1}}
 &\le C_{m,A}\|g-h\|_{B^{1}_{2,1}}.
 \label{oct5:eq-discrete-lipschitz}
\end{align}
\end{lemma}

\begin{proof}
Let $E=U_g-U_h$.  Subtracting the equations gives
\[
 \partial_tE+\PP(U_g\cdot\nabla E)-\nu\Delta E
 =-\PP(E\cdot\nabla U_h).
\]
Dyadic $L^2$ energy, with the nonnegative diffusion dropped,
\eqref{oct5:eq-critical-commutator}, and
\eqref{oct5:eq-critical-transport-product} give
\[
 \frac{d}{dt}\|E\|_{B^{1}_{2,1}}
 \le C(\|U_g\|_{B^{2}_{2,1}}+\|U_h\|_{B^{2}_{2,1}})\|E\|_{B^{1}_{2,1}}
 \le CL\|E\|_{B^{1}_{2,1}}.
\]
Gronwall proves \eqref{oct5:eq-continuous-lipschitz}.  For
$E^n=u_g^n-u_h^n$ the exact difference equation is
\[
 \frac{E^{n+1}-E^n}{\tau}
 +\PP(u_g^n\cdot\nabla E^{n+1})-\nu\Delta E^{n+1}
 =-\PP(E^n\cdot\nabla u_h^{n+1}).
\]
The same block estimates give
\begin{equation}
 (1-CL\tau)\|E^{n+1}\|_{B^{1}_{2,1}}
 \le(1+CL\tau)\|E^n\|_{B^{1}_{2,1}}.
 \label{oct5:eq-lipschitz-step}
\end{equation}
Taking $CL\tau_*\le1/2$ and iterating proves
\eqref{oct5:eq-discrete-lipschitz}. Approximation and Fatou's lemma
extend the estimates to the stated data.
\end{proof}

\begin{lemma}[Tame first-order estimate]
\label{oct5:lem-tame-first-order}
Suppose $\|g\|_{B^{m}_{2,1}}\le A$ and $g\in B^{3}_{2,1}$, and put
$H=\|g\|_{B^{3}_{2,1}}$.  On the same low-norm interval, after reducing
$\tau_*$ using only $m,A$, one has
\begin{align}
 \sup_{t\le T_*}\|U_g(t)\|_{B^{3}_{2,1}}
 +\max_{n\tau\le T_*}\|u_g^n\|_{B^{3}_{2,1}}
 +\nu\int_0^{T_*}\|U_g(t)\|_{B^{5}_{2,1}}\,dt
 &\le C_{m,A}H,
 \label{oct5:eq-tame-high-bound}\\
 \max_{n\tau\le T_*}\|u_g^n-U_g(n\tau)\|_{B^{1}_{2,1}}
 &\le C_{m,A}\tau(1+H).
 \label{oct5:eq-tame-first-order}
\end{align}
The high norm $H$ enters linearly. One also has
$\nu\tau\sum_{n<M}\|u_g^{n+1}\|_{B^{5}_{2,1}}\le C_{m,A}H$.
\end{lemma}

\begin{proof}
Lemma~\ref{oct5:lem-low-stability} bounds the continuous and discrete
$B^{2}_{2,1}$ norms by $L$, and their Lipschitz norms by $CL$.
Applying \eqref{oct5:eq-tame-commutator} at $s=3$ gives
\[
 \frac{d}{dt}\|U_g\|_{B^{3}_{2,1}}+c\nu\|U_g\|_{B^{5}_{2,1}}
 \le CL\|U_g\|_{B^{3}_{2,1}}+C\nu\|g\|_{L^2}.
\]
For the discrete scheme, apply the two terms in the tame commutator
estimate with $a=u_g^n$ and $b=u_g^{n+1}$:
\begin{align*}
 \|u_g^{n+1}\|_{B^{3}_{2,1}}+c\nu\tau\|u_g^{n+1}\|_{B^{5}_{2,1}}
 \le{}&\|u_g^n\|_{B^{3}_{2,1}}
 +CL\tau\bigl(\|u_g^n\|_{B^{3}_{2,1}}+\|u_g^{n+1}\|_{B^{3}_{2,1}}\bigr)\\
 &+C\nu\tau\|g\|_{L^2}.
\end{align*}
Absorption and Gronwall give \eqref{oct5:eq-tame-high-bound} and the
discrete weighted estimate, since $\|g\|_{L^2}\le CH$. Their
coefficients depend only on $L$, so these bounds hold on the interval
of Lemma~\ref{oct5:lem-low-stability}. Fourier approximation justifies
the calculation for $g\in B^{3}_{2,1}$.

Write $U=U_g$ and $T=M\tau\le T_*$.  The equation, the low bounds,
and the product estimates give
\begin{align}
 \|\partial_tU\|_{B^{1}_{2,1}}
 &\le CL^2+C\nu\|U\|_{B^{3}_{2,1}},
 \label{oct5:eq-tame-ut1}\\
 \|\partial_tU\|_{B^{2}_{2,1}}
 &\le CL\|U\|_{B^{3}_{2,1}}+C\nu\|U\|_{B^{4}_{2,1}},
 \label{oct5:eq-tame-ut2}\\
 \nu\|\partial_tU\|_{B^{3}_{2,1}}
 &\le CL\nu\bigl(\|U\|_{B^{3}_{2,1}}+\|U\|_{B^{4}_{2,1}}\bigr)
           +C\nu^2\|U\|_{B^{5}_{2,1}}.
 \label{oct5:eq-tame-nuut3}
\end{align}
The highest nonlinear term satisfies
\[
 \|U\cdot\nabla U\|_{B^{3}_{2,1}}
 \le C\bigl(\|U\|_{L^\infty}\|\nabla U\|_{B^{3}_{2,1}}
              +\|\nabla U\|_{L^\infty}\|U\|_{B^{3}_{2,1}}\bigr)
 \le CL(\|U\|_{B^{4}_{2,1}}+\|U\|_{B^{3}_{2,1}}).
\]
Since $0\le\nu\le1$ and
$\|U\|_{B^{4}_{2,1}}\le C(\|U\|_{B^{3}_{2,1}}+\|U\|_{B^{5}_{2,1}})$, the high bound implies
\begin{equation}
 \int_0^T\bigl(\|\partial_tU\|_{B^{1}_{2,1}}+\|\partial_tU\|_{B^{2}_{2,1}}\bigr)\,dt
 +\nu\int_0^T\|\partial_tU\|_{B^{3}_{2,1}}\,dt
 \le C_{m,A}(1+H).
 \label{oct5:eq-tame-time-integrals}
\end{equation}

Use the residual decomposition \eqref{eq:oct-bes-defect-splitting}
with $U=U_g$. The identity
\[
 U^n\cdot\nabla U^{n+1}-U(t)\cdot\nabla U(t)
 =(U^n-U(t))\cdot\nabla U^{n+1}
   +U(t)\cdot\nabla(U^{n+1}-U(t))
\]
and the critical algebra estimate bound the convective difference by
\[
 CL\left(\int_{t_n}^t\|\partial_tU(s)\|_{B^{1}_{2,1}}\,ds
            +\int_t^{t_{n+1}}\|\partial_tU(s)\|_{B^{2}_{2,1}}\,ds\right).
\]
Fubini's theorem and \eqref{oct5:eq-tame-time-integrals} give
\begin{equation}
 \tau\sum_{n<M}\|\eta_c^{n+1}\|_{B^{1}_{2,1}}
 +\nu\tau\sum_{n<M}\|q^{n+1}\|_{B^{3}_{2,1}}
 \le C_{m,A}\tau(1+H).
 \label{oct5:eq-tame-consistency}
\end{equation}
For the second sum use
$\nu\tau\sum_n\|q^{n+1}\|_{B^{3}_{2,1}}
 \le\nu\tau\int_0^T\|\partial_tU\|_{B^{3}_{2,1}}\,dt$.

Finally $e^n=u_g^n-U^n$ satisfies
\[
 \frac{e^{n+1}-e^n}{\tau}
 +\PP(u_g^n\cdot\nabla e^{n+1})-\nu\Delta e^{n+1}
 =-\PP(e^n\cdot\nabla U^{n+1})
     -\eta_c^{n+1}-\nu\Delta q^{n+1}.
\]
Dyadic $L^2$ energy and the critical estimates yield
\begin{align*}
 (1-CL\tau)\|e^{n+1}\|_{B^{1}_{2,1}}
 \le{}&(1+CL\tau)\|e^n\|_{B^{1}_{2,1}}
       +\tau\|\eta_c^{n+1}\|_{B^{1}_{2,1}}
       +C\nu\tau\|q^{n+1}\|_{B^{3}_{2,1}}.
\end{align*}
Here the residual $\nu\Delta q^{n+1}$ is estimated in $B^{1}_{2,1}$.
Since $e^0=0$, absorption, discrete Gronwall, and
\eqref{oct5:eq-tame-consistency} prove \eqref{oct5:eq-tame-first-order}.
\end{proof}

\subsection{Smoothing argument}

\begin{proof}[Proof of Theorem~\ref{oct5:thm-fractional-time}]
For $R\ge1$, use the smooth cutoff of Section~\ref{sec:m2an-notation} to set
\[
 u_{0,R}=\chi(D/R)u_0.
\]
This datum is smooth, divergence free, and has the same
mean as $u_0$. Let $U_R^\nu$ and $u_R^n$ be its continuous and
time-discrete solutions. Blockwise $L^2$ contraction and the frequency supports give
\begin{align}
 \|u_{0,R}\|_{B^{m}_{2,1}}&\le C\|u_0\|_{B^{m}_{2,1}},
 \label{oct5:eq-smooth-low}\\
 \|u_0-u_{0,R}\|_{B^{1}_{2,1}}&\le C R^{-(m-1)}\|u_0\|_{B^{m}_{2,1}},
 \label{oct5:eq-smooth-bias}\\
 \|u_{0,R}\|_{B^{3}_{2,1}}&\le C R^{3-m}\|u_0\|_{B^{m}_{2,1}}.
 \label{oct5:eq-smooth-high}
\end{align}
Indeed, the bias contains only blocks with $2^j\ge cR$, while the
smoothed datum contains only blocks with $2^j\le CR$. On these sets,
$2^j\le CR^{1-m}2^{jm}$ and $2^{3j}\le CR^{3-m}2^{jm}$, respectively.

Choose $A\ge1$ bounding the two low norms in
\eqref{oct5:eq-smooth-low}, uniformly for every $R\ge1$.
Lemmas~\ref{oct5:lem-low-stability}--\ref{oct5:lem-tame-first-order}
then apply on a common interval with a common step threshold.
The triangle inequality, the two initial-data comparisons, and the tame
first-order estimate give
\begin{align*}
 \max_{n\tau\le T_*}\|u^n-U^\nu(t_n)\|_{B^{1}_{2,1}}
 \le{}&\max_n\|u^n-u_R^n\|_{B^{1}_{2,1}}
       +\max_n\|u_R^n-U_R^\nu(t_n)\|_{B^{1}_{2,1}}\\
 &+\sup_{t\le T_*}\|U_R^\nu(t)-U^\nu(t)\|_{B^{1}_{2,1}}\\
 \le{}& C_{m,A}\bigl(R^{-(m-1)}+\tau(1+R^{3-m})\bigr).
\end{align*}
Reduce $\tau_*$ further to be at most one and choose
$R=\tau^{-1/2}$.  Since $m=2+\varepsilon$,
\[
 R^{-(m-1)}=\tau^{(1+\varepsilon)/2},\qquad
 \tau R^{3-m}=\tau^{(1+\varepsilon)/2},\qquad
 \tau\le\tau^{(1+\varepsilon)/2}.
\]
This proves \eqref{oct5:eq-fractional-B1}.  Bernstein gives
\[
 \|f\|_{B^0_{\infty,1}}
 \le C\sum_j2^j\|\Delta_jf\|_{L^2}=C\|f\|_{B^{1}_{2,1}},
\]
which proves \eqref{oct5:eq-fractional-endpoint}.
\end{proof}

\subsection{Spatial and viscosity errors}

\begin{corollary}[Fourier approximation for $B^{2+\varepsilon}_{2,1}$ data]
\label{oct5:cor-fractional-sharp}
Under the assumptions of Theorem~\ref{oct5:thm-fractional-time}, let
$u_N^n$ solve \eqref{eq:oct-scheme} with $N\ge1$.
After a reduction of the same low-norm $\tau_*$ if needed,
\begin{equation}
 \max_{n\tau\le T_*}
 \|u_N^n-U^\nu(t_n)\|_{B^0_{\infty,1}}
 \le C\bigl(\tau^{(1+\varepsilon)/2}+N^{-\varepsilon}\bigr).
 \label{oct5:eq-fractional-sharp-rate}
\end{equation}
The same bound holds with the error measured in $B^{1}_{2,1}$.
The constants and thresholds are independent of $N$ and $\nu$.
\end{corollary}

\begin{proof}
Compare the Galerkin and time-discrete solutions for the same datum
$u_0$. Set
\[
 \rho^n=(\Pi_N-I)u^n,\qquad
 \theta^n=u_N^n-\Pi_Nu^n,
 \qquad u_N^n-u^n=\theta^n+\rho^n.
\]
Lemma~\ref{lem:oct-bes-sharp-projection} gives, for $s<m$,
\begin{equation}
 \|(I-\Pi_N)f\|_{B^{s}_{2,1}}\le C N^{s-m}\|f\|_{B^{m}_{2,1}}.
 \label{oct5:eq-sharp-low-tail}
\end{equation}
Together with local stability, this yields
\[
 \sup_n\|\rho^n\|_{B^{1}_{2,1}}\le CLN^{-(1+\varepsilon)},\qquad
 \sup_n\|\rho^n\|_{B^{2}_{2,1}}\le CLN^{-\varepsilon}.
\]
Subtract the projection of \eqref{eq:oct-time-scheme} from
\eqref{eq:oct-scheme}.  Since the projection commutes with
diffusion, the exact equation is
\begin{align*}
 \frac{\theta^{n+1}-\theta^n}{\tau}
 +\PP\Pi_N(u_N^n\cdot\nabla\theta^{n+1})
 -\nu\Delta\theta^{n+1}
 ={}&-\PP\Pi_N\bigl((\theta^n+\rho^n)
                         \cdot\nabla u^{n+1}\bigr)\\
 &-\PP\Pi_N(u_N^n\cdot\nabla\rho^{n+1}).
\end{align*}
Since $\Delta_j\theta^{n+1}\in X_N$, the projectors drop out of
the block energy test. Transport cancellation, nonnegative diffusion,
and the critical product and commutator estimates give
\begin{align*}
 (1-CL\tau)\|\theta^{n+1}\|_{B^{1}_{2,1}}
 \le{}&(1+CL\tau)\|\theta^n\|_{B^{1}_{2,1}}\\
 &+CL\tau\bigl(\|\rho^n\|_{B^{1}_{2,1}}+\|\rho^{n+1}\|_{B^{2}_{2,1}}\bigr).
\end{align*}
Since $\theta^0=0$, absorption and discrete Gronwall give
\[
 \max_n\|u_N^n-u^n\|_{B^{1}_{2,1}}
 \le C_{m,A}N^{-\varepsilon}.
\]
Combining this with \eqref{oct5:eq-fractional-B1} gives the claimed
bound in $B^{1}_{2,1}$; the endpoint embedding then proves
\eqref{oct5:eq-fractional-sharp-rate}.
\end{proof}

\begin{corollary}[Inviscid comparison for $B^{2+\varepsilon}_{2,1}$ data]
\label{oct5:cor-fractional-inviscid}
Let $V$ be the Euler velocity with the same datum as in
Theorem~\ref{oct5:thm-fractional-time}. On the same
low-norm interval and with $\alpha=(1+\varepsilon)/2$,
\begin{equation}
 \sup_{t\le T_*}\|U^\nu(t)-V(t)\|_{B^{1}_{2,1}}
 \le C\nu^\alpha,
 \qquad 0\le\nu\le1.
 \label{oct5:eq-fractional-inviscid}
\end{equation}
After reducing $\tau_*$ to the minimum of the thresholds in
Theorem~\ref{oct5:thm-fractional-time} and
Corollary~\ref{oct5:cor-fractional-sharp}, one has, for
$0<\tau\le\tau_*$ and $N\ge1$,
\begin{align}
 \max_{n\tau\le T_*}\|u^n-V(t_n)\|_{B^0_{\infty,1}}
 &\le C(\tau^\alpha+\nu^\alpha),
 \label{oct5:eq-fractional-time-Euler}\\
 \max_{n\tau\le T_*}\|u_N^n-V(t_n)\|_{B^0_{\infty,1}}
 &\le C(\tau^\alpha+N^{-\varepsilon}+\nu^\alpha).
 \label{oct5:eq-fractional-sharp-Euler}
\end{align}
Both numerical bounds also hold with the errors measured in $B^{1}_{2,1}$.
The constants depend only on $\varepsilon$ and
$\|u_0\|_{B^{2+\varepsilon}_{2,1}}$.
\end{corollary}

\begin{proof}
For $\nu>0$ use the same smoothed datum $u_{0,R}$ as above and
denote its viscous and Euler velocities by $U_R^\nu$ and $V_R$.
Their difference $D_R=U_R^\nu-V_R$ satisfies
\[
 \partial_tD_R+\PP(U_R^\nu\cdot\nabla D_R)-\nu\Delta D_R
 =-\PP(D_R\cdot\nabla V_R)+\nu\Delta V_R,
 \qquad D_R(0)=0.
\]
Both background $B^{2}_{2,1}$ norms are controlled on the common
low-norm interval.  The critical product and commutator estimates,
dyadic $L^2$ energy, and nonnegative diffusion give
\[
 \frac{d}{dt}\|D_R\|_{B^{1}_{2,1}}
 \le CL\|D_R\|_{B^{1}_{2,1}}+C\nu\|V_R\|_{B^{3}_{2,1}}.
\]
The tame estimate at viscosity zero gives
$\sup_t\|V_R(t)\|_{B^{3}_{2,1}}\le C_{m,A}\|u_{0,R}\|_{B^{3}_{2,1}}$.
Gronwall thus gives
\[
 \sup_{t\le T_*}\|D_R(t)\|_{B^{1}_{2,1}}
 \le C_{m,A}\nu\|u_{0,R}\|_{B^{3}_{2,1}}.
\]
Apply the continuous initial-data estimate to $u_0$ and $u_{0,R}$,
first at viscosity $\nu$ and then at viscosity zero.
Equations~\eqref{oct5:eq-smooth-bias} and
\eqref{oct5:eq-smooth-high} imply
\[
 \sup_{t\le T_*}\|U^\nu(t)-V(t)\|_{B^{1}_{2,1}}
 \le C_{m,A}\bigl(R^{-(1+\varepsilon)}
                         +\nu R^{1-\varepsilon}\bigr).
\]
Taking $R=\nu^{-1/2}\ge1$ proves
\eqref{oct5:eq-fractional-inviscid}.  For $\nu=0$ the two
continuous velocities agree by uniqueness and the difference is zero.
The triangle inequality with
Theorem~\ref{oct5:thm-fractional-time} and
Corollary~\ref{oct5:cor-fractional-sharp} gives both numerical bounds
in $B^{1}_{2,1}$; the endpoint embedding gives the displayed estimates.
\end{proof}

The time and viscosity estimates above have the same exponent
$(1+\varepsilon)/2$ in $B^1_{2,1}$. The next proposition shows that this
exponent cannot be increased in Theorem~\ref{oct5:thm-fractional-time}
or in the $B^1_{2,1}$ estimate of
Corollary~\ref{oct5:cor-fractional-inviscid}: a single shear datum,
independent of $\tau$ and $\nu$, yields matching algebraic lower bounds
up to a logarithmic factor. These lower bounds concern $B^1_{2,1}$;
the embedding into $B^0_{\infty,1}$ does not provide lower bounds
in the endpoint norm.

\begin{proposition}[A fixed shear with fractional $B^{1}_{2,1}$ error]
\label{oct5:prop-fixed-shear-calibration}
Fix $0<\varepsilon<1$, $a>0$, and a nonzero real number $c$, and set
\begin{gather*}
 m=2+\varepsilon,\qquad \alpha=\frac{1+\varepsilon}{2},\qquad
 u_0(x)=\bigl(c,h(x_1)\bigr),\\
 h(x_1)=\sum_{j\ge1}A_j\cos(2^jx_1),\quad
 A_j=\frac{a\,2^{-mj}}{(j+1)^2}.
\end{gather*}
Then $u_0$ is divergence free and belongs to
$B^m_{\infty,1}\setminus B^3_{\infty,1}$.
Choose a positive dyadic rational $T$ within the common local interval,
and put $\tau_j=\nu_j=2^{-2j}$ and $M_j=T/\tau_j$; for all sufficiently
large $j$, $M_j$ is an integer.  Let $u^{n,0}$ denote the pure
time-discrete velocity at viscosity zero.  There are constants
$0<C_0\le C_1$, depending only on $\varepsilon,a,c,T$, such that,
for all sufficiently large $j$,
\begin{align}
 C_0\frac{\tau_j^\alpha}{(j+1)^2}
 &\le \|u^{M_j,0}-V(T)\|_{B^{1}_{2,1}}
 \le C_1\frac{\tau_j^\alpha}{(j+1)^2},
 \label{oct5:eq-fixed-shear-time-lower}\\
 C_0\frac{\nu_j^\alpha}{(j+1)^2}
 &\le \|U^{\nu_j}(T)-V(T)\|_{B^{1}_{2,1}}
 \le C_1\frac{\nu_j^\alpha}{(j+1)^2}.
 \label{oct5:eq-fixed-shear-viscosity-lower}
\end{align}
Thus neither $B^{1}_{2,1}$ estimate admits a larger algebraic
exponent uniformly over this class. The endpoint norm
$B^0_{\infty,1}$ is considered separately below.
\end{proposition}

\begin{proof}
With the cutoffs of Section~\ref{sec:m2an-notation}, frequency $2^j$
lies in block $j$ alone because $\chi(1)=1$ and $\chi(2)=0$.
The $B^m_{\infty,1}$ membership follows from
$\sum_j2^{mj}A_j<\infty$, whereas
$\sum_j2^{3j}A_j=a\sum_j2^{(1-\varepsilon)j}/(j+1)^2=\infty$
excludes membership in $B^3_{\infty,1}$.
Both continuous equations and the numerical step preserve the shear
form. The vertical component solves
$h_t+c\partial_1h=\nu\partial_1^2h$, on which the Leray projector is
the identity. The continuous and one-step discrete multipliers at
frequency $k$ are
\[
 e^{-t(\nu k^2+ick)},\qquad
 \bigl(1+\tau(\nu k^2+ick)\bigr)^{-1},
\]
respectively.  With $k_j=2^j$, $\nu=0$, and $\tau=\tau_j=k_j^{-2}$,
the modulus of the numerical multiplier at time $T$ is
\[
 \left|\bigl(1+ic\tau_jk_j\bigr)^{-M_j}\right|
 =\bigl(1+c^2\tau_j\bigr)^{-T/(2\tau_j)}
 \longrightarrow e^{-Tc^2/2}<1.
\]
Since the exact multiplier has modulus one, their difference is bounded
below by a positive constant for large $j$. Block $j$ therefore gives
the lower bound $C_0k_jA_j=C_0a\tau_j^\alpha/(j+1)^2$ in
$B^{1}_{2,1}$; the normalized $L^2$ norm of a cosine is
$1/\sqrt2$. For the inviscid comparison with $\nu_j=k_j^{-2}$,
the multiplier difference at this frequency has modulus $1-e^{-T}$,
which proves the second lower bound.

For the upper bounds, telescope the $M$ one-step factors and use
$|(1+ix)^{-1}-e^{-ix}|\le Cx^2$. When $M\tau=T$, this gives
\[
 \left|(1+ic\tau k)^{-M}-e^{-ickT}\right|
 \le C_{T,c}\min\{1,\tau k^2\}.
\]
Likewise the continuous multiplier difference is bounded by
$C_T\min\{1,\nu k^2\}$.  For either parameter $q\in\{\tau,\nu\}$,
the corresponding $B^{1}_{2,1}$ error is therefore bounded by
\[
 C\sum_{\ell\ge1} k_\ell A_\ell\min\{1,q k_\ell^2\}.
\]
For $0<q\le1/4$, choose
$J=\lfloor\tfrac12\log_2(1/q)\rfloor\ge1$ and split this sum at $J$.
Since $1-\varepsilon>0$ and $1+\varepsilon>0$, the two geometric
tails with their polynomial weights satisfy
\begin{align*}
 \sum_{\ell\ge1} k_\ell A_\ell\min\{1,q k_\ell^2\}
 &\le a q\sum_{\ell=1}^J
             \frac{2^{(1-\varepsilon)\ell}}{(\ell+1)^2}
       +a\sum_{\ell>J}
             \frac{2^{-(1+\varepsilon)\ell}}{(\ell+1)^2}\\
 &\le \frac{C_{\varepsilon,a}}{(J+1)^2}
          \left(q2^{(1-\varepsilon)J}
                    +2^{-(1+\varepsilon)J}\right)\\
 &\le C_{\varepsilon,a}
            \frac{q^\alpha}{(1+|\log q|)^2}.
\end{align*}
For the growing sum, split at $J/2$: the denominators above that index
are comparable to $(J+1)^2$, and the exponential decay below it absorbs
the polynomial factor. The decreasing tail is bounded by its first
term. Taking $q=\tau_j$ and $q=\nu_j$ proves the upper bounds.

For $\beta>\alpha$, the ratio
$2^{2j(\beta-\alpha)}/(j+1)^2$ diverges, proving the assertion about
larger algebraic exponents.

In the endpoint norm the frequency weight is absent. Since
$\sum_jk_j^2A_j<\infty$, the same multiplier bounds give
\[
 \|u^{M,0}-V(T)\|_{B^0_{\infty,1}}
 \le C_{T,c}\tau\sum_j k_j^2A_j,
 \qquad
 \|U^\nu(T)-V(T)\|_{B^0_{\infty,1}}
 \le T\nu\sum_j k_j^2A_j.
\]
Here $M\tau=T$. The first estimate follows from telescoping and the
second from $1-e^{-y}\le y$. These endpoint bounds are first order.
\end{proof}

For the same shear, the preceding proof gives first-order upper bounds
in $B^0_{\infty,1}$. The next proposition identifies their positive
leading constants, showing that both endpoint errors are asymptotically
first order. Its exact formulas are also used in the numerical
comparisons of Section~\ref{m2an:sec-shear-comparison}.

\Needspace{10\baselineskip}
\begin{proposition}[First-order endpoint asymptotics for the fixed shear]
\label{prop:m2an-shear-endpoint-first-order}
For the datum of Proposition~\ref{oct5:prop-fixed-shear-calibration}, write
$k_j=2^j$ and $A_j=a2^{-(2+\varepsilon)j}/(j+1)^2$, and fix the same
positive time $T$.  Let $V$ and $U^\nu$ be the continuous Euler and
Navier--Stokes velocities, respectively, and let $u^{n,0}$ be the
inviscid time-discrete velocity.  For $h>0$ with $M=T/h\in\mathbb N$,
define
\begin{align}
 E_0^{\rm time}(h)
 &=\|u^{M,0}-V(T)\|_{B^0_{\infty,1}}
 =\sum_{j\ge1}A_j
    \left|(1+ihck_j)^{-T/h}-e^{-iTck_j}\right|,
 \label{m2an:eq-shear-endpoint-time}\\
 E_0^{\rm visc}(h)
 &=\|U^h(T)-V(T)\|_{B^0_{\infty,1}}
 =\sum_{j\ge1}A_j(1-e^{-hTk_j^2}).
 \label{m2an:eq-shear-endpoint-visc}
\end{align}
The viscosity formula is also defined for every $h>0$.  Then
\begin{equation}\label{m2an:eq-shear-endpoint-limits}
 \begin{aligned}
 \lim_{\substack{h\downarrow0\\T/h\in\mathbb N}}
       \frac{E_0^{\rm time}(h)}h
 &=\frac{Tc^2}{2}\sum_{j\ge1}k_j^2A_j=:L_0^{\rm time}>0,\\
 \lim_{h\downarrow0}\frac{E_0^{\rm visc}(h)}h
 &=T\sum_{j\ge1}k_j^2A_j=:L_0^{\rm visc}>0.
 \end{aligned}
\end{equation}
Both endpoint errors are asymptotically first order for this datum.
These estimates do not determine the optimal uniform endpoint rate
for the full $B^{2+\varepsilon}_{\infty,1}$ class.
\end{proposition}

\begin{proof}
The multiplier formulas from the preceding proof and the fact that
$\cos(k_jx_1)$ lies in block $j$ give the exact sums
\eqref{m2an:eq-shear-endpoint-time}--\eqref{m2an:eq-shear-endpoint-visc}.
A real cosine with complex amplitude $z$ has $L^\infty$ norm $|z|$.
For each fixed $j$,
\[
 -\frac Th\log(1+ihck_j)+iTck_j
 =-\frac{Tc^2k_j^2}{2}h+O_j(h^2),
\]
and hence the time multiplier difference divided by $h$ tends in
modulus to $Tc^2k_j^2/2$.  Similarly,
$(1-e^{-hTk_j^2})/h\to Tk_j^2$.
The time and viscosity multiplier differences are bounded by
$C_{T,c}\min\{1,hk_j^2\}$ and $\min\{1,hTk_j^2\}$, respectively.
After division by $h$, they are dominated by $C_{T,c}k_j^2$. Since
\[
 \sum_{j\ge1}k_j^2A_j
 =a\sum_{j\ge1}\frac{2^{-\varepsilon j}}{(j+1)^2}<\infty
 \qquad(\varepsilon>0),
\]
dominated convergence for the two nonnegative series proves
\eqref{m2an:eq-shear-endpoint-limits}.  The sums are strictly positive
because $a,T>0$ and $c\ne0$.
\end{proof}

\begin{remark}[Inexact solution of the discrete equations]
\label{oct5:rem-fractional-scope}
The preceding estimates concern the exact discrete updates.
Corollary~\ref{cor:m2an-implemented-endpoint} extends them to computed
iterates by controlling the accumulated residual of the discrete equations.
\end{remark}

\section{Iteration and residuals}\label{rev:sec-picard}

Each Fourier step is a finite-dimensional linear problem. Write
$Q_N=\Pi_N\PP$ and work in the divergence-free Fourier space $X_N$
defined above.

\subsection{Picard iteration}

For fixed $u_N^n$, set $u^{(0)}=u_N^n$ and solve
\begin{equation}\label{rev:picard}
 \frac{u^{(m+1)}-u_N^n}{\tau}
 +Q_N(u_N^n\cdot\nabla u^{(m)})
 =\nu\Delta u^{(m+1)}.
\end{equation}
The diffusion solve is diagonal in Fourier space. With sufficient zero
padding, FFTs evaluate products without aliasing in the retained
coefficients. Since the nonlinear term has zero mean, every Picard
iterate satisfies $\widehat u^{(m)}(0)=\widehat u_N^n(0)$.

\begin{theorem}[Picard convergence]\label{oct-solver-thm-picard}
Suppose $\|u_N^n\|_\infty\le K$, with $K>0$.
For an absolute constant $C$, consecutive iterates satisfy, for $m\ge1$,
\begin{equation}\label{rev:picard-two-rates}
 \|u^{(m+1)}-u^{(m)}\|_2
 \le q\|u^{(m)}-u^{(m-1)}\|_2,
\end{equation}
where one may take
\begin{equation}\label{rev:picard-q}
 q=\begin{cases}
 \min\{C\tau NK,(C\tau K^2/\nu)^{1/2}\},&\nu>0,\\
 C\tau NK,&\nu=0.
 \end{cases}
\end{equation}
If $q<1$, the iterates converge to the unique Galerkin update
$u_N^{n+1}$. In particular, either of the conditions
\begin{equation}\label{rev:picard-disjunction}
 \tau<c(NK)^{-1},\qquad
 \text{or}\qquad \nu>0\ \text{and}\ \tau<c\nu K^{-2}
\end{equation}
is sufficient. A stopping rule
$\|u^{(m+1)}-u^{(m)}\|_2\le\varepsilon_{\rm it}$ gives
\begin{equation}\label{rev:picard-aposteriori}
 \|u^{(m+1)}-u_N^{n+1}\|_2
 \le\frac{\varepsilon_{\rm it}}{1-q}.
\end{equation}
If $K=0$, the initial guess is already the solution.
\end{theorem}

\begin{proof}
Set $w^{(m+1)}=u^{(m+1)}-u^{(m)}$. Subtracting successive
equations in \eqref{rev:picard} gives
\begin{equation}\label{rev:picard-difference}
 \tau^{-1}w^{(m+1)}+Q_N(u_N^n\cdot\nabla w^{(m)})
 =\nu\Delta w^{(m+1)}.
\end{equation}
The $L^2$ test and the inverse inequality on $X_N$ yield
\[
 \tau^{-1}\|w^{(m+1)}\|_2^2
 \le CKN\|w^{(m)}\|_2\|w^{(m+1)}\|_2.
\]
For $\nu>0$, integration by parts instead gives
\begin{align*}
 \tau^{-1}\|w^{(m+1)}\|_2^2
 +\nu\|\nabla w^{(m+1)}\|_2^2
 &\le K\|w^{(m)}\|_2\|\nabla w^{(m+1)}\|_2\\
 &\le\frac\nu2\|\nabla w^{(m+1)}\|_2^2
 +\frac{CK^2}{\nu}\|w^{(m)}\|_2^2.
\end{align*}
The two estimates give the contraction bounds. The fixed point solves
\eqref{eq:oct-scheme}; uniqueness follows by the $L^2$ test
of the difference of two solutions. Summing the geometric tail
of the iteration proves \eqref{rev:picard-aposteriori}.
\end{proof}

\subsection{Residual estimates}

Let $u_N^n$ be the exact Galerkin trajectory and let
$\widetilde u_N^n\in X_N$ be computed states with
$\widetilde u_N^0=u_N^0$. Define the residual by
\begin{equation}\label{oct-solver-pde-residual}
 r^{n+1}:=\frac{\widetilde u_N^{n+1}-\widetilde u_N^n}{\tau}
 +Q_N(\widetilde u_N^n\cdot\nabla\widetilde u_N^{n+1})
 -\nu\Delta\widetilde u_N^{n+1}.
\end{equation}
This normalization includes division of the time-step equation by $\tau$.

\begin{proposition}[Accumulated residual error]
\label{oct-solver-prop-pde-residual}
Assume $N\ge1$, $M\tau\le T$, and
\begin{equation}\label{oct-solver-residual-tolerance}
 \max_{0\le n<M}\|r^{n+1}\|_2\le\delta,
 \qquad
 \max_{0\le n<M}\|\nabla u_N^{n+1}\|_\infty\le K_*.
\end{equation}
Then, for every $\tau>0$,
\begin{equation}\label{oct-solver-global-L2-residual}
 \max_{0\le n\le M}\|\widetilde u_N^n-u_N^n\|_2
 \le C(T,K_*)\left(\tau\sum_{n=0}^{M-1}
                    \|r^{n+1}\|_2^2\right)^{1/2}
 \le C(T,K_*)\delta.
\end{equation}
Consequently,
\begin{equation}\label{oct-solver-global-endpoint-residual}
 \max_{0\le n\le M}
 \left(\|\widetilde u_N^n-u_N^n\|_{B^1_{2,1}}
 +\|\widetilde u_N^n-u_N^n\|_{B^0_{\infty,1}}\right)
 \le C(T,K_*)N\delta.
\end{equation}
The constants are independent of $N$ and $\nu\ge0$.
Membership in $X_N$ is the only regularity assumption on the
computed trajectory.
\end{proposition}

\begin{proof}
For divergence-free $a\in X_N$, define $\mathcal L_a:X_N\to X_N$ by
$\mathcal L_a z=z+\tau Q_N(a\cdot\nabla z)-\nu\tau\Delta z$.
For $z\in X_N$, it satisfies
\[
 (\mathcal L_a z,z)=\|z\|_2^2+\nu\tau\|\nabla z\|_2^2.
\]
Thus
\begin{equation}\label{oct-solver-resolvent-bound}
 \|\mathcal L_a^{-1}\|_{L^2\to L^2}\le1.
\end{equation}
For the accumulated error $e^n=\widetilde u_N^n-u_N^n$,
subtraction of the exact steps gives
\[
 \frac{e^{n+1}-e^n}{\tau}
 +Q_N(\widetilde u_N^n\cdot\nabla e^{n+1})
 -\nu\Delta e^{n+1}
 =r^{n+1}-Q_N(e^n\cdot\nabla u_N^{n+1}).
\]
Applying \eqref{oct-solver-resolvent-bound} with
$a=\widetilde u_N^n$ yields
\[
 \|e^{n+1}\|_2\le(1+\tau K_*)\|e^n\|_2
                         +\tau\|r^{n+1}\|_2.
\]
Iteration from $e^0=0$ gives
\[
 \|e^n\|_2\le e^{K_*T}\tau\sum_{j=0}^{n-1}\|r^{j+1}\|_2
 \le e^{K_*T}\sqrt T
       \left(\tau\sum_{j=0}^{M-1}\|r^{j+1}\|_2^2\right)^{1/2},
\]
which proves \eqref{oct-solver-global-L2-residual}.
For $s>0$ and $z\in X_N$, finite overlap of the dyadic supports
and Cauchy--Schwarz imply
\begin{equation}\label{oct-solver-inverse-Besov}
 \|z\|_{B^s_{2,1}}
 \le C_s\left(\sum_{0\le j\le J_N}2^{2js}\right)^{1/2}
          \left(\sum_{0\le j\le J_N}\|\Delta_jz\|_2^2\right)^{1/2}
 \le C_sN^s\|z\|_2,
 \qquad 2^{J_N}\le CN.
\end{equation}
Taking $s=1$ and using
$B^1_{2,1}\hookrightarrow B^0_{\infty,1}$ proves
\eqref{oct-solver-global-endpoint-residual}.
\end{proof}

\begin{corollary}[Computed endpoint approximation]
\label{cor:m2an-implemented-endpoint}
Let $m=2+\varepsilon$, $0<\varepsilon<1$, and
$\alpha=(1+\varepsilon)/2$. Assume that $u_0\in B^m_{2,1}$
is divergence free and fix $A\ge\max\{1,\|u_0\|_{B^m_{2,1}}\}$.
On the common interval $0<T\le T_*(m,A)$ of
Theorem~\ref{oct5:thm-fractional-time}, let
$\widetilde u_N^n\in X_N$, $0\le n\le M$, start from $\Pi_Nu_0$,
where $M\tau\le T$, and satisfy
$\max_{0\le n<M}\|r^{n+1}\|_2\le\delta$.
There is a threshold $\tau_{\rm comp}>0$, depending only on
$m,A,T$, such that for $0<\tau\le\tau_{\rm comp}$,
$N\ge1$, and $0\le\nu\le1$,
\begin{align}
 \max_{0\le n\le M}\|\widetilde u_N^n-U^\nu(t_n)\|_{B^0_{\infty,1}}
 &\le C_{m,A,T}(\tau^\alpha+N^{-\varepsilon}+N\delta),
 \label{oct-solver-computed-NS}\\
 \max_{0\le n\le M}\|\widetilde u_N^n-V(t_n)\|_{B^0_{\infty,1}}
 &\le C_{m,A,T}(\tau^\alpha+N^{-\varepsilon}+\nu^\alpha+N\delta).
 \label{oct-solver-computed-Euler}
\end{align}
Both estimates also hold in $B^1_{2,1}$.
If in addition $N^2\delta\le B$, with $B$ fixed, then
\begin{equation}\label{oct-solver-computed-B2}
 \max_{0\le n\le M}\|\widetilde u_N^n\|_{B^2_{2,1}}
 \le C_{m,A,T}(1+B).
\end{equation}
\end{corollary}

\begin{proof}
Lemma~\ref{oct5:lem-low-stability} bounds the exact Galerkin
trajectory in $B^m_{2,1}$ by a constant depending only on $m,A$.
Since $m>2$, its Lipschitz norm is bounded by the same quantities.
Proposition~\ref{oct-solver-prop-pde-residual} therefore applies
on the same time-step range as the exact Galerkin estimates.
Adding \eqref{oct-solver-global-endpoint-residual} to
Corollaries~\ref{oct5:cor-fractional-sharp}
and~\ref{oct5:cor-fractional-inviscid} proves the error bounds.
For the regularity bound, use \eqref{oct-solver-inverse-Besov}
with $s=2$:
\[
 \|\widetilde u_N^n\|_{B^2_{2,1}}
 \le\|u_N^n\|_{B^2_{2,1}}+CN^2\|e^n\|_2
 \le C_{m,A,T}(1+N^2\delta).
\]
\end{proof}

\begin{remark}[Choice of stopping tolerance]
\label{oct-solver-rem-residual-normalization}
An algebraic residual
$R^{n+1}=\mathcal L_{\widetilde u_N^n}\widetilde u_N^{n+1}
-\widetilde u_N^n$ satisfies $R^{n+1}=\tau r^{n+1}$.
Thus a PDE-residual tolerance $\delta$ requires
$\|R^{n+1}\|_2\le\tau\delta$.
For example, $N\delta\lesssim\tau^\alpha+N^{-\varepsilon}$
makes the solver term no larger than the temporal and spatial
terms in \eqref{oct-solver-computed-NS}.
The condition $N^2\delta\le B$ is used only for
\eqref{oct-solver-computed-B2}.
For Picard iteration, the contraction condition must hold in addition
to the time-step restrictions of the error estimate.
\end{remark}

\section{Numerical experiments}
\label{sec:oct-numerics}
We examine temporal, spatial, and vanishing-viscosity errors for
fixed rough data and smooth solutions, followed by qualitative
comparisons of flow evolution.

\subsection{Rough data}
\label{sec:expanded-rough}
The nonlinear tests use projections of one fixed rough velocity and zero
forcing. An infinite shear gives exact error formulas in $B^1_{2,1}$ and
$B^0_{\infty,1}$.

\subsubsection{Discretization}
\label{sec:expanded-discretization}
For $\varepsilon\in\{1/4,3/4\}$, let
\begin{align}
 u_{0,\varepsilon}={}&(0.2,-0.1)+(\partial_2\psi,-\partial_1\psi)
 \notag\\
 &+\sum_{j\ge3}A_j\left[(0,1)\cos(k_j\cdot x+\phi_j)
 +\frac{(-1,2)}{\sqrt5}\cos(\ell_j\cdot x+\vartheta_j)\right],
 \label{oct5:eq-numerical-rough-data}\\
 \psi={}&0.12\cos(x_1+x_2)+0.04\cos(2x_1-x_2+0.4)
 +0.01\cos(3x_1+2x_2+0.7),\notag\\
 A_j={}&\frac{2^{-(2+\varepsilon)j}}{(j+1)^2},\qquad
 k_j=(2^j,0),\qquad \ell_j=(2^j,2^{j-1}),\notag\\
 \phi_j={}&0.37j+0.11,\qquad \vartheta_j=0.53j+0.29.\notag
\end{align}
The polarizations are perpendicular to the corresponding frequencies.
The two families are nonparallel, and their projected nonlinear
interaction is nonzero. Finite overlap of the dyadic blocks gives
$u_{0,\varepsilon}\in B^{2+\varepsilon}_{\infty,1}$.
The multiplier of the $k_j$ mode equals one in block $j$.
Orthogonality therefore gives
$\|\Delta_j u_{0,\varepsilon}\|_\infty\ge cA_j$ for large $j$, and
the divergence of $\sum_j2^{(1-\varepsilon)j}(j+1)^{-2}$ implies
$u_{0,\varepsilon}\notin B^3_{\infty,1}$.
Each computation starts from the exact projection
$\Pi_Nu_{0,\varepsilon}$ of this same infinite datum.

We implement \eqref{eq:oct-scheme} by Fourier differentiation, the
Fourier Leray projector, and Picard iteration.
We denote the computed iterates by $\widetilde u_N^n$, as in
Section~\ref{rev:sec-picard}.
Quadratic products are evaluated on a grid of side $M_g>3N$, where
$M_g$ is the smallest such integer with prime factors in $\{2,3,5\}$.
This choice eliminates aliasing in the retained convolution.
The iteration stops when $\|r^{n+1}\|_2\le\delta$,
with $\delta=2\times10^{-12}$ and
\[
 r^{n+1}
 =\frac{\widetilde u_N^{n+1}-\widetilde u_N^n}{\tau}
  +\Pi_N\PP(\widetilde u_N^n\cdot\nabla\widetilde u_N^{n+1})
  -\nu\Delta\widetilde u_N^{n+1}.
\]
The residual is measured in PDE units, as in
Corollary~\ref{cor:m2an-implemented-endpoint}.

The dyadic cutoff is
\[
 \chi(r)=\frac{e^{-1/(2-r)}}{e^{-1/(r-1)}+e^{-1/(2-r)}},\quad 1<r<2,
 \qquad \chi(r)=1\ (r\le1),\quad \chi(r)=0\ (r\ge2).
\]
We compute each block $L^2$ norm by Parseval and approximate each block
supremum on a uniform grid. We write $E_1$ for a $B^1_{2,1}$ difference
and $E_\infty$ for the resulting sampled $B^0_{\infty,1}$ difference.
The sampling grid has side $8N$ in time and viscosity comparisons;
all spatial comparisons use side 2048. The final time and checkpoints are
\[
 T=1/16,\qquad \mathcal T=\{0,T/4,T/2,3T/4,T\}.
\]
A superscript $\mathcal T$ denotes a maximum over these checkpoints.

\subsubsection{Time and viscosity errors}
For time refinement we use $N=64$, $\nu\in\{0,0.01\}$, and
$\tau=2^{-7},\ldots,2^{-10}$. The reference is classical RK4 applied
to the same dealiased Galerkin system with step $2^{-11}$.
For the inviscid comparison we use $N=128$, RK4 step $2^{-11}$, and
$\nu=0.02,0.01,0.005,0.0025,0.00125$, comparing each viscous solution
with the Euler solution from the same initial projection.
Table~\ref{oct5:tab-rough-time} lists the last two levels of each
sequence. The final-time differences equal the checkpoint maxima in
these runs.

\begin{table}[htbp]
\centering\small
\begin{tabular}{llrrrr}
\toprule
 & & \multicolumn{2}{c}{$\varepsilon=1/4$}
     & \multicolumn{2}{c}{$\varepsilon=3/4$}\\
Comparison & Parameter & $E_1(T)$ & $E_\infty(T)$
                        & $E_1(T)$ & $E_\infty(T)$\\
\midrule
Time, $\nu=0$ & $\tau=2^{-9}$ & 8.469e-06 & 4.827e-06 & 5.268e-06 & 4.187e-06\\
 & $\tau=2^{-10}$ & 4.238e-06 & 2.414e-06 & 2.635e-06 & 2.094e-06\\
Time, $\nu=0.01$ & $\tau=2^{-9}$ & 9.342e-06 & 4.879e-06 & 5.515e-06 & 4.231e-06\\
 & $\tau=2^{-10}$ & 4.679e-06 & 2.440e-06 & 2.759e-06 & 2.116e-06\\
Viscosity & $\nu=0.0025$ & 5.894e-04 & 2.116e-04 & 3.826e-04 & 1.930e-04\\
 & $\nu=0.00125$ & 3.123e-04 & 1.064e-04 & 1.933e-04 & 9.662e-05\\
\bottomrule
\end{tabular}
\caption{Time refinement at $N=64$ and viscous--inviscid differences at
$N=128$ for the data \eqref{oct5:eq-numerical-rough-data}.}
\label{oct5:tab-rough-time}
\label{oct5:tab-rough-inviscid}
\end{table}

Halving the RK4 reference step from $2^{-10}$ to $2^{-11}$ changes the
checkpoint values by at most $1.412\times10^{-10}$ in $B^1_{2,1}$
and $2.516\times10^{-12}$ in the sampled endpoint norm.
The temporal errors decrease at first order. The endpoint viscosity
differences are also nearly linear in $\nu$, while the $B^1_{2,1}$
differences approach this regime more slowly for $\varepsilon=1/4$.
At fixed $N$, the Galerkin system is smooth, which explains the
first-order temporal regime. These orders alone do not determine the
continuum rates for rough data.

\subsubsection{Spatial errors}
We hold $\tau=2^{-10}$ and $\nu=0$ fixed and compare the semi-implicit
solutions at consecutive resolutions:
\[
 D_q^{\mathcal T}(N)=\max_{t_n\in\mathcal T}
    \|\widetilde u_{2N}^n-\widetilde u_N^n\|_{X_q},\qquad
 X_1=B^1_{2,1},\quad X_\infty=B^0_{\infty,1}.
\]
Endpoint norms are sampled as above, after the Fourier vectors have
been lifted to a common grid.
Table~\ref{oct5:tab-rough-space} also gives the initial
$B^2_{2,1}$ tail, which enters the spatial consistency estimate.
It is evaluated from the infinite coefficient series by Parseval,
retaining modes through $j=160$; the omitted contribution is below
$10^{-15}$.

\begin{table}[htbp]
\centering\small
\begin{tabular}{rrrrr}
\toprule
$\varepsilon$ & $N$ & $D_1^{\mathcal T}(N)$
 & $D_\infty^{\mathcal T}(N)$
 & $\|(I-\Pi_N)u_{0,\varepsilon}\|_{B^2_{2,1}}$\\
\midrule
0.25 & 32  & 1.525e-04 & 8.095e-06 & 2.240e-02\\
0.25 & 64  & 6.066e-05 & 2.035e-06 & 1.519e-02\\
0.25 & 128 & 2.658e-05 & 5.606e-07 & 1.055e-02\\
0.75 & 32  & 2.113e-05 & 1.194e-06 & 1.690e-03\\
0.75 & 64  & 6.274e-06 & 2.209e-07 & 7.886e-04\\
0.75 & 128 & 2.056e-06 & 4.442e-08 & 3.781e-04\\
\bottomrule
\end{tabular}
\caption{Spatial Cauchy differences at a common time step. Each row
compares resolutions $N$ and $2N$. Endpoint block maxima are sampled
on the same $2048^2$ grid in every row.}
\label{oct5:tab-rough-space}
\end{table}

The initial $B^2_{2,1}$ tail has size
$N^{-\varepsilon}(\log N)^{-2}$ along dyadic resolutions. The Cauchy
differences decrease faster, as does the initial tail in the target
norms. Changing the comparison resolution from 128 to 256 changes
the $N=64$ endpoint difference at $T$ from $2.035\times10^{-6}$ to
$1.876\times10^{-6}$ when $\varepsilon=1/4$, and from
$2.209\times10^{-7}$ to $2.018\times10^{-7}$ when $\varepsilon=3/4$.
We therefore report consecutive-resolution differences without treating
either resolution as an exact reference.

For the $N=64$ and 128 rows, doubling the sampling grid from side 2048
to 4096 changes $D_\infty^{\mathcal T}$ by at most $0.269\%$,
measured relative to the refined value. This measures sampling sensitivity,
without providing a bound on the continuous supremum.

\subsubsection{Exact shear}
\label{sec:expanded-exact-shear}
\label{m2an:sec-shear-comparison}
For the datum of
Proposition~\ref{oct5:prop-fixed-shear-calibration}, take $a=1$,
$c=1/2$, $T=1/16$, and $\varepsilon\in\{1/4,3/4\}$.
Write $k_j=2^j$ and $A_j=2^{-(2+\varepsilon)j}(j+1)^{-2}$.
At $h=2^{-r}$, $r=12,24,\ldots,192$, the exact scalar multipliers give
\begin{align*}
 E_1^{\rm time}(h)
 &=\frac1{\sqrt2}\sum_{j\ge1}k_jA_j
       \left|(1+ihck_j)^{-T/h}-e^{-iTck_j}\right|,\\
 E_1^{\rm visc}(h)
 &=\frac1{\sqrt2}\sum_{j\ge1}k_jA_j(1-e^{-hTk_j^2}).
\end{align*}
The exact endpoint errors $E_0^{\rm time}$ and $E_0^{\rm visc}$ are
given by \eqref{m2an:eq-shear-endpoint-time} and
\eqref{m2an:eq-shear-endpoint-visc}; they contain neither the factor
$k_j$ nor the normalized $L^2$ factor $1/\sqrt2$.
The infinite-series formulas give the errors directly, independently
of a spatial Galerkin cutoff or numerical reference solution.

For stable evaluation of the temporal multiplier, set $x=hck_j$ and
remove the exact transport phase:
\[
 \left|(1+ix)^{-T/h}-e^{-iTx/h}\right|
 =|e^{d+i\theta}-1|,\qquad
 d=-\frac{T}{2h}\log(1+x^2),\quad
 \theta=\frac{T}{h}(x-\arctan x).
\]
Small-argument formulas are used for the logarithm, phase, and exponential
difference. The sums are truncated adaptively using
\[
 \sum_{j>J}\frac{2^{-bj}}{(j+1)^2}
 \le\frac{2^{-b(J+1)}}{(J+2)^2(1-2^{-b})},\qquad b>0.
\]
The resulting cutoffs range from $J=32$ to 135, with relative tail
bounds below $2\times10^{-14}$ for all four errors.
Independent evaluations at 100 decimal digits differ from the
double-precision values by at most $9.471\times10^{-15}$ in relative
terms.

\begin{table}[htbp]
\centering\small
\begin{tabular}{rrrrrrrr}
\toprule
$\varepsilon$ & $r$ & $E_1^{\rm time}$ & $p_{\rm time}$
 & $E_1^{\rm visc}$ & $p_{\rm visc}$
 & \shortstack{$E_0^{\rm time}$\\$/ (hL_0^{\rm time})$}
 & \shortstack{$E_0^{\rm visc}$\\$/ (hL_0^{\rm visc})$}\\
\midrule
0.25 & 24  & 1.302e-08 & 0.761 & 6.029e-08 & 0.778 & 0.997662 & 0.996440\\
0.25 & 96  & 3.050e-23 & 0.655 & 1.187e-22 & 0.656 & 0.999999 & 0.999999\\
0.25 & 192 & 7.074e-42 & 0.640 & 2.674e-41 & 0.640 & 1.000000 & 1.000000\\
0.75 & 24  & 4.030e-10 & 0.967 & 3.014e-09 & 0.965 & 0.999990 & 0.999974\\
0.75 & 96  & 1.632e-30 & 0.915 & 1.088e-29 & 0.917 & 1.000000 & 1.000000\\
0.75 & 192 & 1.880e-56 & 0.891 & 1.197e-55 & 0.891 & 1.000000 & 1.000000\\
\bottomrule
\end{tabular}
\caption{Exact-mode errors for the fixed infinite shear. The orders are
$p=\log_2(E(2^{-(r-12)})/E(2^{-r}))/12$, using the preceding level
of the full sequence.}
\label{oct5:tab-rough-shear}
\end{table}

\begin{table}[htbp]
\centering\small
\begin{tabular}{rrrr}
\toprule
$\varepsilon$ & $\sum_{j\ge1}k_j^2A_j$
 & $L_0^{\rm time}$ & $L_0^{\rm visc}$\\
\midrule
0.25 & 0.3800440531 & 2.9690941649e-03 & 2.3752753319e-02\\
0.75 & 0.2098289586 & 1.6392887388e-03 & 1.3114309910e-02\\
\bottomrule
\end{tabular}
\caption{Endpoint limits from \eqref{m2an:eq-shear-endpoint-limits}.
The coefficient sums were evaluated at 100 decimal digits with relative
tail bounds below $10^{-70}$.}
\label{m2an:tab-shear-endpoint-limits}
\end{table}

\begin{figure}[htbp]
\centering

\begingroup%
\providecommand{\mathdefault}[1]{#1}%
\makeatletter%
%
%
\makeatother%
\endgroup%

\caption{Errors for the fixed infinite shear in two norms, with
$\alpha=(1+\varepsilon)/2$. Left: auxiliary errors scaled by
$(r+2)^2/h^\alpha$. Right: endpoint errors scaled by their positive
first-order limits in Table~\ref{m2an:tab-shear-endpoint-limits}.}
\label{m2an:fig-shear-dual-norm}
\end{figure}

The auxiliary orders approach $0.625$ and $0.875$ slowly because of
the squared logarithmic factor in
Proposition~\ref{oct5:prop-fixed-shear-calibration}. The
endpoint ratios tend to one, as proved in
Proposition~\ref{prop:m2an-shear-endpoint-first-order}.
Thus the datum that realizes the auxiliary exponent has first-order
endpoint error.

\subsubsection{Residual checks}
Across the nonlinear rough-data runs, the largest PDE residual was
$1.994\times10^{-12}$ and no step required more than nine Picard
iterations. The largest absolute defect in the kinetic-energy identity
was $3.852\times10^{-17}$, and the largest Fourier divergence
coefficient was $1.355\times10^{-20}$.
Corollary~\ref{cor:m2an-implemented-endpoint} bounds the accumulated
endpoint contribution of a residual tolerance $\delta$ by $CN\delta$,
under its regularity assumptions. Here
$N^2\delta\le1.31\times10^{-7}$ for $N\le256$; the solution-dependent
constant in that corollary has not been evaluated numerically.

\subsection{Smooth solutions}
\label{sec:expanded-smooth}
We test temporal and inviscid convergence using a manufactured solution
and unforced evolution from the same initial velocity. A transported
shear provides an explicit spatial tail for checking the projection
and dyadic norms.

The Fourier normalization, dyadic cutoff, and velocity update are those
of Section~\ref{sec:expanded-discretization}. For the quantitative smooth-flow tests below,
products are evaluated on a grid of side $3N+2$ and endpoint block
suprema are sampled on a grid of side $16N+2$. The product grid removes
aliasing in the retained quadratic convolution. Picard iteration uses
the PDE-residual tolerance $2\times10^{-12}$. All errors in these tests are evaluated at the
stated final time.

\subsubsection{Manufactured solution}
\label{sec:expanded-manufactured}
Let
\begin{align}
 \psi_0(x_1,x_2)={}&0.30\cos(x_1+x_2)
       +0.08\cos(2x_1-x_2+0.4)\notag\\
 &+0.020\cos(5x_1+3x_2+0.7)
       +0.006\cos(9x_1-2x_2+1.1),\notag\\
 w_0={}&(\partial_2\psi_0,-\partial_1\psi_0),\qquad
 c=(0.2,-0.1),\qquad u_e(t)=c+e^{-t}w_0.
 \label{eq:expanded-smooth-data}
\end{align}
The nonparallel modes occupy several dyadic blocks, and their projected
nonlinear interaction is nonzero. We prescribe the solenoidal force
\begin{equation}
 f_\nu(t)=\partial_tu_e(t)+\PP(u_e(t)\cdot\nabla u_e(t))
                 -\nu\Delta u_e(t)
 \label{eq:expanded-manufactured-force}
\end{equation}
and use the forced update
\begin{equation}
 \frac{u_N^{n+1}-u_N^n}{\tau}
 +\Pi_N\PP(u_N^n\cdot\nabla u_N^{n+1})
 =\nu\Delta u_N^{n+1}+\Pi_N f_\nu(t_{n+1}).
 \label{eq:expanded-forced-step}
\end{equation}
The force is evaluated at the new time level in both this equation and
its Picard realization. The stopping residual includes
$-\Pi_Nf_\nu(t_{n+1})$. We take $T=0.4$, $N=16$,
$\tau=0.04/2^k$, $k=0,1,2,3$, and
$\nu\in\{0,10^{-4},10^{-2},0.2\}$. Since $u_e$ is supported within
the retained Fourier square, its projection introduces no error in the
comparison solution.

Table~\ref{tab:expanded-manufactured} gives the two finest time steps.
The errors decrease at first order in all four norms. The curves for
$\nu=0$ and $10^{-4}$ in
Figure~\ref{fig:expanded-smooth-rates}(a) nearly coincide, with no
visible growth in the temporal error as viscosity decreases between
these values. Since the force
\eqref{eq:expanded-manufactured-force} depends on $\nu$, this
comparison concerns consistency with the prescribed trajectory;
the fixed-force inviscid limit is examined below.

\begin{table}[htbp]
\centering\small
\begin{tabular}{rrrrrrr}
\toprule
$\nu$ & $\tau$ & $L^2$ & $H^1$ & $B^0_{\infty,1}$
 & $B^0_{\infty,2}$ & $p_1$\\
\midrule
0 & 0.010 & 6.461e-04 & 2.910e-03 & 2.080e-03 & 1.030e-03 & 0.993\\
0 & 0.005 & 3.238e-04 & 1.465e-03 & 1.043e-03 & 5.163e-04 & 0.996\\
$10^{-4}$ & 0.010 & 6.458e-04 & 2.906e-03 & 2.078e-03 & 1.029e-03 & 0.993\\
$10^{-4}$ & 0.005 & 3.237e-04 & 1.463e-03 & 1.042e-03 & 5.159e-04 & 0.996\\
$10^{-2}$ & 0.010 & 6.222e-04 & 2.532e-03 & 1.927e-03 & 9.500e-04 & 0.991\\
$10^{-2}$ & 0.005 & 3.118e-04 & 1.275e-03 & 9.665e-04 & 4.766e-04 & 0.995\\
0.2 & 0.010 & 4.994e-04 & 1.053e-03 & 1.088e-03 & 6.058e-04 & 0.992\\
0.2 & 0.005 & 2.502e-04 & 5.280e-04 & 5.453e-04 & 3.036e-04 & 0.996\\
\bottomrule
\end{tabular}
\caption{Manufactured temporal errors at $T=0.4$ and $N=16$.
The order $p_1$ is computed in $B^0_{\infty,1}$ from the preceding
level of the full time-refinement sequence.}
\label{tab:expanded-manufactured}
\end{table}

\subsubsection{Unforced flow}
\label{sec:expanded-smooth-time}
We next use the same initial velocity $u_0=c+w_0$, set $f=0$, and
evolve to $T=0.2$. The semi-implicit calculation uses $N=32$ and
$\tau=0.01/2^k$, $k=0,1,2,3$. For each
$\nu\in\{0,0.01\}$, the reference is obtained by classical RK4
applied to the same dealiased Galerkin ODE with time step
$1.25\times10^{-4}$. Both time integrators use the same spatial
discretization.

\begin{table}[htbp]
\centering\small
\begin{tabular}{rrrrrrr}
\toprule
$\nu$ & $\tau$ & $L^2$ & $H^1$ & $B^0_{\infty,1}$
 & $B^0_{\infty,2}$ & $p_1$\\
\midrule
0 & 0.01000 & 5.679e-04 & 5.229e-03 & 2.606e-03 & 1.911e-03 & ---\\
0 & 0.00500 & 2.858e-04 & 2.632e-03 & 1.313e-03 & 9.637e-04 & 0.989\\
0 & 0.00250 & 1.433e-04 & 1.321e-03 & 6.590e-04 & 4.838e-04 & 0.995\\
0 & 0.00125 & 7.177e-05 & 6.614e-04 & 3.301e-04 & 2.424e-04 & 0.997\\
0.01 & 0.01000 & 4.944e-04 & 4.494e-03 & 2.245e-03 & 1.652e-03 & ---\\
0.01 & 0.00500 & 2.493e-04 & 2.268e-03 & 1.133e-03 & 8.345e-04 & 0.986\\
0.01 & 0.00250 & 1.252e-04 & 1.139e-03 & 5.691e-04 & 4.193e-04 & 0.993\\
0.01 & 0.00125 & 6.271e-05 & 5.707e-04 & 2.852e-04 & 2.101e-04 & 0.997\\
\bottomrule
\end{tabular}
\caption{Unforced nonlinear temporal errors against the RK4 reference
at $T=0.2$ and $N=32$; $p_1$ denotes the $B^0_{\infty,1}$ order.}
\label{tab:expanded-unforced-time}
\end{table}

The first-order behavior in Table~\ref{tab:expanded-unforced-time}
agrees with the smooth-data estimates in
Theorem~\ref{thm:oct-refined-sobolev} and
Propositions~\ref{prop:oct-bes-time-Besov1}
and~\ref{prop:oct-bes-time-Besov2}.
The $\ell^1$ and $\ell^2$ block sums give different error magnitudes
but the same observed temporal order.

Halving the RK4 step from $2.5\times10^{-4}$ to
$1.25\times10^{-4}$ changes the reference by at most
$4.415\times10^{-15}$ in $H^1$, at the scale of floating-point
roundoff. A separate Euler reference at $N=48$ differs from the lifted
$N=32$ solution by $1.793\times10^{-8}$ in $L^2$,
$6.079\times10^{-7}$ in $H^1$, and $7.053\times10^{-8}$ in sampled
$B^0_{\infty,1}$. These spatial differences lie below the displayed
time errors. A continuum error bound would still require control
of the remaining spatial tail.

\subsubsection{Inviscid limit}
\label{sec:expanded-smooth-inviscid}
Keep $u_0=c+w_0$, $f=0$, $N=32$, and $T=0.2$ fixed. We compare
the Galerkin Navier--Stokes and Euler solutions using RK4 with step
$1.25\times10^{-4}$ and viscosities $\nu=0.02/2^k$,
$k=0,\ldots,4$. We also compare the two semi-implicit solutions at
the common step $\tau=2.5\times10^{-4}$. Both sequences use the same
initial projection and zero force at every viscosity.

\begin{table}[htbp]
\centering\small
\begin{tabular}{rrrrrrr}
\toprule
 & \multicolumn{4}{c}{RK4 Galerkin comparison}
 & \multicolumn{2}{c}{Semi-implicit comparison}\\
$\nu$ & $H^1$ & $B^0_{\infty,1}$ & $p_1$ & $B^0_{\infty,2}$
 & $B^0_{\infty,1}$ & $B^0_{\infty,2}$\\
\midrule
0.02000 & 1.221e-01 & 4.028e-02 & --- & 2.464e-02 & 4.028e-02 & 2.464e-02\\
0.01000 & 6.543e-02 & 2.123e-02 & 0.924 & 1.310e-02 & 2.123e-02 & 1.310e-02\\
0.00500 & 3.391e-02 & 1.091e-02 & 0.961 & 6.760e-03 & 1.091e-02 & 6.760e-03\\
0.00250 & 1.727e-02 & 5.530e-03 & 0.980 & 3.435e-03 & 5.531e-03 & 3.435e-03\\
0.00125 & 8.713e-03 & 2.784e-03 & 0.990 & 1.731e-03 & 2.785e-03 & 1.731e-03\\
\bottomrule
\end{tabular}
\caption{Viscous--inviscid differences at $T=0.2$ and $N=32$.
Both solutions in each comparison have the same initial data and time
step. The RK4 order $p_1$ is measured under viscosity halving.}
\label{tab:expanded-smooth-inviscid}
\end{table}

The differences in Table~\ref{tab:expanded-smooth-inviscid} approach
a linear dependence on $\nu$, consistently with the smooth inviscid
bounds in Theorem~\ref{thm:oct-refined-sobolev} and
Corollary~\ref{cor:oct-bes-inviscid}. Figure~\ref{fig:expanded-smooth-rates}(b)
also includes the $L^2$ differences. At the smallest viscosity,
halving the semi-implicit step changes all four difference norms by
at most $9.5983\times10^{-5}$ in relative terms. The spatial
resolution is held fixed throughout; a joint limit in $(\tau,N,\nu)$
would also require spatial refinement.

\begin{figure}[htbp]
\centering

\begingroup%
\providecommand{\mathdefault}[1]{#1}%
\makeatletter%
%
\makeatother%
\endgroup%

\caption{Smooth-data temporal and viscosity errors.
Left: the manufactured $B^0_{\infty,1}$ time error at $T=0.4$, $N=16$.
Right: unforced viscous--inviscid differences from the RK4 Galerkin
solutions at $T=0.2$, $N=32$. The two panels use different forcing
conventions, as specified in their respective tests.}
\label{fig:expanded-smooth-rates}
\end{figure}

\subsubsection{Spatial tail}
\label{sec:expanded-spatial-shear}
Consider the transported shear
\begin{equation}
 u_0(x_1,x_2)=\left(0.4,
   \sum_{j\ge0}\frac{0.3\,2^{-4j}}{(j+1)^2}\sin(2^jx_1)\right),
 \qquad \nu=0,\quad f=0.
 \label{eq:expanded-transported-shear}
\end{equation}
This datum belongs to $B^4_{\infty,1}$ and its second component is
translated by $0.4t$. For a retained frequency $k$, the exact
semi-implicit coefficient is
\[
 \widehat u_2^n(k)=(1+i\tau\,0.4k)^{-n}\widehat u_{0,2}(k).
\]
We evaluate this coefficient as
$\exp[-n\log(1+i\tau\,0.4k)]\widehat u_{0,2}(k)$, with
$T=0.2$, $\tau=10^{-12}$, and $N=4,8,16,32,64$.
At this time step, the temporal contribution is below the displayed
spatial errors; the closed formula avoids repeated solves. The comparison
solution retains the infinite datum through $j=9$ and is sampled on
8192 points; the omitted velocity amplitudes sum to less than
$2.4\times10^{-15}$.

\begin{table}[htbp]
\centering\small
\begin{tabular}{rrrrrrr}
\toprule
$N$ & $H^1$ & $p_H$ & $B^0_{\infty,1}$ & $p_1$
 & $B^0_{\infty,2}$ & $p_2$\\
\midrule
4  & 2.618e-05 & --- & 4.769e-06 & --- & 4.581e-06 & ---\\
8  & 2.083e-06 & 3.651 & 1.914e-07 & 4.639 & 1.833e-07 & 4.644\\
16 & 1.807e-07 & 3.528 & 8.330e-09 & 4.522 & 7.955e-09 & 4.526\\
32 & 1.659e-08 & 3.445 & 3.832e-10 & 4.442 & 3.653e-10 & 4.445\\
64 & 1.588e-09 & 3.385 & 1.836e-11 & 4.384 & 1.748e-11 & 4.386\\
\bottomrule
\end{tabular}
\caption{Spatial errors for the transported shear at $T=0.2$.
Orders are computed from consecutive Fourier cutoffs.}
\label{tab:expanded-spatial-shear}
\end{table}

The rates in Table~\ref{tab:expanded-spatial-shear} follow from the
prescribed coefficients. For $N=2^J$, each unresolved sine lies in a
single dyadic block. The exact endpoint projection tails are
\[
 \sum_{j>J}\frac{0.3\,2^{-4j}}{(j+1)^2}
 \quad\hbox{and}\quad
 \left(\sum_{j>J}\frac{0.3^2\,2^{-8j}}{(j+1)^4}\right)^{1/2}
\]
in $B^0_{\infty,1}$ and $B^0_{\infty,2}$, respectively.
Both have size $N^{-4}(\log N)^{-2}$, while the $H^1$ tail has size
$N^{-3}(\log N)^{-2}$. The logarithmic factor accounts for the
orders above four and three on the displayed range.
The absence of transfer among Fourier modes is essential here. The
general $N^{-2}$ bound for
$B^4_{\infty,1}$ data, which allows nonlinear interactions, is given in
Theorem~\ref{thm:oct-bes-sharp-Fourier-Besov1}.

\begin{figure}[htbp]
\centering

\begingroup%
\providecommand{\mathdefault}[1]{#1}%
\makeatletter%
%
\makeatother%
\endgroup%

\caption{Spatial-tail convergence for
\eqref{eq:expanded-transported-shear}. The different Sobolev and Besov
slopes reflect the explicit coefficient decay of this shear.}
\label{fig:expanded-spatial-shear}
\end{figure}

\subsubsection{Numerical checks}
\label{sec:expanded-smooth-checks}
For the unforced time-error sample with $\nu=0.01$ and $\tau=0.0025$,
refining the endpoint sampling grid from side $16N+2$ to $32N+2$
changes the two Besov norms by relative amounts
$1.972\times10^{-4}$ and $1.654\times10^{-4}$.
The maximum accepted PDE residual in these smooth-flow experiments
was $1.995\times10^{-12}$, and no step required more than 18 Picard
iterations. All unforced convergence runs had nonincreasing computed kinetic
energy. The largest defect in \eqref{eq:oct-energy-identity} was
$2.258\times10^{-16}$. At the initial state, the largest Fourier
divergence coefficient was $1.388\times10^{-17}$ and the absolute
transport energy pairing was $8.674\times10^{-19}$.

\raggedbottom
\subsection{Flow dynamics}
\label{sec:expanded-dynamics}
We compare cellular flow, double shear layers, and a Gaussian vortex
pair. For one case in each example, we refine the time step and Fourier cutoff.
All $N=96$ refinements use a $300\times300$ product grid.
Relative differences are normalized by the corresponding norm of the
finer calculation; they are not errors against an exact solution.

\subsubsection{Cellular flow}
\label{sec:expanded-cellular}

For the cellular example considered in~\cite{LMS22}, the
initial streamfunction is
\begin{equation}\label{eq:expanded-cellular-data}
 \psi_m(x,y)=\frac12\cos^m x\cos^m y,\qquad
 u_{0,m}=(\partial_y\psi_m,-\partial_x\psi_m),
 \qquad m\in\{2,8,20\}.
\end{equation}
Thus
\[
 u_{0,m}=\frac m2\bigl(-\cos^m x\cos^{m-1}y\sin y,
                         \cos^{m-1}x\cos^m y\sin x\bigr).
\]
These fields are smooth, periodic, and divergence free.
As $m$ increases,
the velocity becomes more concentrated and its higher derivatives
grow; these values of $m$ thus correspond to different initial
norms in the local estimates.

We obtain the initial Fourier coefficients from the finite binomial
expansion of $\cos^m$. The calculations use
\eqref{eq:oct-scheme} with $N=64$, $\nu=10^{-4}$, and
$\tau=2.5\times10^{-3}$. Products are evaluated on a
$200\times200$ grid, and the PDE-residual tolerance is
$2\times10^{-12}$. Figure~\ref{fig:expanded-cellular} shows
instantaneous streamlines at $t=0.01,1,5$, obtained as level curves
of the computed streamfunction. Equal-length arrows indicate the velocity
direction; colour does not encode speed.
For each $m$, the contour levels are fixed across the three times.

The broad cells at $m=2$ deform over the displayed interval.
For $m=8$ and $20$, the contours are concentrated around narrower
cores, which largely retain their shape. For $m=20$, halving
$\tau$ changes the velocity and vorticity at $t=5$ by
$0.0108\%$ and $0.0432\%$ in relative $L^2$, respectively.
Increasing $N$ from $64$ to $96$ at fixed $\tau$ gives
$0.0669\%$ and $0.298\%$, respectively.

\begin{figure}[!tbp]
\centering
\includegraphics[width=\textwidth]
 {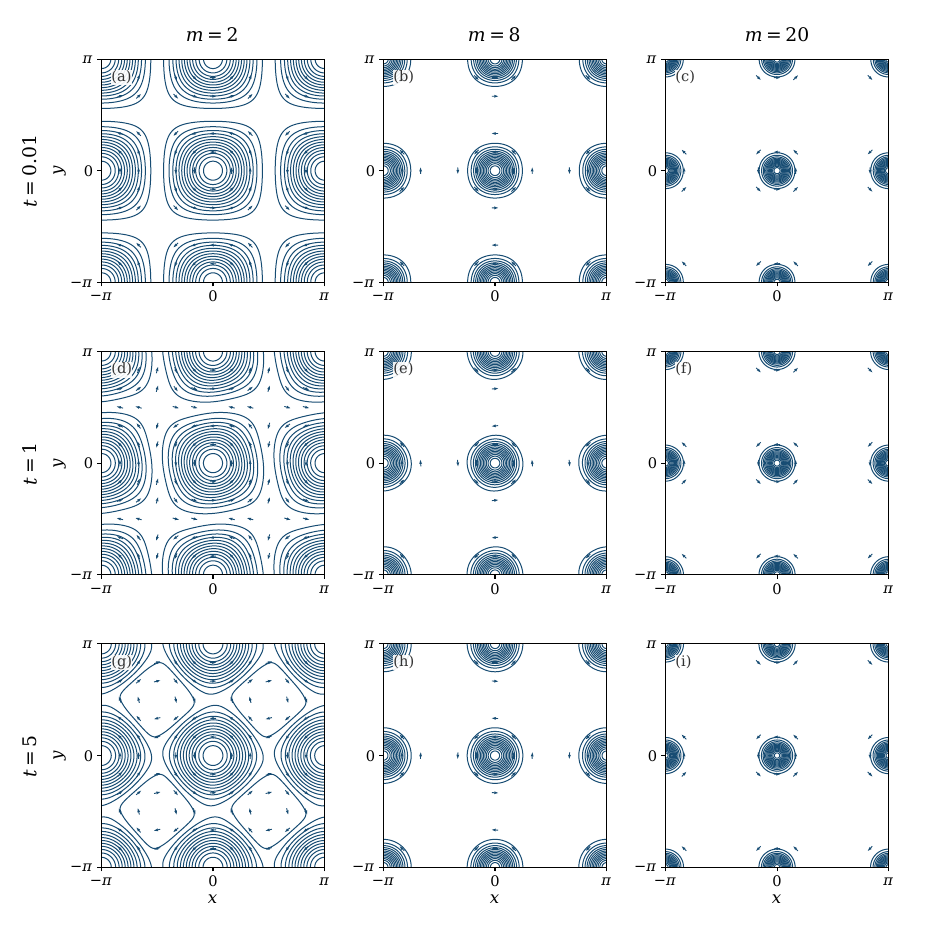}
\caption{Streamlines for the cellular data
\eqref{eq:expanded-cellular-data}. Columns correspond to $m=2,8,20$;
rows to $t=0.01,1,5$. Parameters: $\nu=10^{-4}$,
$\tau=2.5\times10^{-3}$, and $N=64$.
Each column uses the same contour levels at all three times.
}
\label{fig:expanded-cellular}
\end{figure}

\subsubsection{Double shear layers}
\label{sec:expanded-double-shear}

The double-shear example of~\cite{Hakim22} uses the periodic
velocity
\begin{equation}\label{eq:expanded-double-shear-data}
 \begin{aligned}
 u_0(x,y)&=\bigl(U(y),0.05\sin(x+\pi)\bigr),\qquad
 \rho=\frac{\pi}{15},\\
 U(y)&=\begin{cases}
 \tanh\bigl((y+\pi/2)/\rho\bigr),&-\pi\le y\le0,\\
 \tanh\bigl((\pi/2-y)/\rho\bigr),&0<y\le\pi,
 \end{cases}
 \end{aligned}
\end{equation}
It is divergence free. The vorticity
$\omega_0=0.05\cos(x+\pi)-U'(y)$, defined almost everywhere,
consists of two layers with
opposite signs and a transverse perturbation. The profile $U$ is
continuous and periodic, but its derivative has small jumps at
$y=0$ and at the periodic seam. We use this example to examine the flow evolution,
without applying the smooth-data estimates to this profile.

We compute the initial Fourier projection by one-dimensional
quadrature of $U$ on its smooth pieces. The calculations use
$N=64$, $\tau=10^{-3}$, and $\nu\in\{0,10^{-3},10^{-1}\}$,
with the same product grid and residual tolerance as the cellular
tests. Figure~\ref{fig:expanded-double-shear} shows the vorticity
at $t=0,5,10$, divided in each panel by
$A=\max_{\mathcal G}|\omega|$, where $\mathcal G$ is the
$512\times512$ plotting grid. The colour scale for $\omega/A$
is linear and fixed at $[-1,1]$. Each panel gives $A$ to three
significant digits for comparison of the physical amplitudes.

For $\nu=0$ and $10^{-3}$, the layers bend and roll up into
vortices. At $\nu=10^{-1}$, diffusion broadens the layers and
the sampled peak vorticity at $t=10$ is $0.4893$.
In the inviscid calculation, the maximum of $|\omega|$ sampled
on the $512\times512$ plotting grid
increases from $4.82465$ initially to $5.32513$ at $t=10$.
Euler transport preserves the vorticity range, so this excess
is numerical. The finest oscillatory bands in the last inviscid
panel should not be interpreted as resolved Euler structures.
For $\nu=0$ at $t=10$, halving the time step gives relative
$L^2$ differences of $0.167\%$ in velocity and $3.65\%$ in
vorticity. Increasing $N$ from $64$ to $96$ at fixed $\tau$
gives $0.607\%$ and $20.5\%$, respectively. The larger spatial
difference in vorticity indicates that the fine structures remain
sensitive to resolution.

\begin{figure}[!tbp]
\centering
\includegraphics[width=\textwidth]
 {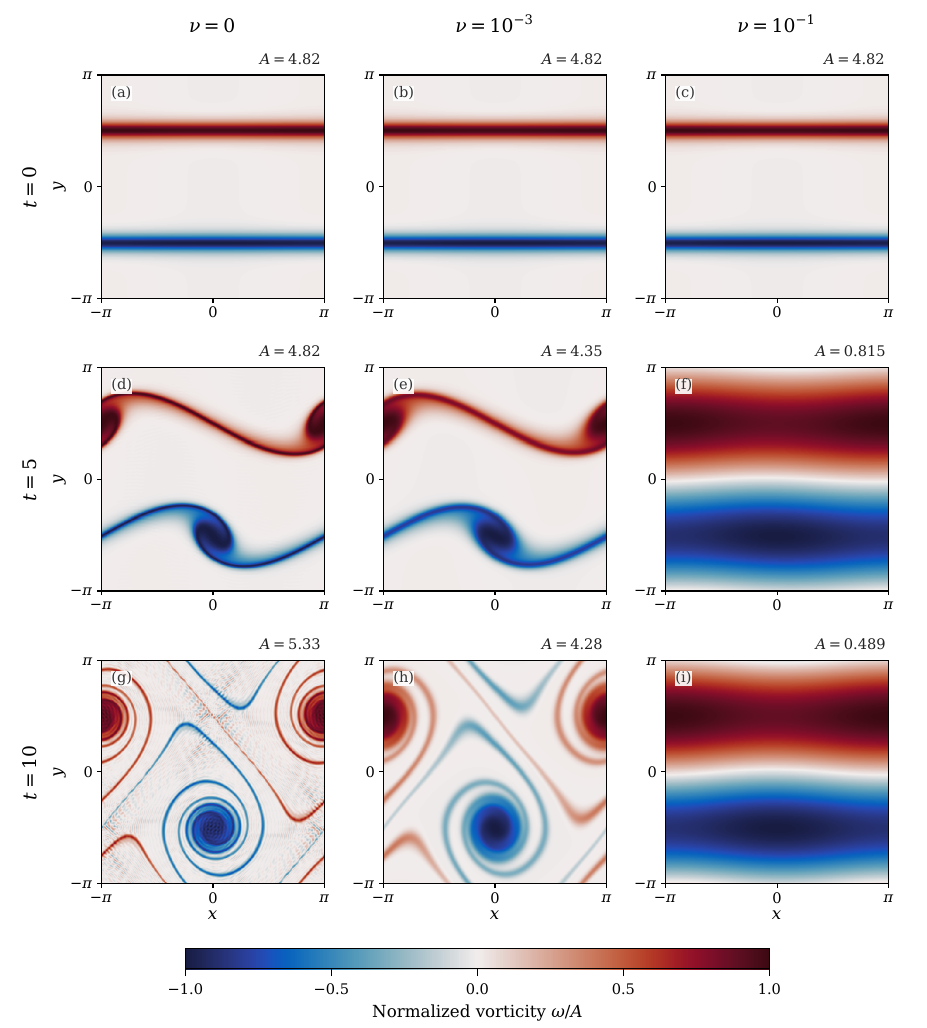}
\caption{Normalized vorticity for the double-shear data
\eqref{eq:expanded-double-shear-data}, with $\tau=10^{-3}$
and $N=64$. Rows correspond to $t=0,5,10$; columns to
$\nu=0,10^{-3},10^{-1}$. The \texttt{cmocean.balance} scale shows
$\omega/A$: blue is negative, near-white is zero, and red is positive.
Each panel states
$A=\max_{\mathcal G}|\omega|$ on the $512\times512$ plotting grid.
Fine features in the inviscid
$t=10$ panel remain sensitive to spatial resolution.}
\label{fig:expanded-double-shear}
\end{figure}

\subsubsection{Vortex waltz}
\label{sec:expanded-waltz}

For a vortex-waltz test motivated by~\cite{Hakim22}, we place
two Gaussian vortices at $a=(\pi/4,0)$ and $-a$.
On the periodic square $[-\pi,\pi]^2$, let
\begin{equation}\label{eq:expanded-waltz-data}
 \begin{split}
 G(x)&=\sum_{m\in\mathbb Z^2}\exp\bigl(-5|x+2\pi m|^2\bigr),\\
 \omega_0(x)&=G(x-a)+G(x+a)-\frac1{10\pi}.
 \end{split}
\end{equation}
The constant subtracts the spatial mean, as required for the curl
of a periodic velocity. With the Fourier convention
$f(x)=\sum_{k\in\mathbb Z^2}\widehat f(k)e^{ik\cdot x}$,
the coefficients are
\begin{equation}\label{eq:expanded-waltz-vorticity}
 \widehat\omega_0(k)=
 \begin{cases}
 \displaystyle\frac1{10\pi}e^{-|k|^2/20}
             \cos\!\left(\frac{\pi k_1}{4}\right),&k\ne0,\\[1ex]
 0,&k=0.
 \end{cases}
\end{equation}
The initial velocity is defined by
\begin{equation}\label{eq:expanded-waltz-velocity}
 \widehat u_0(k)=\frac{i(k_2,-k_1)}{|k|^2}\widehat\omega_0(k)
 \quad(k\ne0),\qquad \widehat u_0(0)=0.
\end{equation}
Thus $u_0=(\partial_2\psi_0,-\partial_1\psi_0)$ and
$\omega_0=\partial_1u_{0,2}-\partial_2u_{0,1}=-\Delta\psi_0$.

We initialize $u_N^0=\Pi_Nu_0$ from these coefficients and compute
\eqref{eq:oct-scheme} with $N=64$, $\tau=10^{-2}$, and
$\nu\in\{0,10^{-3},10^{-1}\}$. Products are evaluated on a
$200\times200$ grid; its side length exceeds $3N$, so the retained
quadratic convolution has no aliasing. Picard iteration uses the
PDE-residual tolerance $2\times10^{-12}$.
Figure~\ref{fig:expanded-waltz} shows the computed vorticity at
$t=0,10,25,50$, evaluated on a $512\times512$ plotting grid
$\mathcal G$. Each panel displays $\omega/A$, where
$A=\max_{\mathcal G}|\omega|$, and reports $A$ to three significant
digits. The common linear colour range $[-1,1]$ shows relative
vorticity profiles; physical amplitudes are given by $A$.

The two cores rotate and deform, while viscosity broadens the
vorticity distribution. For $\nu=10^{-1}$, the sampled maximum vorticity
at $t=50$ is $7.03\times10^{-4}$.

For $\nu=0$, we repeat the calculation with $\tau=5\times10^{-3}$
at $N=64$, and with $N=96$ at $\tau=10^{-2}$. At $t=50$, halving
the time step changes the velocity and vorticity by $0.530\%$ and
$5.33\%$ in relative $L^2$, respectively. Increasing $N$ changes
them by $0.270\%$ and $8.74\%$. The vorticity is
more sensitive to refinement than the velocity. In the plotted
inviscid solution, its sampled minimum falls from $-0.03183$ initially
to $-0.09907$ at $t=50$. Smooth Euler evolution preserves the
vorticity range, so this undershoot is numerical. The late-time fine
bands are not fully resolved.

\begin{figure}[p]
\centering
\includegraphics[width=\textwidth]
 {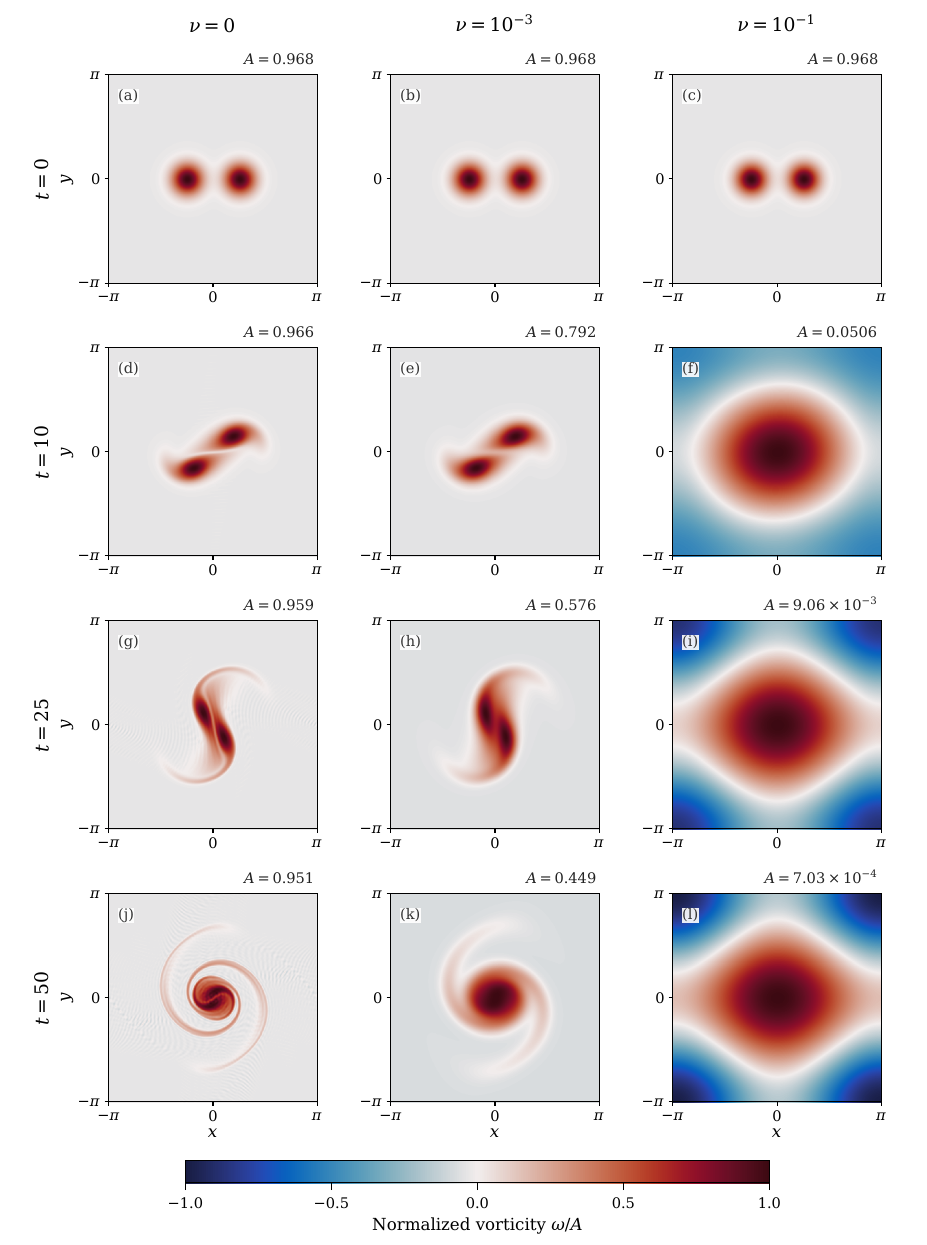}
\caption{Normalized vorticity for the Gaussian-pair data
\eqref{eq:expanded-waltz-data}, computed with $N=64$ and
$\tau=10^{-2}$. Rows correspond to $t=0,10,25,50$; columns to
$\nu=0,10^{-3},10^{-1}$. The \texttt{cmocean.balance} scale shows
$\omega/A$: blue is negative, near-white is zero, and red is positive.
Each panel states
$A=\max_{\mathcal G}|\omega|$ on the $512\times512$ plotting grid.
Fine vorticity bands in the late inviscid panels remain sensitive
to refinement and are not fully resolved.}
\label{fig:expanded-waltz}
\end{figure}

\subsection{Discussion}
\label{sec:expanded-discussion}

Both smooth tests show first-order convergence in time. The
manufactured solution contains several dyadic bands and provides
an exact reference. The unforced calculations compare viscous and
inviscid velocities from the same initial data. The changes caused
by refining the reference solution and the endpoint sampling grid
remain below the reported errors.

At fixed Fourier resolution, the projected rough datum is also
smooth. Its observed first-order time convergence is compatible
with a fractional estimate uniform over the full initial-data class.
For the infinite shear, the error sums include every frequency and
can be evaluated without a numerical reference solution. Modes near
the transition frequency produce the fractional scale, with its
logarithmic factor, in $B^1_{2,1}$. Removing the derivative weight
gives first-order convergence in $B^0_{\infty,1}$ for the same datum.

The nonlinear spatial Cauchy differences depend on both resolutions
in the comparison. For the transported shear, the Fourier-tail error
can be computed directly. Its faster decay reflects the absence of
mode interactions and does not determine the optimal spatial rate
for nonlinear flow. In both sets of convergence tests, the implicit
solver residuals are well below the displayed errors. Their effect
is bounded by $CN\delta$ as in
Proposition~\ref{oct-solver-prop-pde-residual}. The same residual
estimate applies to the manufactured problem because its prescribed
force cancels when the exact and computed discrete equations are subtracted.

The flow comparisons concern finite-time evolution. For long-time
approximation, Hill and S\"uli studied the Navier--Stokes global
attractor~\cite{HS00}, and Wang proved uniform-in-time bounds and
convergence of stationary statistics for a second-order method on
a periodic domain~\cite{Wang12}. Convergence of attractors or
stationary statistics is not covered by the local estimates proved here.

\FloatBarrier
\flushbottom
\section{Conclusion}

The error estimates are uniform in viscosity on a local interval
fixed by the initial data. The consistency argument uses
viscosity-weighted higher derivatives, while dyadic $L^2$ energy
estimates give the endpoint bound through
$B^1_{2,1}\hookrightarrow B^0_{\infty,1}$.
An $L^2$ tolerance $\delta$ for the PDE residual, with the discrete
equation divided by $\tau$, adds $CN\delta$ to the endpoint error.

The optimal spatial and fractional rates in $B^0_{\infty,1}$ remain
open for general nonlinear flow. The fixed-shear calculation distinguishes
the fractional scale in
$B^1_{2,1}$ from the first-order endpoint error of that example.
The limiting initial class $B^2_{2,1}$ requires a separate
argument. The bounded-domain estimates assume uniform regularity
of the continuous family and, for finite elements, the approximation
properties stated in Appendix~\ref{rev:sec-fem}.

\appendix

\section{No-slip boundaries}
\label{rev:sec-boundary}

For projection methods, the treatment of pressure at the boundary
affects the velocity error. E and Liu analyzed the resulting
numerical boundary layers for time-discrete projection
methods~\cite{EL95} and fully discrete staggered-grid
methods~\cite{EL02}. Hou and Wetton proved second-order convergence
for a projection scheme with boundaries~\cite{HW93}, and Shirokoff
and Rosales obtained high accuracy near irregular no-slip boundaries
through a pressure-Poisson formulation~\cite{SR11}. The step below
is posed directly in the divergence-free space.

The estimate below requires uniform regularity of the no-slip
solutions; boundary layers may invalidate these bounds as
$\nu\downarrow0$.

Let $\Omega\subset\mathbb R^2$ be a bounded Lipschitz domain and set
\[
 V:=\{z\in H^1_0(\Omega)^2:\nabla\cdot z=0\}.
\]
We write $(\cdot,\cdot)$ for the $L^2(\Omega)$ inner product and
\[
 b(a,w,z):=\int_\Omega (a\cdot\nabla w)\cdot z\,dx.
\]
If $a\in V$, then
\begin{equation}\label{rev:b-skew}
 b(a,z,z)=0\qquad(z\in V).
\end{equation}

For $0<\nu\le1$, the no-slip semi-implicit step is: given $u^n\in V$,
find $u^{n+1}\in V$ such that
\begin{equation}\label{rev:boundary-scheme}
 \left(\frac{u^{n+1}-u^n}{\tau},z\right)
 +b(u^n,u^{n+1},z)
 +\nu(\nabla u^{n+1},\nabla z)=0
 \qquad(z\in V).
\end{equation}
By \eqref{rev:b-skew}, the bilinear form is coercive for each fixed
$\nu>0$. The Lax--Milgram theorem gives a unique weak solution,
and the de Rham argument recovers the pressure up to an additive constant.

\begin{theorem}[Time error under uniform regularity]
\label{oct-bdry-thm-L2}
For every $0<\nu\le1$, let $U^\nu$ be a strong solution of the no-slip
Navier--Stokes equations in the weak form
\begin{equation}\label{rev:boundary-continuous}
 (\partial_tU^\nu,z)+b(U^\nu,U^\nu,z)
 +\nu(\nabla U^\nu,\nabla z)=0
 \qquad(z\in V)
\end{equation}
on $[0,T]$.  Assume that
\begin{equation}\label{rev:boundary-uniform-class}
 \sup_{0<\nu\le1}\left(
 \|\nabla U^\nu\|_{L^\infty((0,T)\times\Omega)}
 +\|\partial_tU^\nu\|_{L^2(0,T;L^2)}
 +\|\partial_{tt}U^\nu\|_{L^2(0,T;L^2)}
 \right)\le K.
\end{equation}
Let $u^n$ be defined by \eqref{rev:boundary-scheme}, with
$u^0=U^\nu(0)$ and $t_n=n\tau$, where $M\tau=T$.  There is
$\tau_0=\tau_0(K)>0$, independent of $\nu$, such that, if
$0<\tau\le\tau_0$, then
\begin{equation}\label{rev:boundary-L2-error}
 \max_{0\le n\le M}\|u^n-U^\nu(t_n)\|_{L^2(\Omega)}
 \le C(T,K)\tau .
\end{equation}
The constant in \eqref{rev:boundary-L2-error} is independent of
$\nu\in(0,1]$.
\end{theorem}

\begin{proof}
Since $\partial_tU^\nu\in H^1(0,T;L^2)$, we use the continuous $L^2$
representatives of $U^\nu$ and $\partial_tU^\nu$ at the time nodes.
The no-slip condition and uniform spatial Lipschitz bound imply, by
weak compactness along almost-everywhere time sequences, that
$U^\nu(t)\in V$ and $\|\nabla U^\nu(t)\|_\infty\le K$ for
every $t\in[0,T]$, with weak continuity in $H^1$.
Strong $L^2$ continuity and the resulting uniform $L^\infty$ bound
make $b(U^\nu,U^\nu,z)$ continuous in time for each $z\in V$:
one may use $b(U^\nu,U^\nu,z)=-(U^\nu\otimes U^\nu,\nabla z)$.
Each term in \eqref{rev:boundary-continuous} therefore extends
continuously to every node.  
Put $U^n=U^\nu(t_n)$ and define the consistency defect
\begin{equation}\label{rev:boundary-defect}
 \eta^{n+1}:=
 \frac{U^{n+1}-U^n}{\tau}-\partial_tU^\nu(t_{n+1})
 +(U^n-U^{n+1})\cdot\nabla U^{n+1}.
\end{equation}
Equation \eqref{rev:boundary-continuous}, evaluated at $t_{n+1}$,
shows that
\begin{align}\label{rev:boundary-exact-discrete}
 \left(\frac{U^{n+1}-U^n}{\tau},z\right)
 +b(U^n,U^{n+1},z)
 +\nu(\nabla U^{n+1},\nabla z)
 =(\eta^{n+1},z).
\end{align}
The two pieces of \eqref{rev:boundary-defect} satisfy
\begin{align}
 \left\|
 \frac{U^{n+1}-U^n}{\tau}-\partial_tU^\nu(t_{n+1})
 \right\|_2^2
 &\le \frac{\tau}{3}
 \int_{t_n}^{t_{n+1}}\|\partial_{tt}U^\nu(t)\|_2^2\,dt,
 \label{rev:defect-time}\\
 \|(U^n-U^{n+1})\cdot\nabla U^{n+1}\|_2^2
 &\le K^2\tau
 \int_{t_n}^{t_{n+1}}\|\partial_tU^\nu(t)\|_2^2\,dt.
 \label{rev:defect-lag}
\end{align}
Consequently,
\begin{equation}\label{rev:defect-sum}
 \tau\sum_{n=0}^{M-1}\|\eta^{n+1}\|_2^2
 \le C(K^2+K^4)\tau^2.
\end{equation}

Let $e^n=u^n-U^n$.  Subtracting
\eqref{rev:boundary-exact-discrete} from
\eqref{rev:boundary-scheme} gives
\begin{align*}
 \left(\frac{e^{n+1}-e^n}{\tau},z\right)
 +b(u^n,e^{n+1},z)+b(e^n,U^{n+1},z)
 +\nu(\nabla e^{n+1},\nabla z)
 =-(\eta^{n+1},z).
\end{align*}
Taking $z=e^{n+1}$ and using \eqref{rev:b-skew}, we obtain
\begin{align*}
 &\frac{1}{2\tau}\left(
 \|e^{n+1}\|_2^2-\|e^n\|_2^2
 +\|e^{n+1}-e^n\|_2^2\right)
 +\nu\|\nabla e^{n+1}\|_2^2\\
 &\qquad\le
 K\|e^n\|_2\|e^{n+1}\|_2
 +\|\eta^{n+1}\|_2\|e^{n+1}\|_2.
\end{align*}
Young's inequality implies
\begin{equation}\label{rev:boundary-recurrence}
 \bigl(1-(K+1)\tau\bigr)\|e^{n+1}\|_2^2
 \le (1+K\tau)\|e^n\|_2^2
 +\tau\|\eta^{n+1}\|_2^2.
\end{equation}
Choose $\tau_0$ so that $(K+1)\tau_0\le1/2$.
Discrete Gronwall, $e^0=0$, and \eqref{rev:defect-sum} yield
\eqref{rev:boundary-L2-error}.
\end{proof}

The estimate is conditional on \eqref{rev:boundary-uniform-class}
and does not establish a no-slip inviscid limit. Higher-order
estimates require additional boundary regularity and compatibility;
on Lipschitz domains, the domains of powers of the Stokes operator
must be treated separately~\cite{GS11}.

\section{Finite elements}
\label{rev:sec-fem}

We use the domain $\Omega$, space $V$, and trilinear form $b$ of
Appendix~\ref{rev:sec-boundary}.
Let $V_h\subset V$ be a conforming, pointwise divergence-free finite
element space.  Consider
\begin{equation}\label{rev:fem-scheme}
 \left(\frac{u_h^{n+1}-u_h^n}{\tau},v_h\right)
 +b(u_h^n,u_h^{n+1},v_h)
 +\nu(\nabla u_h^{n+1},\nabla v_h)=0
 \qquad(v_h\in V_h).
\end{equation}
The Ritz comparison below requires approximation and Lipschitz
bounds for the chosen space and mesh family.

\begin{proposition}[Finite element error]
\label{oct-bdry-cor-ritz-fem}
Let $0<h\le1$, let $r\ge2$, and let $V_h\subset V$ be conforming and
pointwise divergence free.  Define $R_h:V\to V_h$ by
\begin{equation}\label{oct-bdry-ritz-definition}
 (\nabla R_hz,\nabla v_h)=(\nabla z,\nabla v_h)
 \qquad(v_h\in V_h).
\end{equation}
Assume that, uniformly in $h$,
\begin{equation}\label{oct-bdry-ritz-approximation}
 \|R_hz-z\|_2+h\|\nabla(R_hz-z)\|_2
 \le C_Rh^r\|z\|_{H^r}
 \qquad(z\in V\cap H^r(\Omega)^2).
\end{equation}
Let $t_n=n\tau$ with $M\tau=T$, and let $U^\nu$ be strong no-slip
solutions of \eqref{rev:boundary-continuous} on $[0,T]$.
Suppose that, uniformly for $0<\nu\le1$,
\begin{align}
 &\|U^\nu\|_{L^\infty(0,T;H^r)}
 +\|\partial_tU^\nu\|_{L^2(0,T;H^r)}\notag\\
 &\qquad
 +\|\partial_{tt}U^\nu\|_{L^2(0,T;L^2)}
 +\|\nabla U^\nu\|_{L^\infty((0,T)\times\Omega)}
 \le K_U, \label{oct-bdry-fem-uniform-solution}\\
 &\sup_{0\le t\le T}\|\nabla R_hU^\nu(t)\|_{L^\infty}
 \le K_R . \label{oct-bdry-ritz-W1infty}
\end{align}
The constants $C_R,K_U,K_R$ are assumed independent of $h$ and $\nu$.
If $u_h^0=R_hU^\nu(0)$, then the finite-element scheme
\eqref{rev:fem-scheme} satisfies, for $\tau(K_R+1)\le1/2$,
\begin{equation}\label{oct-bdry-ritz-fem-error}
 \max_{0\le n\le M}\|u_h^n-U^\nu(t_n)\|_2
 \le C(T,C_R,K_U,K_R)(\tau+h^{r-1}),
\end{equation}
where the constant is independent of $\nu\in(0,1]$.
\end{proposition}

\begin{proof}
Write $U^n=U^\nu(t_n)$, set $W_h^n=R_hU^n$, and put
$\zeta^n=W_h^n-U^n$. The projection estimate gives
$\|\zeta^n\|_2\le C_Rh^rK_U$, and
\eqref{oct-bdry-ritz-W1infty} bounds $\|\nabla W_h^n\|_\infty$.

Let $\mathcal R_{n+1}(v_h)$ denote the left-hand side of
\eqref{rev:fem-scheme} evaluated at $W_h^n,W_h^{n+1}$. Subtracting
\eqref{rev:boundary-exact-discrete} and using the linearity of $R_h$
gives
\begin{align}
 \mathcal R_{n+1}(v_h)
 ={}&(\eta^{n+1},v_h)
 +\bigl((R_h-I)D_\tau U^{n+1},v_h\bigr)\notag\\
 &+b(W_h^n,\zeta^{n+1},v_h)
 +b(\zeta^n,U^{n+1},v_h),\label{oct-bdry-ritz-residual-identity}
\end{align}
where $D_\tau U^{n+1}=(U^{n+1}-U^n)/\tau$. Ritz orthogonality
\eqref{oct-bdry-ritz-definition} eliminates the viscous projection term:
\[
 \nu(\nabla\zeta^{n+1},\nabla v_h)=0
 \qquad(v_h\in V_h).
\]
Since $W_h^n$ has zero trace,
$\|W_h^n\|_{L^\infty}\le C_\Omega
\|\nabla W_h^n\|_{L^\infty}$.  Consequently
\begin{equation}\label{oct-bdry-ritz-residual-bound}
 |\mathcal R_{n+1}(v_h)|\le \rho_{n+1}\|v_h\|_2,
\end{equation}
\begin{align}
 \rho_{n+1}:={}&\|\eta^{n+1}\|_2
 +\|(R_h-I)D_\tau U^{n+1}\|_2\notag\\
 &+C_\Omega K_R\|\nabla\zeta^{n+1}\|_2
 +\|\nabla U^{n+1}\|_{L^\infty}\|\zeta^n\|_2 .
 \label{oct-bdry-ritz-rho}
\end{align}
The projection estimate and Cauchy--Schwarz in time show that
\begin{align}
 \|(R_h-I)D_\tau U^{n+1}\|_2^2
 &\le \frac{C_R^2h^{2r}}{\tau}
 \int_{t_n}^{t_{n+1}}\|\partial_tU^\nu(t)\|_{H^r}^2\,dt,
 \label{oct-bdry-ritz-time-residual}\\
 \|\nabla\zeta^{n+1}\|_2&\le C_Rh^{r-1}\|U^{n+1}\|_{H^r},
 \qquad
 \|\zeta^n\|_2\le C_Rh^r\|U^n\|_{H^r}.
 \label{oct-bdry-ritz-space-residual}
\end{align}
Combining these estimates with \eqref{rev:defect-sum} gives
\begin{equation}\label{oct-bdry-ritz-residual-sum}
 \tau\sum_{n=0}^{M-1}\rho_{n+1}^2
 \le C(T,C_R,K_U,K_R)
 \bigl(\tau^2+h^{2(r-1)}\bigr).
\end{equation}
Set $\theta^n=u_h^n-W_h^n$. Subtracting the residual equation from
\eqref{rev:fem-scheme} and testing with $\theta^{n+1}$, we use
$b(u_h^n,\theta^{n+1},\theta^{n+1})=0$ by \eqref{rev:b-skew}.
The remaining transport term is bounded by
$K_R\|\theta^n\|_2\|\theta^{n+1}\|_2$. As in
\eqref{rev:boundary-recurrence},
\[
 \bigl(1-(K_R+1)\tau\bigr)\|\theta^{n+1}\|_2^2
 \le (1+K_R\tau)\|\theta^n\|_2^2+\tau\rho_{n+1}^2.
\]
Since $\theta^0=0$, discrete Gronwall and
\eqref{oct-bdry-ritz-residual-sum} give
$\max_n\|\theta^n\|_2\le C(\tau+h^{r-1})$.
Adding $\|\zeta^n\|_2\le C_Rh^rK_U$ proves
\eqref{oct-bdry-ritz-fem-error}.
\end{proof}

\begin{remark}[Projected Lipschitz bound]
\label{oct-bdry-rem-ritz-W1infty}
A stable divergence-preserving interpolant gives a sufficient condition
for \eqref{oct-bdry-ritz-W1infty}. On a quasi-uniform two-dimensional
mesh, suppose there is a divergence-preserving $I_h:V\cap H^r\to V_h$
such that
\[
 \|\nabla I_hz\|_{L^\infty}\le C_I\|\nabla z\|_{L^\infty},
 \qquad
 \|\nabla(z-I_hz)\|_2\le C_Ih^{r-1}\|z\|_{H^r}.
\]
Ritz best approximation and the finite-element inverse inequality give
\begin{align*}
 \|\nabla R_hz\|_{L^\infty}
 &\le\|\nabla I_hz\|_{L^\infty}
 +C_{\rm inv}h^{-1}\|\nabla(R_hz-I_hz)\|_2\\
 &\le C C_I\|\nabla z\|_{L^\infty}
 +C(C_I,C_{\rm inv})h^{r-2}\|z\|_{H^r}.
\end{align*}
Since $r\ge2$ and $h\le1$, the bound is uniform. These interpolation
properties and the Ritz approximation estimate
\eqref{oct-bdry-ritz-approximation} must be verified for the chosen
element and mesh family; they do not follow from a Fortin operator
and discrete inf--sup stability alone.
\end{remark}

\section{\texorpdfstring{$H^1$}{H1} dissipation}
\label{rev:sec-H1-dissipation}

For the periodic Fourier scheme, let
$Q_N:=\Pi_N\mathbb P=\mathbb P\Pi_N$, where $\Pi_N$ is the square
cutoff $|k|_\infty\le N$. Let
$u_N^n,u_N^{n+1}\in X_N$ be divergence-free trigonometric polynomials
with this Fourier support satisfying
\begin{equation}\label{rev:periodic-scheme}
 \frac{u_N^{n+1}-u_N^n}{\tau}
 +Q_N(u_N^n\cdot\nabla u_N^{n+1})
 =\nu\Delta u_N^{n+1}
 \qquad\hbox{on }\mathbb T^2.
\end{equation}

\begin{proposition}[$H^1$ dissipation]
\label{oct-solver-prop-H1-dissipation}
Assume
\begin{equation}\label{rev:H1-K}
 \|\nabla u_N^{n+1}\|_2\le K.
\end{equation}
There is an absolute $c>0$ such that, whenever
\begin{equation}\label{rev:H1-step}
 \tau K^2\le c\nu,
\end{equation}
one has
\begin{equation}\label{rev:H1-dissipation}
 \|\nabla u_N^{n+1}\|_2^2
 +\frac12\|\nabla(u_N^{n+1}-u_N^n)\|_2^2
 +\nu\tau\|\Delta u_N^{n+1}\|_2^2
 \le\|\nabla u_N^n\|_2^2.
\end{equation}
A uniform $H^s$ bound with $s>2$ supplies the constant in
\eqref{rev:H1-K}; no $H^3$ estimate is needed.
\end{proposition}

\begin{proof}
Let $d^{n+1}=u_N^{n+1}-u_N^n$.  Testing
\eqref{rev:periodic-scheme} with $-\tau\Delta u_N^{n+1}$ gives
\begin{align}\label{rev:H1-energy-identity}
 &\frac12\left(\|\nabla u_N^{n+1}\|_2^2-\|\nabla u_N^n\|_2^2
 +\|\nabla d^{n+1}\|_2^2\right)
 +\nu\tau\|\Delta u_N^{n+1}\|_2^2\nonumber\\
 &\qquad=\tau (u_N^n\cdot\nabla u_N^{n+1},\Delta u_N^{n+1}).
\end{align}
In two dimensions,
\begin{equation}\label{rev:2d-enstrophy-cancel}
 (w\cdot\nabla w,\Delta w)=0
\end{equation}
for every smooth periodic divergence-free field $w$.  Since
$u_N^n=u_N^{n+1}-d^{n+1}$, the right-hand side of
\eqref{rev:H1-energy-identity} is
\[
 -\tau(d^{n+1}\cdot\nabla u_N^{n+1},\Delta u_N^{n+1}).
\]
Integration by parts, using $\nabla\cdot d^{n+1}=0$, and the
two-dimensional Gagliardo--Nirenberg inequality give
\begin{align*}
 |(d^{n+1}\cdot\nabla u_N^{n+1},\Delta u_N^{n+1})|
 &\le \|\nabla d^{n+1}\|_2\|\nabla u_N^{n+1}\|_4^2\\
 &\le C K\|\nabla d^{n+1}\|_2\|\Delta u_N^{n+1}\|_2.
\end{align*}
Consequently,
\[
 \tau |(d^{n+1}\cdot\nabla u_N^{n+1},\Delta u_N^{n+1})|
 \le\frac14\|\nabla d^{n+1}\|_2^2
 +C\tau^2K^2\|\Delta u_N^{n+1}\|_2^2.
\]
Under \eqref{rev:H1-step}, the last term is bounded by
$\frac12\nu\tau\|\Delta u_N^{n+1}\|_2^2$.  Substitution into
\eqref{rev:H1-energy-identity}, followed by multiplication by two,
proves \eqref{rev:H1-dissipation}.
\end{proof}

\begin{remark}
When $K$ is supplied by the local Sobolev bound, the condition is
$\tau\le c\nu/K^2$. This additional $H^1$ monotonicity thus
requires positive viscosity unless the solution is spatially constant.
The kinetic-energy identity of Proposition~\ref{prop:oct-energy}
holds without this restriction.
\end{remark}

\section*{Acknowledgements}
The authors thank Brian Wetton for suggesting the iterative Fourier
spectral scheme.

\section*{Funding and declarations}
\noindent\textbf{Funding.}
X. Cheng is partially supported by NSFC (Grants No. 12401270 and
No. 42450192), the Natural Science Foundation of Shanghai
(Grant No. 24ZR1404200), and the Shanghai Magnolia Talent Plan
Pujiang Project (Grant No. 24PJA007).
S. Wang is partially supported by the China National Postdoctoral
Program for Innovative Talents (Grant No. BX20250061) and the
China Postdoctoral Science Foundation (Grant No. 2024M760058).

\smallskip\noindent\textbf{Data availability.}
The code and numerical data supporting the reported experiments are
available upon request. No external datasets were used.

\smallskip\noindent\textbf{Use of artificial intelligence.}
OpenAI Codex assisted with drafting, restructuring and review of the
mathematical arguments, and with numerical-code development and checking.
The authors are responsible for the content of the manuscript.


\begin{thebibliography}{99}

\bibitem{BCD11}
H. Bahouri, J.-Y. Chemin, and R. Danchin,
\emph{Fourier Analysis and Nonlinear Partial Differential Equations},
Grundlehren der mathematischen Wissenschaften 343, Springer, 2011.

\bibitem{CLW24}
X. Cheng, Z. Luo, and S. Wang,
\emph{On semi-implicit schemes for the incompressible Euler equations via
the vanishing viscosity limit}, preprint, 2024.
\href{https://arxiv.org/abs/2406.12320v1}{arXiv:2406.12320v1}
(original version).

\bibitem{CLWB25}
X. Cheng, Z. Luo, and S. Wang,
\emph{A Besov-based integration-by-parts method for the incompressible
Navier--Stokes equations}, preprint, 2025.
\href{https://arxiv.org/abs/2509.23192v1}{arXiv:2509.23192v1}
(original version).

\bibitem{CLWCH26}
X. Cheng, Z. Luo, and S. Wang,
Stability and convergence of an iterative low-regularity method for the
Cahn--Hilliard equation,
\emph{J. Differential Equations} \textbf{478} (2026), 114622.
\href{https://doi.org/10.1016/j.jde.2026.114622}{doi:10.1016/j.jde.2026.114622}.

\bibitem{CLWSQG}
X. Cheng, Z. Luo, and S. Wang,
Optimal error estimates of a low-regularity integrator for the SQG
equations under endpoint regularity,
\emph{IMA J. Numer. Anal.}, to appear.

\bibitem{CLWVV}
X. Cheng, Z. Luo, and S. Wang,
The vanishing viscosity limit of the surface quasi-geostrophic equations
in borderline Besov spaces: Rates and same topology convergence,
\emph{Proc. Amer. Math. Soc.}, to appear.

\bibitem{CLYY24}
X. Cheng, Z. Luo, Z. Yang, and C. Yuan,
\emph{Global well-posedness and uniform-in-time vanishing damping limit
for the inviscid Oldroyd--B model}, preprint, 2024.
\href{https://arxiv.org/abs/2410.09340}{arXiv:2410.09340}.

\bibitem{Chorin68}
A. J. Chorin, Numerical solution of the Navier--Stokes equations,
\emph{Math. Comp.} \textbf{22} (1968), 745--762.

\bibitem{Chorin69}
A. J. Chorin, On the convergence of discrete approximations to the
Navier--Stokes equations, \emph{Math. Comp.} \textbf{23} (1969), 341--353.

\bibitem{EL95}
W. E and J.-G. Liu, Projection method I: convergence and numerical
boundary layers, \emph{SIAM J. Numer. Anal.} \textbf{32} (1995), 1017--1057.

\bibitem{EL02}
W. E and J.-G. Liu, Projection method III: spatial discretization on the
staggered grid, \emph{Math. Comp.} \textbf{71} (2002), 27--47.

\bibitem{GS11}
J.-L. Guermond and A. Salgado, A note on the Stokes operator and its powers,
\emph{J. Appl. Math. Comput.} \textbf{36} (2011), 241--250.

\bibitem{GZ03}
B.-Y. Guo and J. Zou, Fourier spectral projection method and nonlinear
convergence analysis for Navier--Stokes equations,
\emph{J. Math. Anal. Appl.} \textbf{282} (2003), 766--791.

\bibitem{GLY19}
Z. Guo, J. Li, and Z. Yin, Local well-posedness of the incompressible Euler
equations in $B^1_{\infty,1}$ and the inviscid limit of the Navier--Stokes
equations, \emph{J. Funct. Anal.} \textbf{276} (2019), 2821--2830.

\bibitem{Hakim22}
A. Hakim, \emph{JE13: 2D Incompressible Euler Solver},
Simulation Journal, online computational notebook,
\url{https://ammar-hakim.org/sj/je/je13/je13-incomp-euler-2d.html}
(accessed 8 October 2026).

\bibitem{He08}
Y. He, The Euler implicit/explicit scheme for the 2D time-dependent
Navier--Stokes equations with smooth or non-smooth initial data,
\emph{Math. Comp.} \textbf{77} (2008), 2097--2124.

\bibitem{HR82}
J. G. Heywood and R. Rannacher, Finite element approximation of the
nonstationary Navier--Stokes problem. I. Regularity of solutions and
second-order error estimates for spatial discretization,
\emph{SIAM J. Numer. Anal.} \textbf{19} (1982), 275--311.

\bibitem{HR90}
J. G. Heywood and R. Rannacher, Finite-element approximation of the
nonstationary Navier--Stokes problem. Part IV: Error analysis for
second-order time discretization,
\emph{SIAM J. Numer. Anal.} \textbf{27} (1990), 353--384.

\bibitem{HS00}
A. T. Hill and E. S\"uli, Approximation of the global attractor for the
incompressible Navier--Stokes equations,
\emph{IMA J. Numer. Anal.} \textbf{20} (2000), 633--667.

\bibitem{HK08}
T. Hmidi and S. Keraani, Incompressible viscous flows in borderline Besov
spaces, \emph{Arch. Ration. Mech. Anal.} \textbf{189} (2008), 283--300.

\bibitem{HW93}
T. Y. Hou and B. T. R. Wetton, Second-order convergence of a projection
scheme for the incompressible Navier--Stokes equations with boundaries,
\emph{SIAM J. Numer. Anal.} \textbf{30} (1993), 609--629.

\bibitem{L34}
J. Leray, Sur le mouvement d'un liquide visqueux emplissant l'espace,
\emph{Acta Math.} \textbf{63} (1934), 193--248.

\bibitem{LMS22}
B. Li, S. Ma, and K. Schratz, A semi-implicit exponential low-regularity
integrator for the Navier--Stokes equations,
\emph{SIAM J. Numer. Anal.} \textbf{60} (2022), 2273--2292.

\bibitem{LMU22}
B. Li, S. Ma, and Y. Ueda, Analysis of fully discrete finite element methods
for 2D Navier--Stokes equations with critical initial data,
\emph{ESAIM Math. Model. Numer. Anal.} \textbf{56} (2022), 2105--2139.

\bibitem{LRZZ25}
B. Li, Q. Rao, H. Zhang, and Z. Zhou,
\emph{Numerical analysis of 2D Navier--Stokes equations with nonsmooth
initial value in the critical space}, preprint, 2025.
arXiv:2510.00393v1.

\bibitem{Li13}
D. Li, On a frequency localized Bernstein inequality and some generalized
Poincar\'e-type inequalities,
\emph{Math. Res. Lett.} \textbf{20} (2013), no. 5, 933--945.

\bibitem{Li19}
D. Li, On Kato--Ponce and fractional Leibniz,
\emph{Rev. Mat. Iberoam.} \textbf{35} (2019), no. 1, 23--100.

\bibitem{Ma26}
S. Ma, \emph{A second-order energy-stable low-regularity integrator for
the Navier--Stokes equations}, preprint, 2026.
\href{https://arxiv.org/abs/2610.05045v1}{arXiv:2610.05045v1}.

\bibitem{MT98}
M. Marion and R. Temam, Navier--Stokes equations: theory and approximation,
in \emph{Handbook of Numerical Analysis}, vol. 6,
Elsevier--North Holland, Amsterdam, 1998, 503--689.

\bibitem{M07}
N. Masmoudi, Remarks about the inviscid limit of the Navier--Stokes system,
\emph{Comm. Math. Phys.} \textbf{270} (2007), 777--788.

\bibitem{RS21}
F. Rousset and K. Schratz, A general framework of low regularity
integrators, \emph{SIAM J. Numer. Anal.} \textbf{59} (2021), 1735--1768.

\bibitem{SR11}
D. Shirokoff and R. R. Rosales, An efficient method for the incompressible
Navier--Stokes equations on irregular domains with no-slip boundary
conditions, high order up to the boundary,
\emph{J. Comput. Phys.} \textbf{230} (2011), 8619--8646.

\bibitem{Wang12}
X. Wang, An efficient second order in time scheme for approximating long
time statistical properties of the two dimensional Navier--Stokes
equations, \emph{Numer. Math.} \textbf{121} (2012), 753--779.

\bibitem{Y63}
V. I. Yudovich, Non-stationary flows of an ideal incompressible fluid,
\emph{Zh. Vychisl. Mat. Mat. Fiz.} \textbf{3} (1963), 1032--1066;
English translation in \emph{U.S.S.R. Comput. Math. Math. Phys.}
\textbf{3} (1963), 1407--1456.

\end{thebibliography}
\end{document}